\documentclass[a4paper,11pt]{article}

\usepackage{amsmath, amscd}
\usepackage{fullpage}

\usepackage{amsthm}

\usepackage{slashbox}
 
\usepackage{tikz}
\usetikzlibrary{3d}
\usepackage{tikz-cd}
\usepackage{amssymb}
\usepackage{latexsym}
\usepackage{enumerate}

\newtheorem{Theorem}{Theorem}[section]

\newtheorem{Definition}[Theorem]{Definition}
\newtheorem{Proposition}[Theorem]{Proposition}
\newtheorem{Lemma}[Theorem]{Lemma}
\newtheorem{Corollary}[Theorem]{Corollary}
\newtheorem{Remark}[Theorem]{Remark}

\newtheorem{Example}[Theorem]{Example}

\begin{document}

\title{Dimension of representation and character varieties\\ for two and three-orbifolds}

\author{Joan Porti\footnote{Partially supported by the 
Micinn/FEDER  grant
PGC2018-095998-B-I00}} 

\date{\today}
\maketitle

\begin{abstract}
We consider varieties of representations and characters of 2 and 3-dimensional orbifolds in semisimple Lie groups, and we focus on computing their dimension. For hyperbolic 3-orbifolds,
we consider the component
of the variety of characters 
that contains the holonomy composed with the principal representation, 
we show that its dimension equals half the dimension of the
variety of characters of the boundary. We also show that this is a lower bound for the dimension of generic components.
We furthermore provide tools for computing dimensions of varieties of characters of 2-orbifolds, including the Hitchin component.
We apply this computation  to the dimension growth  of varieties of characters of some 3-dimensional manifolds in 
$\mathrm{SL}(n,\mathbb{C})$.
\end{abstract}

\section{Introduction}

In this paper we are interested in varieties of representations 
and characters
of the fundamental group of two and three-dimensional orbifolds $\mathcal O$ 
in complex semi-simple algebraic Lie groups $G$, denoted respectively by 
$$
R(\mathcal O,G)=\hom(\pi_1(\mathcal O), G) 
\qquad\textrm{ and  }\qquad
X(\mathcal O, G)=R(\mathcal O,G)/\!\!/ G.
$$
Since those are not irreducible, we use $\dim_\rho R(\mathcal O, G)$ and $\dim_{[\rho]} X(\mathcal O, G)$ to denote the dimension 
of a component  that contains $\rho$.

We discuss first a particular representation for oriented hyperbolic 3-orbifolds, with holonomy in 
$\mathrm{Isom}^+(\mathbb{H}^3)\cong \mathrm{PSL}(2,\mathbb C)$.
Let $G$ be an adjoint complex simple Lie group, eg.~$\mathrm{PSL}(n,\mathbb C)$, $\mathrm{PO}(n,\mathbb C)$ or
$\mathrm{PSp}(2m,\mathbb C)$.
Let 
$$\tau\colon \mathrm{PSL}(2,\mathbb C)\to G
$$ 
be the principal representation
(see \S\ref{subsection:principal} for the definition), eg.~$\operatorname{Sym}^{n-1} 
\colon \mathrm{PSL}(2,\mathbb C)\to \mathrm{PSL}(n,\mathbb C)$, 
whose image is contained in  $\mathrm{PSp}(2m,\mathbb C)$ for $n=2m$ or in
$\mathrm{PO}(2m+1,\mathbb C)$ for $n=2m+1$.
When $G_{\mathbb{R}}$ is a 
split real form of $G$, it restricts to 
$\tau\colon \mathrm{PGL}(2,\mathbb R)\to G_{\mathbb{R}}$, so that the Hitchin component of a surface or a 2-orbifold
is   the component of the variety of representations that contains the composition of $\tau$ with a Fuchsian representation.
Here we consider orientable hyperbolic 3-orbifolds and we compose the holonomy  in
$\mathrm{PSL}(2,\mathbb C)$ with the principal representation $\tau$. We prove
the following generalization of \cite{MFPABP, MFP12}:

\begin{Theorem}
\label{thm:main}
 Let $\mathcal O^3$ be a compact orientable 3-orbifold whose interior 
 is hyperbolic with holonomy $\mathrm{hol}\colon\pi_1(\mathcal O^3)\to \mathrm{PSL}(2,\mathbb C)$.
 Then the character of $\tau\circ \mathrm{hol}$ is a smooth point of $X(\mathcal O^3, G)$ of dimension
 $$
 \dim_{[\tau\circ \mathrm{hol}]} X(\mathcal O^3, G) =\frac{1}{2} \dim X(\partial \mathcal{O}^3, G).
 $$
 Furthermore, the map induced by restriction $
 X(\mathcal O^3, G)\to X(\partial \mathcal O^3, G)$ is locally injective at  the character of 
 $\tau\circ \mathrm{hol}$.
\end{Theorem}

The proof of Theorem~\ref{thm:main} is based on a  vanishing theorem 
of $L^2$-cohomology, as for manifolds in \cite{MFPABP, MFP12}. 
Notice that for a Euclidean component $\mathcal O^2$ 
of $\partial \mathcal{O}^3$, the restriction of $\tau\circ\mathrm{hol}$ can be a singular point of $X(\mathcal O^2,G)$, and to prove that 
$
 X(\mathcal O^3, G)\to X(\partial \mathcal O^3, G)$
is locally injective we use evaluation of characters. 

\medskip

A representation $\rho\colon \Gamma\to G$ is called \emph{good} if it is irreducible  
and the centralizer of $\rho(\Gamma)$ in $G$ is the \emph{center} $\mathcal{Z}(G)$ of $G$, in particular
$\mathfrak{g}^{\rho(\Gamma)}=0$.
We are interested in representations of Euclidean 2-orbifolds (eg.~$\chi(\mathcal O^2)=0$) that may be not
good. For instance,  nontrivial representations of a 2-torus are never good by Kolchin's theorem.
A representation of a Euclidean 2-orbifold $\rho\colon\pi_1(\mathcal O^2)\to G$ is called 
\emph{strongly regular} if
for a maximal torsion free subgroup
$\Gamma_0 < \pi_1(\mathcal O^2)$,  it holds:
\begin{itemize}
 \item $\dim \mathfrak{g}^{\rho(\Gamma_0)}=\operatorname{rank} G$, where 
 $\mathfrak{g}^{\rho(\Gamma_0)}=\{v\in\mathfrak{g}\mid \mathrm{Ad}_{\rho(\gamma)}(v)=v,\ \forall\gamma\in\Gamma_0 \}$, and 
\item the projection
of $\rho(\Gamma_0)$ is contained in a connected abelian subgroup of $G/\mathcal Z(G)$.
\end{itemize}
The following is a generalization of a theorem of Falbel and Guilloux for manifolds  in~\cite{FalbelGuilloux}:

\begin{Theorem}
\label{Thm:lowerbound}
 Let $\mathcal O^3$ be a compact orientable  good 3-orbifold, 
 with boundary $\partial \mathcal O^3= \partial_1\mathcal O^3\sqcup\dotsb\sqcup \partial_k\mathcal O^3 $,
 and $\chi(\partial_ i\mathcal O^3)\leq 0$ for $i=l,\dotsc,k$.
 Let $\rho\colon \pi_1(\mathcal O^3)\to G$ be a   good representation such that:
 \begin{itemize}
  \item If $\chi(\partial_i\mathcal O^3)<0$, then $\rho\vert_{\pi_1 (\partial_i\mathcal O^2)}$ is good. 
  \item If $\chi(\partial_i\mathcal O^3)=0$, then $\rho\vert_{\pi_1 (\partial_i\mathcal O^2)}$ is strongly regular. 
\end{itemize}
Then 
$$
\dim_{[\rho]} X(\mathcal O^3,G)\geq \frac{1}{2}\dim X(\partial \mathcal O^3,G).
$$
\end{Theorem}

If $\partial \mathcal{O}^3=\partial_1\mathcal O^3\sqcup\dotsb\sqcup \partial_k\mathcal O^3$  
denotes the decomposition in connected components, then
$
 X(\partial \mathcal{O}^3, G)= X(\partial_1\mathcal O^3, G)\times\dotsb \times X(\partial_k\mathcal O^3, G)$
and
$$ \dim X(\partial \mathcal O^3, G)= \dim X(\partial_1\mathcal O^3, G)+\dotsb +\dim X(\partial_k\mathcal O^3, G). 
$$
When the component $\partial_i\mathcal O^3$ is a (closed orientable) surface $S$:
$$
\dim X(S, G)=\begin{cases}
              -\chi(S)\dim G &\textrm{ if }\chi(S)<0 ,\\
              2\operatorname{rank} G &\textrm{ if }\chi(S)=0,
             \end{cases}
$$
which can be rewritten as
$$
\dim X(S, G)=
              -\chi(S)\dim G+
               2\dim\mathfrak{g}^{\rho(\pi_1S)},
$$
where $\mathfrak{g}^{H }=\{v\in\mathfrak{g}\mid \mathrm{Ad}_{\rho(\gamma)}(v)=v,\ \forall\gamma\in H\}$ 
denotes the centralizer in the Lie algebra of a subset $H\subset G$,
via the adjoint action.
When the $\partial_i\mathcal O^3$ is a (closed orientable) 2-orbifold $\mathcal O^2$ with branching locus $\Sigma\subset \mathcal O^2$, we prove in Theorems~\ref{thm:dim2hyp} and
\ref{thm:dim2Euc}:
\begin{equation}
 \label{eqn:intro}
\dim X(\mathcal O^2, G)=
               -\chi(\mathcal O^2\setminus \Sigma)\dim G +\sum_{x\in\Sigma} \dim \mathfrak{g}^{\rho(\operatorname{Stab}(x) )}   +
               2\dim\mathfrak{g}^{\rho(\pi_1(\mathcal O^2)) }  .
\end{equation}
Formula~\eqref{eqn:intro}
suggests that, instead of the usual Euler characteristic of an orbifold, we need another
quantity to compute dimensions of varieties of characters and representations. For that purpose, consider $K$
a $CW$-structure on a compact orbifold $\mathcal O^n$ and $\rho\colon \pi_1(\mathcal{O}^n)\to G$ a representation.
The \emph{twisted  Euler characteristic} is defined (Definition~\ref{Def:OEC}) as
$$
\widetilde\chi(\mathcal O^n, \mathrm{Ad}\rho) =\sum_{e\textrm{ cell of }K}(-1)^{\dim e} 
\dim \mathfrak{g}^{\mathrm{Ad}(\rho(\mathrm{Stab}(\tilde e))) } \in\mathbb{Z}.
$$
This is always an integer and should not be confused with the orbifold 
Euler characteristic, that is a rational number. For the trivial representation, this is just the Euler 
characteristic of the underlying space of the orbifold times $\dim G$. It is the alternated sum of dimensions
of cohomology groups of $\mathcal O^n$ twisted by $\mathrm{Ad}\rho$ (Proposition~\ref{Prop:Euler}):
$$
 \widetilde\chi(\mathcal{O}^n,\mathrm{Ad}\rho)=
 \sum_i (-1)^i\dim H^i(\mathcal O^n, \mathrm{Ad}\rho).
$$ 
In Section~\ref{Section:dim2} we prove
the following results, based on Goldman's work on surfaces \cite{Goldman} (see also~\cite{Weil}):

\begin{Theorem}
\label{thm:dim2hyp}
Let  $\mathcal O^2$ be a compact connected 2-orbifold, with $\chi(\mathcal O^2)\leq 0$. 
Let $\rho\in R(\mathcal{O}^2, G)$ be a good representation.
Then $[\rho]$ is a smooth point of $X(\mathcal O^2, G)$ of dimension
 $-\widetilde \chi(\mathcal O^2, \mathrm{Ad}\rho)$.
\end{Theorem}

\begin{Theorem}
\label{thm:dim2Euc}
Let $\mathcal O^2$ be a closed Euclidean 2-orbifold and 
$\rho\colon\pi_1(\mathcal{O}^2)\to G$ a strongly  regular representation.
 Then
 it belongs to a single component of $X(\mathcal O^2, G)$ that has dimension
  \begin{enumerate}[(a)]
   \item 
  $-\widetilde \chi(\mathcal O^2, \mathrm{Ad}\rho)+2\dim \mathfrak{g}^{\rho(\pi_1(\mathcal O^2))}$ if $\mathcal O^2$ is orientable,
  \item $-\tilde \chi(\mathcal O^2, \mathrm{Ad}\rho)+\dim \mathfrak{g}^{\rho(\pi_1( \widetilde {\mathcal O^2} ))}$ if $\mathcal O^2$ is  non-orientable,
  where $\widetilde {\mathcal O^2}\to \mathcal O^2$ denotes the orientation covering of $\mathcal O^2$.
\end{enumerate}
\end{Theorem}

Notice that Theorem~\ref{thm:dim2Euc} does not conclude smoothness. For instance, a parabolic representation of $\mathbb{ Z}^2$
in $\mathrm{SL}(2,\mathbb{C})$ is strongly regular but its character
(the trivial character) is 
a singular point of $X(\mathbb{Z}^2,\mathrm{SL}(2,\mathbb{C}))$
and $X(\mathbb{Z}^2,\mathrm{PSL}(2,\mathbb{C}))$.

Under the hypothesis of  Theorem~\ref{thm:dim2hyp} or Theorem~\ref{thm:dim2Euc}, 
when 
$\mathcal O^2$ is \emph{closed} and \emph{orientable} then $\widetilde \chi(\mathcal O^2, \mathrm{Ad}\rho)$ is even 
(Corollary~\ref{Cor:even}) 
and therefore $\dim_{[\rho]} X(\mathcal O^2, G)$ is even.

\smallskip

Most of the computations apply to  real Lie groups, in particular we spend a section discussing 
applications to the Hitchin component, that we denote by $\mathrm{Hit}(\mathcal O^2,G_{\mathbb{R}})$,
where $G_{\mathbb{R}}$ is  the split real form of the adjoint group $G$. As we consider \emph{non-orientable}
2-orbifolds, we deal with \emph{non-connected} real forms $G_{\mathbb{R}}$: we require 
that $G_{\mathbb{R}}$  contains the image by the principal representation of $\operatorname{PGL}(2,\mathbb{R})$, 
not only of its identity component $\operatorname{PSL}(2,\mathbb{R})$. For instance $G_{\mathbb R}=\operatorname{PSL}(n,\mathbb{R})$,
 $\mathrm{PSp}^\pm(2m)$ or $\mathrm{PO}(n, n+1)$.

The dimension of $\mathrm{Hit}(\mathcal O^2,G_{\mathbb{R}})$
has already been computed in \cite{ALS} by Alessandrini, Lee and Schaffhauser, but we 
give a different approach,  closer to the one of
Long and Thistlethwaite in \cite{LongThistlethwaite} for turnovers.
For instance, we show that
$$ \big\vert \dim (\mathrm{Hit}(\mathcal O^2,\mathrm{PGL}(n,\mathbb{R})))+
\chi(\mathcal O^2) \dim (\mathrm{PGL}(n,\mathbb{R}))
\big\vert \leq C(\mathcal O^2),
$$
where $ C(\mathcal O^2)$ is a constant that depends only on $\mathcal O^2$, Proposition~\ref{prop:growthHitchin}.
For  $\mathrm{PSp}^\pm(2m,\mathbb{R})$
the bound is not uniform but linear on $m$ and 
we need to introduce another term depending on the orbifold
$$
 \big\vert \dim (\mathrm{Hit}(\mathcal O^2, \mathrm{PSp}^\pm(2m,\mathbb{R})))+
\chi(\mathcal O^2) \dim ( \mathrm{PSp}^\pm(2m,\mathbb{R}))
+ c(\mathcal O^2) m 
\big\vert \leq C(\mathcal O^2)
$$
where $ C(\mathcal O^2)$ and $ c(\mathcal O^2)$ are constants that depend only on $\mathcal O^2$,
Proposition~\ref{Prop:AssPSp}.

We also show that $\mathrm{Hit}(\mathcal O^2,G_{\mathbb R})$  maximizes the dimension among all components
 of the variety of (good) representations of $\mathcal O^2$ in $G_{\mathbb R}$, for
 $G_{\mathbb R}=\mathrm{PGL}(n,\mathbb{R})$,  Proposition~\ref{Prop:MaxHSL}, or $\mathrm{PSp}^\pm(2m,\mathbb{R})$, Proposition~\ref{Prop:MaxHPSp}.

For $\mathcal O^2$ closed and orientable, the differential form of Atiyah-Bott-Goldman gives naturally a  symplectic structure on
$\mathrm{Hit}(\mathcal O^2,G_{\mathbb{R}})$. Furthermore (see Proposition~\ref{Prop:LagrangianHitchin}):

\begin{Proposition}
\label{Prop:LagrangianHitchinIntro}
When  $\mathcal O^2$ is closed and non-orientable, with orientation covering 
$\widetilde {\mathcal O^2}$, then $\mathrm{Hit}(\mathcal O^2,G_{\mathbb{R}})$
embeds in $\mathrm{Hit}(\widetilde{\mathcal O^2},G_{\mathbb{R}})$ as a Lagrangian submanifold. 
\end{Proposition}

 We apply these computations on 2-orbifolds to estimate the growth of $X(M^3,\mathrm{SL}(n,\mathbb C))$
with respect to $n$
for some 3-manifolds, in particular for some knot exteriors. For instance we show:
 
 \begin{Proposition}
 \label{Prop:fig8intro}
Let $\Gamma$ be the fundamental group of the exterior of the figure eight knot. Besides the canonical 
component (that has dimension $n-1$), for large $n$  
$X(\Gamma, \mathrm{SL}(n,\mathbb C))$ has at least $3$ components
that contain irreducible representations, whose dimension grow respectively as ${n^2}/{12}$, 
${n^2}/{20}$ and ${n^2}/{42}$.
\end{Proposition}


\medskip

The paper is organized as follows. 
In Section~\ref{Section:reps} we recall some basic notions and tools on varieties of representations, and 
in Section~\ref{Section:cohomology} we introduce tools of orbifold cohomology (some well known, some other new) 
that we use later to compute Zariski tangent spaces
to varieties of representations and characters. In Section~\ref{Section:dim2} we discuss varieties of representations of
two dimensional orbifolds, and we apply some of the results to discuss dimensions of Hitchin components of orbifolds in 
Section~\ref{Section:H2O}. 
In Section~\ref{Section:dim3} we prove the results on three-dimensional orbifolds. 
Finally, some explicit examples are computed in Section~\ref{Section:SL3}.


\section{Varieties of representations and characters}
\label{Section:reps}

Let $\Gamma$ be a finitely generated group, we are mainly interested in the fundamental group of a compact orbifold. 
Let  
$G$ be a  complex semi-simple algebraic Lie group.
The variety of representations
$$
R(\Gamma,G)=\hom(\Gamma, G) 
$$
is an affine algebraic set, perhaps not radical (i.e.~ possible with non-reduced function ring $\mathbb{C}[R(\Gamma,G)]^G$).
Its quotient by conjugation in the algebraic category is the variety of characters
$$
X(\Gamma, G)=R(\Gamma,G)/\!/ G.
$$
Namely, the algebra of invariant functions $\mathbb{C}[R(\Gamma,G)]^G$ is of finite type and 
$X(\Gamma, G)$ can be defined as the affine variety this function ring (or affine scheme when the ring is non-reduced): 
$\mathbb C[X(\Gamma, G)]\cong \mathbb{C}[R(\Gamma,G)]^G  $.
When $\Gamma=\pi_1(\mathcal O)$, they are denoted respectively by $R(\mathcal O, G)$ and 
$X(\mathcal O, G)$.

Following for instance \cite{JohnsonMillson}, we define:

\begin{Definition}
A representation   $\rho\colon\Gamma\to G$ is:
\begin{enumerate}[(a)]
 \item \emph{irreducible} if 
$\rho(\Gamma)$ is not contained in a proper parabolic subgroup of $G$, and 
\item \emph{good} if it is irreducible and the centralizer of the image equals the center of $G$, $\mathcal{Z}(G)$. 
\end{enumerate}
\end{Definition}

When $G= \mathrm{SL}(n,\mathbb C)$, an irreducible representation is also good.

Let $R^{\mathrm{good}}(\Gamma, G)$ denote the subset of good representations in $R(\Gamma, G)$. 
The set of conjugacy classes 
$$\mathcal R^{\mathrm{good}}(\Gamma, G)=R^{\mathrm{good}}(\Gamma, G)/G$$
is a Zariski open subset of $X(\Gamma, G)$, cf.~\cite{JohnsonMillson}.

We discuss also representations in real Lie groups. For $G_{\mathbb{R}}$, the 
variety of characters is more subtle, cf.~\cite{ChoiLeeMarquis}, but we just consider 
$\mathcal R^{\mathrm{good}}(\Gamma, G_{\mathbb{R}})= R^{\mathrm{good}}(\Gamma, G_{\mathbb{R}})/G_{\mathbb{R}} $.

For a compact $2$ orbifold $\mathcal O^2$, we may also consider the \emph{relative character variety},
defined as 
$$
\mathcal R^{\mathrm{good}} (\mathcal{O}^2,\partial\mathcal{O}^2,G)=
\{[\rho]\in \mathcal R^{\mathrm{good}}(\mathcal{O}^2,G)\mid
[\rho\vert_{\partial_1\mathcal O^2}],\dotsc,[\rho\vert_{\partial_k\mathcal O^2}] \textrm{ are constant}\},
$$
where $\partial\mathcal O^2=\partial_1\mathcal O^2\sqcup\dotsb\sqcup \partial_k\mathcal O^2$ is the decomposition in connected components and 
$[\rho\vert_{\partial_i \mathcal O^2}]$ denotes the conjugacy class of the restriction to the $i$-th boundary component.

\subsection{Zarisi tangent space and cohomology}
\label{subsection:ZariskiTS}

 For a representation $\rho\colon \Gamma\to G$, its composition with
the adjoint representation is denoted by
$$
\mathrm{Ad}\rho\colon\Gamma\to \operatorname{Aut}( \mathfrak{g}).
$$
The adjoint representation preserves the Killing form
$$
\mathcal B\colon \mathfrak{g}\times \mathfrak{g}\to\mathbb C,
$$
that is non-degenerate, as we assume $G$ semi-simple. 

A \emph{derivation} or \emph{crossed morphism} is 
a  $\mathbb C$-linear map $d\colon\Gamma\to\mathfrak{g}$ that satisfies
$$
d(\gamma_1\gamma_2)=d(\gamma_1)+ \mathrm{Ad}_{\rho(\gamma_1)}(d(\gamma_2)),\qquad\forall\gamma_1,\gamma_2\in\Gamma.
$$
The space of crossed morphisms is denoted by $Z^1(\Gamma, \mathrm{Ad}\rho)$.
A crossed morphism is called \emph{inner} if there exists $a\in\mathfrak{g}$ such that 
$$
d(\gamma)= (\mathrm{Ad}_{\rho(\gamma)}-1) (a), \qquad \forall \gamma\in \Gamma,
$$
 and the subspace of inner crossed
morphisms is denoted by $B^1(\Gamma, \mathrm{Ad}\rho)$. The quotient is the first cohomology group:
$$
H^1(\Gamma, {\mathfrak{g}}_{\mathrm{Ad}\rho})=H^1(\Gamma, \mathrm{Ad}\rho) \cong Z^1(\Gamma, \mathrm{Ad}\rho)/ B^1(\Gamma, \mathrm{Ad}\rho).
$$

\begin{Theorem}[Weil]
\label{Thm:Weil}
Let $ T^{Zar}_\rho R(\Gamma, G)$ be the Zariski tangent space 
of $ R(\Gamma, G)$ as a \emph{scheme}
at~$\rho$.
%
There is a natural isomorphism 
$$
Z^1(\Gamma, \mathrm{Ad}\rho)\cong T^{Zar}_\rho R(\Gamma, G)
$$ 
that maps
 $B^1(\Gamma, \mathrm{Ad}\rho)$ to the tangent space to the orbit by conjugation.
Furthermore, when $\rho$ is good, 
$$
H^1(\Gamma, \mathrm{Ad}\rho)\cong T^{Zar}_\rho X(\Gamma, G).
$$
 \end{Theorem}

 For a discussion of Theorem~\ref{Thm:Weil}, see for instance \cite{HP2020} and the references therein.
We do not need the precise definition of  scheme, just mention that 
the polynomial ideal that defines $R(\Gamma, G)$ or $X(\Gamma, G)$  may be non-reduced, and this
is taken into account in the Zariski tangent space. In particular, when the dimension of the Zariski tangent space
at some representation or character equals the dimension of the component, the variety is smooth (and
the scheme is reduced and smooth)
at this
representation or character.

There is a real version of this theorem:
when the image of a good representations is contained
in a real form $G_{\mathbb{R}}$, then
$$
H^1(\Gamma, \mathrm{Ad}\rho_{\mathbb{R}})\cong T^{Zar}_\rho \mathcal R^{\mathrm{good}}(\Gamma, G_{\mathbb{R}}).
$$

%

\begin{Proposition}[\cite{HP2020}]
\label{Prop:RelTang} 
For a compact 2-orbifold $\mathcal O^2$, if  $\rho\colon\mathcal \pi_1(\mathcal{O}^2)\to G$ is good, then
$$
T_\rho^{Zar} \mathcal R^{\mathrm{good}}(\mathcal{O}^2,\partial \mathcal{O}^2, G)\cong \ker \big(H^1(\pi_1(\mathcal O^2),\mathrm{Ad}\rho)\to 
\bigoplus_{i=1}^k H^1( \pi_1( \partial_i\mathcal{O}^2),\mathrm{Ad}\rho) \big).
$$ 
\end{Proposition}

\subsection{The principal representation}
\label{subsection:principal}

Let $G$ be a simple complex Lie group, with Lie algebra $\mathfrak{g}$.
An  element in $\mathfrak{g}$ is called \emph{regular} if its centralizer 
has minimal dimension, that is the rank of $\mathfrak{g}$.
Given a regular nilpotent element, Jacobson-Morozov theorem provides a representation of 
Lie algebras
$\tau\colon\mathfrak{sl}(2,\mathbb C)\to\mathfrak{g}$ whose image contains the given regular 
nilpotent element. 
Such a representation is unique up to conjugacy.
Assume that $G$ is an adjoint group, namely that has trivial center.
Then the induced representation of Lie groups
$$
\tau\colon \mathrm{PSL}(2,\mathbb C)\to G
$$
is called the \emph{principal representation}.

\begin{Lemma}
\label{Lemma:PrincipalGood}
The image of the principal representation 
$\tau(\mathrm{PSL}(2,\mathbb C))$  
is irreducible and has trivial center
in $G$. 
\end{Lemma}

This lemma is a particular case of Lemma~2.8 in \cite{ALS}, for the complexification of the Hitchin component.
Alternatively, Lemma~\ref{Lemma:PrincipalGood} can also be proved using only Lie algebras.

By construction, 
the principal representation of the Lie algebra maps nontrivial elements of $\mathfrak{sl}(2,\mathbb C)$
to regular elements of $\mathfrak{g}$, i.e.~with centralizer the rank of $\mathfrak{g}$. Via the exponential map 
we have:

\begin{Remark}
\label{Remark:PrincipalRegular}
For $g \in  \mathrm{PSL}(2,\mathbb C)$ not elliptic nor trivial, $\tau(g)$ 
is regular in $G$ (the centralizer of $\tau(g)$ in $G$ has dimension $\mathrm{rank}(G)$).
 \end{Remark}

For $G= \mathrm{PSL}(n,\mathbb C)$ the principal representation is $\operatorname{Sym}^{n-1}$. 
When defined from $\mathrm{SL}(2,\mathbb C)$ to $\mathrm{SL}(n,\mathbb C)$, $\operatorname{Sym}^{n-1}$
 preserves a bilinear 
form of $\mathbb{C}^{n-1}$, that is symmetric for $n$ odd and skew-symmetric for $n$ even, and its restriction yields
the principal representations in $\mathrm{PSp}(2m,\mathbb C)$ and $\mathrm{PO}(2m+1,\mathbb C)$.

The principal representation restricts
to $\tau\colon \mathrm{PGL}(2,\mathbb R)\to G_{\mathbb R}$, for $ G_{\mathbb R}$ the split real form of $G$, perhaps not connected. 
As we deal with non-orientable 2-orbifolds, we consider both components of $\mathrm{PGL}(2,\mathbb R)$, hence we need to consider
two components of $ G_{\mathbb R}$.
For instance $\mathrm{PGL}(n,\mathbb R)$ contains two components, according to the sign of the determinant.
We also use the notation
$$
\operatorname{Sp}^{\pm}(2m)=\{A\in \mathrm{SL}(2m,\mathbb R) \mid A^tJA=\pm J  \}
$$
for $J$ an antisymmetric, non-degenerate, bilinear form; the sign $+$ or $-$ corresponds to one or the other component of the group.
The Hitchin component is the connected component of  $\mathcal R^{\mathrm{good}}(\pi_1(\mathcal O^2), G_{\mathbb R})$ 
that contains the composition of $\tau$ with the holonomy representation of any Fuchsian structure on $\mathcal O^2$.

 For the principal representation    $\tau\colon\mathrm{PSL}(2,\mathbb C)\to G$ we have a decomposition 
\begin{equation}
 \label{eqn:exponents}
\mathrm{Ad}\circ\tau= \bigoplus_{i=1}^r \operatorname{Sym}^{2 d_i} , 
\end{equation}
where $r=\operatorname{rank} G$, cf.~\cite{TY}.
In particular, 
\begin{equation}
 \label{eqn:dimgexp}
\dim G=\sum_{i=1}^r (2 d_i+1).
\end{equation}

\begin{Definition}
 The $d_1,\ldots ,d_r\in\mathbb N$
are called 
the \emph{exponents} of $G$.
\end{Definition}


For instance, the exponents of $\mathrm{PSL}(n,\mathbb{C})$ are $1,\ldots,n-1$, because
$$
\mathrm{Ad}\circ\operatorname{Sym}^{n-1}= \bigoplus_{i=1}^{n-1} \operatorname{Sym}^{2 i} .
$$
 The exponents of simple Lie algebras are detailed in Tables~\ref{Table:exponentsclassical}
 and~\ref{Table:exponents}, cf.~\cite{ALS}.

 \begin{table}
 \begin{center}
 \begin{tabular}{c l l l}
  Lie algebra & Exponents & Rank & Dimension \\
  \hline
 $\mathfrak{sl}(n,\mathbb C)$ & $1,2,\dotsc,n-1$  & $n-1$ & $n^2-1$\\
 $\mathfrak{sp}(2m,\mathbb C)$ & $1, 3, 5, \dotsc, 2m-1$ & $ m$ & $2m^2+m$ \\
 $\mathfrak{so}(2m+1,\mathbb C)$ & $1, 3, 5, \dotsc, 2m-1$ & $ m$ & $2m^2+m$  \\
 $\mathfrak{so}(2m,\mathbb C)$ & $1, 3, 5, \dotsc, 2m-3, m-1$ & $ m$ & $2m^2-m$ \\
 \hline
 \end{tabular}
 \end{center}  
  \caption{Exponents of simple classical Lie algebras} \label{Table:exponentsclassical}
\end{table}

 \begin{table}
 \begin{center}
 \begin{tabular}{c l l l}
  Lie algebra & Exponents & Rank & Dimension \\
  \hline
 $\mathfrak g_2$ & 1, 5  & 2 & 14\\
 $\mathfrak f_4$ & 1, 5, 7, 11 & 4 & 52 \\
 $\mathfrak e_6$ & 1, 4, 5, 7, 8, 11 & 6 & 78\\
 $\mathfrak e_7$ & 1,  5, 7, 9, 11, 13, 17 & 7 & 133\\
 $\mathfrak e_8$ & 1,  7, 11, 13, 17, 19, 23, 29 & 8 & 248\\
 \hline
 \end{tabular}
 \end{center}  
  \caption{Exponents of exceptional Lie algebras} \label{Table:exponents}
\end{table}
 

\section{Orbifold cohomology}
\label{Section:cohomology}

Group cohomology is useful to study dimensions of varieties of representations, 
by Weil's theorem (Theorem~\ref{Thm:Weil}).
 In this section we discuss  a combinatorial approach to cohomology. 
We are interested in representations in semi-simple complex algebraic groups, but this section applies also to their real forms, by replacing complex dimension by real dimension.
We focus in homology and cohomology with coefficients twisted by the adjoint of a representation, but 
the constructions and results can be easily adapted to coefficients twisted by other representations. 

\medskip

We recall some basic definitions of orbifolds \cite{BMP, ThurstonNotes}. An orbifold $\mathcal O$ is called:
\begin{itemize}
 \item \emph{good} if it has an orbifold cover that is a manifold;
 \item \emph{very good} it has an orbifold cover of finite index that is a manifold;
 \item \emph{aspherical} if its universal covering is a contractible manifold;
 \item \emph{hyperbolic/Euclidean/spherical} if it is the quotient
of hyperbolic space/Euclidean space/the round sphere by a discrete group of isometries.
 \end{itemize}

In a cell decomposition $K$ of an orbifold, we require 
that  isotropy groups or stabilizers are unchanged along open cells. For a cell $e$ in $K$
this isotropy group is denoted by $\mathrm{Stab}(e)$. The \emph{orbifold Euler characteristic} for a compact orbifold is defined as
\begin{equation}
\label{eqn:Euler}
\chi(\mathcal O)=\sum_{e\textrm{ cell of }K} \frac{1}{\vert \mathrm{Stab}(e) \vert}
(-1)^{\dim e}\in\mathbb Q. 
\end{equation}
This is well behaved under coverings: if $\mathcal O'\to\mathcal O$ is an orbifold covering of finite index $k$, then 
$\chi (\mathcal O')= k \chi (\mathcal O)$. See \cite{ThurstonNotes, BMP} for instance.

\subsection{Combinatorial homology and cohomology for orbifolds}

Let $\mathcal O^n$ be a compact $n$-dimensional good orbifold, possibly not orientable. 
Fix a CW-complex structure $K$ on $\mathcal O^n$. 
This means that $K$ is a CW-complex with the same underlying space 
as $\mathcal O^n$, so that  the branching locus of $\mathcal O^n$ is
a subcomplex of $K$. In particular, $K$it lifts to a CW-complex structure $\widetilde K$ 
of the universal covering $\widetilde {\mathcal O}^n$ of $ {\mathcal O}^n$, so that each cell
of $K$ lifts to a disjoint union of homeomorphic cells of $\widetilde K$ (perhaps with nontrivial
stabilizer).

The complex of chains on the universal covering is the free 
${\mathbb{Z}}$-module on the cells of $\widetilde K$, equipped with the usual boundary operator, and it
is denoted by $C_*(\widetilde K,{\mathbb{Z}})$. It has a (non-free) action of $\Gamma=\pi_1(\mathcal O^n)$
induced by deck transformations. 
The  \emph{twisted} chain and cochain complexes are defined as:
\begin{align}
 C_*(K,\mathrm{Ad}\rho)&= \mathfrak g \otimes_{\Gamma} C(\widetilde K,{\mathbb{Z}}), \label{eqn:chains} \\
 C^*(K,\mathrm{Ad}\rho)&= \hom_{\Gamma} ( C(\widetilde K,{\mathbb{Z}}), \mathfrak g) \label{eqn:cochains}.
\end{align}
The group  $\Gamma$ acts on $\mathfrak g$ via $\mathrm{Ad}\rho$ on the left in \eqref{eqn:chains},  and 
for the tensor product in \eqref{eqn:cochains} $\Gamma$ acts on $\mathfrak g$ on the right using inverses. 
Those are complexes and cocomplexes of finite-dimensional vector spaces, and the corresponding homology and cohomology groups are denoted 
respectively by
$$
H_*(K,\mathrm{Ad}\rho)\qquad\textrm{ and }\qquad H^*(K,\mathrm{Ad}\rho).
$$

\subsection{Twisted orbifold Euler characteristic}

For a compact manifold $M$ and a representation $\rho\colon\pi_1(M)\to G$, 
as $\pi_1(M)$ acts freely on the universal covering $\widetilde M$, we
have
$$
\chi(M)\, \dim ( G  )   = \sum_i (-1)^i\dim H_i(M, \mathrm{Ad}\rho) =\sum_i (-1)^i\dim H^i(M, \mathrm{Ad}\rho).
$$
To have a similar formula for an orbifold $\mathcal O$ we need to take into account that 
$\Gamma=\pi_1(\mathcal O)$  acts  non-freely on $\widetilde {\mathcal O}$; this motivates the 
definition of twisted orbifold Euler characteristic below, Definition~\ref{Def:OEC}.

Let $\mathcal{O}$ be compact oriented orbifold, very good, with a CW-structure $K$.
 Given a subgroup $\Gamma_0\subset\Gamma=\pi_1(\mathcal O)$, the space of \emph{invariants} is
\begin{equation}
\mathfrak g^{\rho(\Gamma_0)}= \{v\in\mathfrak g\mid \mathrm{Ad}_{\rho(g)}(v)=v,\ \forall g\in \Gamma_0\},
\end{equation}
and the quotient of \emph{coinvariants},
\begin{equation}
\mathfrak g_{\rho(\Gamma_0)}= \mathfrak g /\langle \mathrm{Ad}_{\rho(g)}(v)-v \mid v\in\mathfrak g,\  g\in \Gamma_0\rangle.
\end{equation}
As the Killing form on $\mathfrak{g}$ is nondegenerate ($G$ is semisimple) and $\mathrm{Ad}$-invariant we have
\begin{equation}
\label{eqn:InvOrtCoInv}
\langle \mathrm{Ad}_{\rho(g)}(v)-v \mid v\in\mathfrak g,\  g\in \Gamma_0\rangle= ( \mathfrak g^{\rho(\Gamma_0)} )^\perp
.\end{equation}
Therefore 
\begin{equation}\label{eqn:diminvcoinv}
\dim (\mathfrak g^{\rho(\Gamma_0)}) =\dim (\mathfrak g_{\rho(\Gamma_0)}).
\end{equation}

\begin{Definition}
\label{Def:OEC}
The \emph{orbifold Euler characteristic of $\mathcal O$ twisted} by $\mathrm{Ad}\rho$ is
 $$
\widetilde\chi(\mathcal{O},\mathrm{Ad}\rho)=\sum_{e\textrm{ cell of }K}(-1)^{\dim e} 
\dim \mathfrak{g}^{\rho(\mathrm{Stab}(\tilde e)) } \in\mathbb{Z}.
$$
\end{Definition}

Here $ \tilde e$ denotes any lift of $e$ to the universal covering of $\mathcal O$ and
$\mathrm{Stab}(\tilde e)\subset  \Gamma$ its stabilizer, whose conjugacy class in $\Gamma$ is independent of the choice of the lift.

It should not be confused with the usual orbifold Euler characteristic $\chi(\mathcal{O})$,
recalled in~\eqref{eqn:Euler}, that is a rational number,
whilst $\widetilde\chi(\mathcal{O},\mathrm{Ad}\rho)$ is always an integer.
Notice also that $\widetilde \chi$ is not multiplicative under coverings either;  its intended to be a tool
to compute dimensions of cohomology groups:

\begin{Proposition}
\label{Prop:Euler}
 For $\mathcal O$ and $\rho$ as above:
 $$
 \widetilde\chi(\mathcal{O},\mathrm{Ad}\rho)=\sum_i (-1)^i\dim H_i(\mathcal O, \mathrm{Ad}\rho) =\sum_i (-1)^i\dim H^i(\mathcal O, \mathrm{Ad}\rho).
 $$
 \end{Proposition}

\begin{proof}
We compute  the dimension of  $C_*(K,  \mathrm{Ad}\rho)$ and $ C^*(K,\mathrm{Ad}  \rho)$ as $\mathbb{C}$-vector spaces.
We aim to show that, for $0\leq i\leq \dim\mathcal{O}$,
\begin{equation}
 \label{eqn:dimC}
\dim C_i(K,  \mathrm{Ad}\rho)= \dim C^i(K,\mathrm{Ad}  \rho)= 
\sum_{j=1}^{k_i} \dim  \mathfrak{g}^{ \rho(\mathrm{Stab}(\tilde e^i_j)) }
 \end{equation}
where $\{ e^i_1,\ldots, e^i_{k_i} \} $ are the $i$-cells of $K$ and 
 $\{ \tilde e^i_1,\ldots, \tilde e^i_{k_i} \} $ a choice of lifts in the universal covering.
Equation~\eqref{eqn:dimC} will imply that
 \begin{equation}
 \widetilde\chi(\mathcal{O},\mathrm{Ad}\rho)=\sum_i (-1)^i\dim C_i(K, \mathrm{Ad}\rho) =\sum_i (-1)^i\dim C^i(K, \mathrm{Ad}\rho)
  \end{equation}
and then the proposition will follow from standard arguments in homological algebra.

To prove~\eqref{eqn:dimC}, 
use the decomposition as ${\mathbb{Z}}[\Gamma]$-module of the chain complex on $\widetilde K$:
\begin{equation}
 \label{eqn:ChainDec}
C_i(\widetilde K, \mathbb Z)=\bigoplus_{j=1}^{k_i} {\mathbb{Z}}[\Gamma]\tilde e^i_j ,
\end{equation}
because the $\Gamma$-orbits of the lifts  $\{ \tilde e^i_1,\ldots, \tilde e^i_{k_i} \} $  give a partition of the $i$-cells of $\widetilde K$.
Then we apply the natural isomorphisms of $\mathbb C$-vector spaces,
for each cell $\tilde{e}$ of $\widetilde{K}$:
\begin{equation}
\label{eqn:ChainStab}
\begin{array}{rcl}
 \hom_\Gamma( {\mathbb{Z}}[\Gamma]\tilde e, \mathfrak g) & \cong &  
 \mathfrak g^{\rho(\mathrm{Stab}(\tilde e)} \\
 \theta & \mapsto & \theta(\tilde e)
\end{array}
\qquad
\begin{array}{rcl}
 \mathfrak g \otimes_\Gamma {\mathbb{Z}}[\Gamma]\tilde e & \cong & 
 \mathfrak g_{\rho(\mathrm{Stab}(\tilde e))} \\
 v\otimes \tilde e & \mapsto & v
\end{array}
\end{equation}
whose proof is straightforward and use Equality~\eqref{eqn:diminvcoinv}.
\end{proof}

For 2-dimensional orientable orbifolds, Proposition~\ref{Prop:Euler} is essentially a formula computed by Andr\'e Weil in \cite[Sections~6 and 7]{Weil}.

\subsection{Regular Coverings}
Let $\mathcal O_0\to \mathcal O$ be a finite regular covering. Namely, 
$\Gamma_0= \pi_1(\mathcal O_0)$ is a normal subgroup of
$\Gamma= \pi_1(\mathcal O)$ of finite index. Though $\mathcal O$ is very good, we do not require $ \mathcal O_0$ 
to be a manifold in this subsection (for instance, $\mathcal{O}_0$ can be the orientation covering). 

The \emph{Galois group}, or \emph{group of deck transformations},
of the covering is $\Gamma/\Gamma_0$. It acts naturally on the chain and cochain complexes, 
$C_*(K_0, \mathrm{Ad}\rho)$ and  $C^*(K_0, \mathrm{Ad}\rho)$. 
Namely, any $\gamma\in\Gamma$ maps a chain 
$m\otimes c\in C_*(K_0, \mathrm{Ad}\rho)$ 
to $m\cdot\gamma^{-1}\otimes \gamma c= \mathrm{Ad}\rho(\gamma)(m)\otimes \gamma c$, cf.~\cite{Memo}.
The action of $\Gamma_0$ is trivial by construction and therefore we have an action of $\Gamma/\Gamma_0$. 
Similarly, $ \gamma\in\Gamma$ maps 
a cochain 
$\theta\in C^*(K_0, \mathrm{Ad}\rho)$ to $\mathrm{Ad}(\rho(\gamma))\circ\theta\circ\gamma^{-1}$, 
which again induces an action of 
the Galois group  $\Gamma/\Gamma_0$.
The subspace of elements \emph{fixed} by this action in homology and cohomology are denoted respectively
by 
$$H_*(\mathcal O_0,\mathrm{Ad}\rho)^{ \Gamma/\Gamma_0 }
\qquad\textrm{ and }\qquad
H^*(\mathcal O_0,\mathrm{Ad}\rho)^{ \Gamma/\Gamma_0 }.
$$

\begin{Proposition}
\label{prop:regularcovering}
The map $\pi\colon\mathcal O_ 0\to\mathcal O$ induces an
epimorphism
$
\pi_*\colon H_*(\mathcal O_0,\mathrm{Ad}\rho)\twoheadrightarrow H_*(\mathcal O,\mathrm{Ad}\rho)
$
 that restricts to an isomorphism $$H_*(\mathcal O_0,\mathrm{Ad}\rho)^{ \Gamma/\Gamma_0 }
 \cong H_*(\mathcal O,\mathrm{Ad}\rho).$$
 Similarly, it induces a monomorphism
 $
\pi^*\colon H^*(\mathcal O,\mathrm{Ad}\rho) \hookrightarrow H^*(\mathcal O_0,\mathrm{Ad}\rho)
$
that yields an isomorphism 
$$ 
H^*(\mathcal O,\mathrm{Ad}\rho) \cong H^*(\mathcal O_0,\mathrm{Ad}\rho) ^{ \Gamma/\Gamma_0 }.
$$
\end{Proposition}

\begin{proof}
The argument is standard and we just outline it. 
 The proof is based in constructing a section to the maps induced by the projection $\pi$.
The section at the chain and cochain level consists in taking an element of the inverse image of the map 
induced by $\pi$ and averaging by the action of $\Gamma/ \Gamma_0$.  One can check
that the  section for chains is a well defined chain morphism, induces sections in homology and has the required property, as well as for
sections of cochains and cohomology.
\end{proof}

The following proposition summarizes the main properties we need about orbifold cohomology:
 
\begin{Proposition}
\label{Prop:duality}
 Let $\mathcal O$ be a compact very good orbifold and $\rho$ a representation of 
 $\Gamma =\pi_1(\mathcal O)$ in $G$. The following hold:
 \begin{enumerate}[(a)]
  \item $H^i(\mathcal O,\mathrm{Ad}\rho)$ and $H_i(\mathcal O,\mathrm{Ad}\rho)  $
  are dual, via the Kronecker pairing.
  \item If  $\mathcal O$ is \emph{orientable}, then the cup product induces a perfect pairing
   $$H^i(\mathcal O,\mathrm{Ad}\rho)\times H^{\dim\mathcal O-i}(\mathcal O,\partial \mathcal O,\mathrm{Ad}\rho) \to   \mathbb C.$$
  \item There is a natural isomorphism $H^1(\Gamma, \mathrm{Ad}\rho)\cong H^1(\mathcal O,\mathrm{Ad}\rho)$.
  \item If $\mathcal O$ is aspherical (its universal covering is a contractible manifold), then there is a natural isomorphism $H^*(\Gamma, \mathrm{Ad}\rho)\cong H^*(\mathcal O,\mathrm{Ad}\rho)$.
 \end{enumerate}
\end{Proposition}

\begin{proof}
 For assertion (a), we use the Kronecker pairing between chains and cochains via the Killing form~$\mathcal B$:
$$
\begin{array}{crl}
 C_i(K,\mathrm{Ad}\rho)\times C^i(K,\mathrm{Ad}\rho) & \to  & \mathbb{C}\\
 (m\otimes c, \theta) & \mapsto & \mathcal{B}(m, \theta(c))
\end{array}
$$
One checks that it is well defined and it is not degenerate at the level of chains,
using  \eqref{eqn:ChainDec}, \eqref{eqn:ChainStab}, and \eqref{eqn:InvOrtCoInv}.
Hence it
induces a perfect pairing  between homology and cohomology.

For (b), using Proposition~\ref{prop:regularcovering}, the strategy is to use the invariance by $\Gamma/\Gamma_0$ of the corresponding properties
for a finite regular covering $\mathcal O_0\to \mathcal O$ that is a manifold.
More precisely,  using that $\mathcal O_0$ is a manifold,  the cup product defines a perfect pairing
$$
H^i(\mathcal O_0, \mathrm{Ad}\rho)\times H^{\dim\mathcal O_0- i}(\mathcal O_0,
\partial\mathcal O_0,\mathrm{Ad}\rho)\to H^{ \dim\mathcal O_0}(\mathcal O_0,\partial\mathcal O_0,\mathbb{C})\cong
\mathbb{C}.
$$
This is compatible with the action of $\Gamma/\Gamma_0$, therefore the assertion follows from
Proposition~\ref{prop:regularcovering}. Notice that orientability of $\mathcal{O}$ is relevant for saying that the action of the Galois group 
$ \Gamma/\Gamma_0$ on  $H^{ \dim\mathcal O_0}(\mathcal O_0,\partial\mathcal O_0,\mathbb{C})\cong \mathbb{C}$ is trivial.

Next we prove (c).
Chose $X$ a $K(\Gamma,1)$, namely a  CW-complex with $\pi_1(X)\cong \Gamma$ and trivial higher homotopy groups. 
Thus  there is a natural isomorphism
$H^*(X,\mathrm{Ad}\rho)\cong H^*(\Gamma,\mathrm{Ad}\rho)$. Furthermore, 
the covering corresponding to $\Gamma_0 <\Gamma$,
$X_0\to X$,  is a $K(\Gamma_0,1)$; in particular 
$H^*(X_0,\mathrm{Ad}\rho)\cong H^*(\Gamma_0,\mathrm{Ad}\rho)$. As $\mathcal O_0$ is a manifold,
a   $K(\Gamma_0,1)$ 
can be constructed from the 2-skeleton of $\mathcal O_0$ by adding cells of dimension $\geq 3$,
therefore: 
$$
H^1(\mathcal O_0,\mathrm{Ad}\rho)\cong H^1(X_0,\mathrm{Ad}\rho).
$$
Thus, by Proposition~\ref{prop:regularcovering} we have natural isomorphisms
$$
H^1(\mathcal O, \mathrm{Ad}\rho )\cong H^1(\mathcal O_0,\mathrm{Ad}\rho)^{\Gamma/\Gamma_0}\cong
H^1(X_0,\mathrm{Ad}\rho)^{\Gamma/\Gamma_0}\cong H^1(X,\mathrm{Ad}\rho),
$$
and $H^1(X,\mathrm{Ad}\rho)\cong H^1(\Gamma,\mathrm{Ad}\rho)$ because $X$ is a $K(\Gamma,1)$.

Finally, we prove (d). In the proof of (c), as ${\mathcal O}_0$ is already a $K(\Gamma_0,1)$,
and therefore homotopy
equivalent to $X_0$, hence 
$$
H^*(\mathcal O_0,\mathrm{Ad}\rho)\cong H^*(X_0,\mathrm{Ad}\rho).
$$
and the conclusion holds for any degree in cohomology. 
%
\end{proof}

From the duality in Proposition~\ref{Prop:duality} (b) we get:

\begin{Corollary}
\label{Cor:even}
When $\mathcal O^n$ is closed and orientable, of dimension $n$,
for $n$ odd $\widetilde \chi(\mathcal O^n,\mathrm{Ad}\rho )=0$, and
for $n$ even $\widetilde \chi(\mathcal O^n,\mathrm{Ad}\rho )$ is also even.
\end{Corollary}

\subsection{Non-orientable orbifolds}
\label{Subsection:orientation}

For a connected non-orientable orbifold $\mathcal O^n$, there exists a unique orientable covering
$\widetilde{\mathcal O^n}\to \mathcal O^n$ of index 2, called the \emph{orientation covering}.
Instead of Poincar\' e duality we have:

\begin{Lemma}
\label{Lemma:orientation}
 Let $\mathcal O^n$ be a compact, non-orientable, very good orbifold 
 of dimension $n$ with orientation covering $\widetilde{\mathcal O^n}\to \mathcal O^n$. Then
$$
\dim H^i(\mathcal{O}^n,\mathrm{Ad}\rho)+\dim H^{n-i}(\mathcal{O}^n,\partial \mathcal{O}^n,\mathrm{Ad}\rho)=
\dim H^i(\widetilde{\mathcal{O}^n},\mathrm{Ad}\rho).
$$
\end{Lemma}

\begin{proof}
The nontrivial deck transformation  of the orientation covering is 
an involution denoted by $\sigma\colon\widetilde{\mathcal O^n}\to\widetilde{\mathcal O^n}$.
Let $\sigma^*$ be the induced morphism in the cohomology group.
By Proposition~\ref{prop:regularcovering}:
\begin{equation}
\label{eqn:cohominv}
H^i(\mathcal{O}^n,\mathrm{Ad}\rho)\cong H^i(\widetilde{\mathcal{O}^n},\mathrm{Ad}\rho)^{\sigma^*}
\quad\textrm{ and }\quad
H^i(\mathcal{O}^n,\partial\mathcal{O}^n ,\mathrm{Ad}\rho)\cong 
H^i(\widetilde{\mathcal{O}^n},\partial\widetilde{\mathcal{O}^n},\mathrm{Ad}\rho)^{\sigma^*}.
\end{equation}
The  
cup product induces a non-degenerate pairing, Proposition~\ref{Prop:duality}:
\begin{equation}
\label{eqn:pairingNO}
 H^i(\widetilde{\mathcal O^n},  \mathrm{Ad}\rho)   \times
 H^{n-i}(\widetilde{\mathcal O^n},  \partial \widetilde{\mathcal O^n}, \mathrm{Ad}\rho) \to 
 H^n(\widetilde {\mathcal O^n},\partial\widetilde{\mathcal{O }^n} ,\mathbb{C})\cong \mathbb{C}.
\end{equation}
Chose $\mathbb{C}$-basis for 
$ H^i(\widetilde{\mathcal O^n},  \mathrm{Ad}\rho)  $ and for
$ H^{n-i}(\widetilde{\mathcal O^n},  \partial \widetilde{\mathcal O^n}, \mathrm{Ad}\rho)$, and 
use them to get matrices; let 
\begin{itemize}
 \item  $J$
denote the matrix of the pairing \eqref{eqn:pairingNO},
\item $S^i$ denote the matrix of $\sigma^*$ on 
$  H^i(\widetilde{\mathcal O^n},  \mathrm{Ad}\rho)$, and 
\item
$T^j$ denote
the matrix of $\sigma^*$ on 
$  H^j(\widetilde{\mathcal O^n}, \partial \widetilde{\mathcal O^n},  \mathrm{Ad}\rho)$.
\end{itemize}
Since $\sigma$ reverses the orientation, $\sigma^*$ acts as minus the identity on 
$H^n(\widetilde {\mathcal O^n},\partial\widetilde{\mathcal{O }^n} ,\mathbb{C})$:
$$
(S^i)^tJ T^{n-i} = -J .
$$
As the pairing is non-degenerate, $J$ is invertible, and since $\sigma$ is an involution:
$$
(S^i)^t= - J T^{n-i} J^{-1} .
$$
Thus 
\begin{equation}
\label{eqn:kertransp}
\dim\ker( T^{n-i}-\operatorname{Id})=\dim\ker( (S^{i})^t+\operatorname{Id})=\dim\ker( S^{i}+\operatorname{Id}).
\end{equation}
Since $(S^i)^2=\operatorname{Id} $:
\begin{equation}
\label{eqn:kerinvol}
  H^i(\widetilde{\mathcal O^n},  \mathrm{Ad}\rho)=\ker( S^{i}+\operatorname{Id})\oplus\ker( S^{i}-\operatorname{Id}).
\end{equation}
Combining \eqref{eqn:kertransp} and \eqref{eqn:kerinvol}:
$$
\dim\ker( T^{n-i}-\operatorname{Id})+ \dim \ker( S^{i}-\operatorname{Id})=\dim  H^i(\widetilde{\mathcal O^n},  \mathrm{Ad}\rho).
$$
With \eqref{eqn:cohominv}
this concludes the proof.
\end{proof}

\begin{Corollary}
\label{Corollary:HalfDimGeneral}
 If $\mathcal O^n$ is non-orientable, very good, closed  and of even dimension $n$,
 then
 $$
\dim H^{\frac{n}{2}}(\mathcal{O}^n,\mathrm{Ad}\rho)=\frac{1}{2}\dim H^{\frac{n}{2}}(\widetilde{\mathcal{O}^n},\mathrm{Ad}\rho).
 $$
\end{Corollary}

\begin{Remark}
 \label{rem:halfdimcohom}
The pairing in the proof of Lemma~\ref{Lemma:orientation} induces a nondegenerate pairing between the kernels
$\ker\big( H^i(\widetilde{\mathcal{O}^n},\mathrm{Ad}\rho)\to H^i(\partial \widetilde{\mathcal{O}^n},\mathrm{Ad}\rho)\big)$
and
$\ker\big( H^{n-i}(\widetilde{\mathcal{O}^n},\mathrm{Ad}\rho)\to H^{n-i}(\partial \widetilde{\mathcal{O}^n},\mathrm{Ad}\rho)\big)$
\cite{LawtonP, HP2020}. 
Thus the same argument in the proof of Lemma~\ref{Lemma:orientation} yields
\begin{multline*}
\dim \ker\big( H^i(\mathcal{O}^n,\mathrm{Ad}\rho) \to H^i(\partial \mathcal{O}^n,\mathrm{Ad}\rho)\big)
\\
+
\dim \ker\big( H^{n-i}(\mathcal{O}^n,\mathrm{Ad}\rho) \to H^{n-i}(\partial \mathcal{O}^n,\mathrm{Ad}\rho)\big)
\\
=\dim \ker\big( H^i(\widetilde{\mathcal{O}^n},\mathrm{Ad}\rho) \to H^i(\partial \widetilde{\mathcal{O}^n},\mathrm{Ad}\rho)\big).
\end{multline*}
%
\end{Remark}

\section{Varieties of representations and characters of 2-orbifolds}
\label{Section:dim2}

We recall the possible stabilizers (or isotropy groups) of a point in a 2-orbifold, so that we fix   notation. Besides the trivial one, the possible 
stabilizers of a point $x\in\mathcal{O}^2$ are:
\begin{itemize}
 \item $\mathrm{Stab}(x)\cong C_k$ is a cyclic group of rotations of the plane $\mathbb R^2$ of order $k$. Then the singular point is isolated and its is  called a 
 \emph{cone point}. Those are the unique possible non-trivial stabilizers for an orientable 2-orbifold (Figure~\ref{Figure:ConeMirrorCorner} left).
 \item $\mathrm{Stab}(x)\cong C_2$ is a group of order two generated by a
 reflection of the plane or the half-plane, ie.~modeled in the interior or in the boundary. This is called a 
 \emph{mirror point}.
 The singular locus is locally an open segment, or a proper half segment in the boundary (Figure~\ref{Figure:ConeMirrorCorner}, both pictures in the center).
 \item $\mathrm{Stab}(x)\cong D_{2k}$ is a dihedral group of $2k$ elements, generated by two reflections at the plane of angle $\pi/k$. 
 The point is called a \emph{corner reflector} (Figure~\ref{Figure:ConeMirrorCorner} right). 
\end{itemize}

\begin{figure}[h]
\begin{center}
\begin{tikzpicture}[line join = round, line cap = round, scale=.7]
%
\begin{scope}[shift={(9,0)}]
 \draw[ white, fill=gray!30!white, opacity=.3] (0,0)--(3,0)--(1,2)  to[out=210, in=90] (0,0);
\draw (0,0)--(3,0)--(1,2); 
\begin{scope}[shift={(0.05,+.1)}]
 \draw (0,0)--(3,0)--(1,2); 
\end{scope}
\draw[fill=black ]  (3-.05,0+.05) circle(.1);
%
 \draw(3.3,.3) node{$x$};
\end{scope}
\begin{scope}
 \draw[ white, fill=gray!30!white, opacity=.3] (0,0)--(4,0)  to[out=90, in=0] (2,1.8)
  to[out=180, in=90] (0,0);
\draw (0,0)--(4,0); 
\begin{scope}[shift={(0,+.1)}]
 \draw (0,0)--(4,0); 
\end{scope}
\draw[fill=black ]  (2,0+.05) circle(.1);
 \draw(2.1,.4) node{$x$};
\end{scope}
\begin{scope}[shift={(5.5 ,0)}];
 \draw[ white, fill=gray!30!white, opacity=.3] (0,0)--(2,0) -- (2,1.8)
  to[out=180, in=90] (0,0);
 \draw (0,0)--(2,0);
\draw[thick] (2,0)--(2,1.8); 
\begin{scope}[shift={(0,+.1)}]
 \draw (0,0)--(2,0); 
\end{scope}
\draw[fill=black ]  (2,0+.05) circle(.1);
 \draw(1.7,.4) node{$x$};
\end{scope}
\begin{scope}[shift={(-3,0)}]
  \draw[ white, fill=gray!30!white, opacity=.3] (0,1) circle(1.5);
  \draw[fill=black ]  (0,1) circle(.1);
  \draw(0.1,1.4) node{$x$};
\end{scope}
\end{tikzpicture}
\end{center}
\caption{From left to right, a cone point, an interior mirror point, a mirror point in the boundary, and a corner reflector.
Mirror points are represented by a double line.}
\label{Figure:ConeMirrorCorner}
\end{figure}
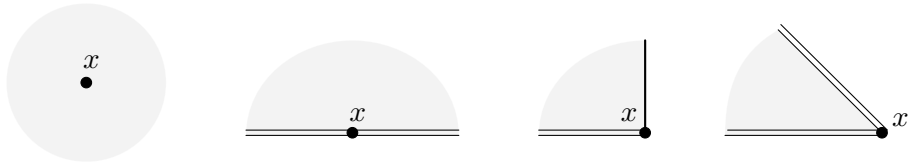

The boundary of a compact $2$-orbifold is a union of closed 1-orbifolds. There are two possibilities for a closed $1$-orbifold, up
to homeomorphism: 
\begin{itemize}
 \item A circle $S^1$ (Figure~\ref{Figure:boundaries} left, as it appears in the boundary). Its fundamental group is $\mathbb{Z}$.
 \item An interval with mirror end-points $[\![0,1]\!]$ (Figure~\ref{Figure:boundaries} right, as it appears in the boundary). 
 It is  the quotient of $S^1$ by an orientation reversing involution. 
 Its fundamental  group is dihedral infinite $D_{\infty}\cong C_2*C_2$, the free product of the stabilizers of the mirror points.
It is called \emph{full} 1-orbifold  in~\cite{ALS}.
\end{itemize}

\begin{figure}[h]
\begin{center}
\begin{tikzpicture}[line join = round, line cap = round, scale=.7]
 \shadedraw[ white, fill=gray!30!white, opacity=.3] (0,0.2) to[out=-15, in=180] (2,0)--(2,-2) to[out=180, in=15] (0,-2.2) to[out=150, in=210] cycle;
 \shadedraw[fill=gray!20!white] (2.5,-1) arc (0:360:0.5cm and 1cm);
\draw(0,0.2) to[out=-15, in=180] (2,0);
\draw(0,-2.2) to[out=15, in=180] (2,-2);
\draw[very thick] (2.5,-1) arc (0:360:0.5cm and 1cm);
\begin{scope}[shift={(5,0)}]
 \draw[ white, fill=gray!30!white, opacity=.3] (0,0.1) to[out=-15, in=180] (2,-.1)--(2,-1.9) to[out=180, in=15] (0,-2.1) to[out=110, in=250] cycle;
 \draw(0,0.2) to[out=-15, in=180] (2,0);
\draw(0,-2.2) to[out=15, in=180] (2,-2);
\begin{scope}[shift={(0,-.1)}]
 \draw(0,0.2) to[out=-15, in=180] (2,0);
\end{scope}
\begin{scope}[shift={(0,.1)}]
 \draw(0,-2.2) to[out=15, in=180] (2,-2);
\end{scope}
\draw[very thick] (2,0)--(2,-2);
\end{scope}
\end{tikzpicture}
\end{center}
\caption{A circle $S^1$ and an  interval with mirror end-points $[\![0,1]\!]$ as boundaries.}
\label{Figure:boundaries}
\end{figure}
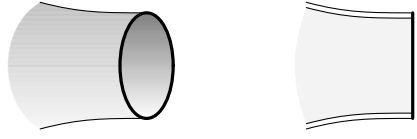

In dimension two, good and very good are equivalent. Furthermore, a closed 2-orbifold is good if and only if it is hyperbolic, Euclidean or spherical. 
When $\chi(\mathcal{O}^2)\leq 0$, 
$\mathcal{ O}^2$
is very good and either hyperbolic or Euclidean. In particular a closed
2-orbifold $\mathcal O^2$ is aspherical iff $\chi(\mathcal{O}^2)\leq 0$.

\subsection{Varieties of representations of $2$-orbifolds}
In this subsection we give a couple of results on the tangent space of the variety of
representations and characters of a 2-orbifold. Those are orbifold versions
of  theorems of Goldman for surfaces~\cite{Goldman}. In the orientable case, those results go back to Andr\' e Weil \cite[Sections~6 and 7]{Weil}.
We start with the following formulas:

\begin{Proposition}
\label{Prop:dimR} Let  $\mathcal O^2$ be a closed aspherical 2-orbifold. Set
$\Gamma=\pi_1(\mathcal O^2)$ and let $\rho\in R(\Gamma,G)$.
\begin{enumerate}[a)]
 \item If $\mathcal O^2$ is orientable, then
$$
\dim T^{Zar}_\rho R(\Gamma,G)=-\widetilde\chi(\mathcal O^2,\mathrm{Ad}\rho)+\dim G+  \dim \mathfrak{g}^{\rho(\Gamma)}.
$$
\item If $\mathcal O^2$ is non-orientable, with orientation covering  $\widetilde{\mathcal O^2}$, 
and $\widetilde\Gamma=\pi_1(\widetilde{\mathcal O^2})$, then 
$$
\dim T^{Zar}_\rho R(\Gamma,G)=-\widetilde\chi(\mathcal O^2,\mathrm{Ad}\rho)+\dim G+  \dim \mathfrak{g}^{\rho(\widetilde\Gamma)}-
\dim \mathfrak{g}^{\rho(\Gamma)}.
$$
\end{enumerate}

\end{Proposition}

\begin{proof}
We first claim 
\begin{equation}
 \label{eqn:TzarH2}
 \dim   T^{Zar}_\rho R(\Gamma,G)   = -\widetilde\chi(\mathcal O^2,\mathrm{Ad}\rho)+\dim G+\dim H^2( \mathcal O^2, \mathrm{Ad}\rho).
\end{equation}
To prove it, we use Weil's theorem, Theorem~\ref{Thm:Weil}: 
\begin{equation}
 \label{eqn:TzarWeil}
 T^{Zar}_\rho R(\Gamma,G)\cong Z^1(\Gamma,\mathrm{Ad}\rho)
 =\dim H^1( \Gamma, \mathrm{Ad}\rho)+ \dim B^1(\Gamma,\mathrm{Ad}\rho)
 .
 \end{equation}
From  
 the isomorphism $H^*( \mathcal O^2, \mathrm{Ad}\rho)\cong H^*( \Gamma, \mathrm{Ad}\rho)$
(as $\mathcal{O}^2$ is   aspherical  we  apply 
Proposition~\ref{Prop:duality} (d)), 
\eqref{eqn:TzarH2} follows from \eqref{eqn:TzarWeil} and the following equations:
\begin{align*}
\dim H^1( \mathcal O^2, \mathrm{Ad}\rho)&=-\widetilde\chi(\mathcal O^2,\mathrm{Ad}\rho)+\dim H^0( \mathcal O^2, \mathrm{Ad}\rho)
+\dim H^2( \mathcal O^2, \mathrm{Ad}\rho)
, \\
\dim B^1(\Gamma,\mathrm{Ad} \rho) &=\dim \mathfrak g-\dim \mathfrak{g}^{\rho(\Gamma)}
=\dim G-\dim  H^0( \Gamma, \mathrm{Ad}\rho).
 \end{align*}
Once \eqref{eqn:TzarH2} has been established, it remains to compute $\dim H^2( \mathcal O^2, \mathrm{Ad}\rho)$.
In the orientable case, by duality  (Proposition~\ref{Prop:duality}):
   $$\dim H^2( \mathcal O^2, \mathrm{Ad}\rho)= \dim H^0( \mathcal O^2, \mathrm{Ad}\rho) = \dim \mathfrak{g}^{\rho(\Gamma)}.
   $$
   In the non-orientable case, Lemma~\ref{Lemma:orientation} yields
   $$
   \dim H^2( \mathcal O^2, \mathrm{Ad}\rho)=\dim H^2( \widetilde{\mathcal O^2}, \mathrm{Ad}\rho)-\dim H^0( \mathcal O^2, \mathrm{Ad}\rho)
   =\dim \mathfrak{g}^{\rho(\widetilde \Gamma)}-\dim \mathfrak{g}^{\rho(\Gamma)},
   $$
and the proposition follows.   
\end{proof}

\begin{Proposition}
\label{Proposition:rhogood}
Let $\mathcal O^2$ be a compact aspherical 2-orbifold 
and  $\rho\colon \Gamma\to G$ a good representation.
When $\mathcal O^2$ is non-orientable, assume furthermore that the restriction of $\rho$ to 
the orientation covering is also good. 

Then the conjugacy class $[\rho]$ is a smooth point of
 $ X (\mathcal O^2, G)$  
 of  dimension $-\widetilde\chi(\mathcal O^2,\mathrm{Ad}\rho)$,
 with tangent space naturally isomorphic to $ H^1(\mathcal O^2,\mathrm{Ad}\rho)$.
\end{Proposition}

\begin{proof}
 As $\rho$ is good, its centralizer is finite, hence 
$$ 
  \dim H^0(\mathcal O^2,\mathrm{Ad}\rho)=
  \dim H^0(\Gamma,\mathrm{Ad}\rho)=\dim \mathfrak{g}^{\rho(\Gamma)}=0.
$$
When $\mathcal O^2$ is closed then 
$H^2(\mathcal O^2,\mathrm{Ad}\rho)=0$ (by duality in the orientable case, Proposition~\ref{Prop:duality},
or by Lemma~\ref{Lemma:orientation} in the non-orientable case). 
When $\mathcal O^2$ is not closed, then it has virtually the homotopy type of a 1-complex and also
 $H^2(\mathcal O^2,\mathrm{Ad}\rho)=0$. 

Moreover, as    $\mathcal O^2$ is aspherical, 
by  Proposition~\ref{Prop:duality} 
$H^2(\Gamma,\mathrm{Ad}\rho)=0$. By Goldman's obstruction theory~\cite{Goldman}, this implies that 
$[\rho]$ is a smooth point of local  dimension
$\dim  H^1(\Gamma,\mathrm{Ad}\rho)= \dim  H^1(\mathcal O^2,\mathrm{Ad}\rho)=
-\widetilde\chi(\mathcal O^2,\mathrm{Ad}\rho)$.
\end{proof}

From Proposition~\ref{Proposition:rhogood} and Corollary~\ref{Corollary:HalfDimGeneral}:

\begin{Corollary}
\label{Corollary:double}
 Let $\mathcal O^2$ be a \emph{closed} non-orientable 2-orbifold with 
$\chi(\mathcal O^2)\leq 0$ and  $\rho\colon \Gamma\to G$. Assume that 
the restriction of $\rho$ to the orientation covering $\widetilde{\mathcal O^2}$
is good. Then
$$
 \dim_{[\rho]}X(\mathcal O^2, G)=\frac{1}{2} \dim_{[\rho]}X(\widetilde{\mathcal O^2}, G).
$$ 
\end{Corollary}

\begin{Proposition}
\label{Prop:Symplectic}
 Let $\mathcal O^2$ be a compact aspherical 2-orbifold
 and  $\rho\colon \Gamma\to G$ a good representation that is $\partial$-regular. 
Assume that $\mathcal O^2$ is orientable. 
Then the cup product defines a  $\mathbb C$-valued symplectic structure on 
$\mathcal R^{\textrm{good}}(\mathcal O^2,\partial \mathcal O^2, G)$. Furthermore,
it is real valued for $G_{\mathbb{R}}$, a real form of~$G$. 
\end{Proposition}

\begin{proof}
From the construction, non-degeneracy and skew symmetry is clear. The non-trivial issue
is to check that this differential form is closed
in $R^{\textrm{good}}(\mathcal O^2,\partial \mathcal O^2, G)$.
For 2-manifolds with boundary, being closed is due to \cite{GHJW}, as explained in \cite{HP2020}.
For an orientable orbifold with boundary, we reduce to the manifold case. 
 Let $\Sigma_{\mathcal O^2}$ denote the  branching locus of $\mathcal O^2$, it is a finite union of
 cone points. The restriction map
 $$
 \mathcal R^{\textrm{good}}(\mathcal O^2,\partial \mathcal O^2, G) \to
  \mathcal R^{\textrm{good}}(\mathcal O^2\setminus\mathcal N( \Sigma_{\mathcal O^2}),
  \partial (\mathcal O^2\setminus\mathcal N( \Sigma_{\mathcal O^2})) , G)
 $$
is a   (local) isomorphism,
  as the conjugacy classes of finite order elements cannot be deformed. 
  To check the isomorphism at the level of tangent spaces, notice that
the map induced by inclusion  
$$
H^1(\mathcal O^2,\mathrm{Ad}\rho)
\to
H^1(\mathcal O^2\setminus\mathcal N( \Sigma_{\mathcal O^2}),\mathrm{Ad}\rho)
$$
restricts to an isomorphism  of kernels:
\begin{multline*}
\ker
\left(
H^1(\mathcal O^2,\mathrm{Ad}\rho)
\to H^1(\partial \mathcal O^2,\mathrm{Ad}\rho)
\right)
\\
\cong  \ker
\left(
H^1(\mathcal O^2\setminus\mathcal N( \Sigma_{\mathcal O^2}),\mathrm{Ad}\rho)
\to
H^1(\partial( \mathcal O^2\setminus\mathcal N( \Sigma_{\mathcal O^2} ),\mathrm{Ad}\rho)
\right) ,
\end{multline*}
that can be proved
for instance by using Mayer-Vietoris exact sequence. 
Furthermore, by naturality, it can be shown that the isomorphism
maps the cup product to the cup product.
\end{proof}

\begin{Proposition}
\label{Prop:Lagrangian}
 Let $\mathcal O^2$ be a \emph{closed} non-orientable 2-orbifold with 
$\chi(\mathcal O^2)\leq 0$. Let  $\rho\colon \Gamma\to G$ be a representation 
whose  restriction to the orientation covering $\widetilde{\mathcal O^2}$
is good  and $\partial$-regular.
Then the restriction map induces an immersion  around $[\rho]$ of
$\mathcal R^{\textrm{good}}(\mathcal O^2,\partial \mathcal O^2, G)$
in $\mathcal R^{\textrm{good}}(\widetilde{\mathcal O^2},\partial \widetilde{\mathcal O^2}, G)$
as a \emph{Lagrangian} submanifold. 
\end{Proposition}

\begin{proof}
Let $\sigma\colon \widetilde{\mathcal O^2}\to \widetilde{\mathcal O^2}$ be the orienting reversing involution so that 
$\widetilde{\mathcal O^2}/\sigma= \mathcal O^2$. 
By  Propositions~\ref{prop:regularcovering} and~\ref{Proposition:rhogood}
$$
T^{Zar}_{[\rho]} \mathcal{R}^{\textrm{good}}(\mathcal O^2, G)\cong 
H^1(\mathcal O^2,\mathrm{Ad}\rho)
\cong H^1(\widetilde{\mathcal O^2},\mathrm{Ad}\rho)^{\sigma^*}
\cong 
\big( T^{Zar}_{[\rho]} \mathcal{R}^{\textrm{good}}(\widetilde{\mathcal O^2}, G)\big)^{\sigma^*}.
$$
As $\sigma^*$ acts as minus the identity on $H^2(\widetilde{\mathcal O^2},\mathbb{C})$ 
and by construction of the non-degenerate bilinear form
$$
H^1(\widetilde{\mathcal O^2},\mathrm{Ad}\rho)\times H^1(\widetilde{\mathcal O^2},\mathrm{Ad}\rho)
\to  H^2(\widetilde{\mathcal O^2},\mathbb{C})
$$
(see the proof of Lemma~\ref{Lemma:orientation}), 
the bilinear form must be trivial on the space fixed by $\sigma^*$,  $H^1(\widetilde{\mathcal O^2},\mathrm{Ad}\rho)^{\sigma^*}$.
Namely, $T^{Zar}_{[\rho]} \mathcal{R}(\mathcal O^2, G)$ is an isotropic subspace. By Corollary~\ref{Corollary:double},
it has the dimension to be Lagrangian.
\end{proof}

\subsection{Orbifolds with boundary}

For an orientable 2-orbifold $\mathcal O^2$, a representation $\rho$ is called \emph{$\partial$-regular}
if for each $\gamma$ generator of a peripheral subgroup,
$\rho(\gamma)$ is a regular element (i.e. $\dim \mathfrak{g}^{\rho(\gamma)}=\operatorname{rank} G$).
For a non-orientable 2-orbifold, we will consider $\partial$-regularity on the orientation covering.

We start with a result on the dimension of the relative variety of (conjugacy classes of) representations $\mathcal R^{\mathrm{good}} (\mathcal O^2, \partial \mathcal O^2, G)$:


\begin{Proposition}
\label{Prop:relsmooth}
 Let $\mathcal O^2$ be a compact 2-orbifold with boundary and 
 $\chi (\mathcal O^2)\leq 0$.
 Assume  $\rho\colon \Gamma\to G$ is good and $\partial$-regular. 
When $\mathcal O^2$ is non-orientable, assume that the restriction to the orientation covering 
is good and $\partial$-regular.
 Then $[\rho]$ is a smooth point of
$\mathcal R^{\mathrm{good}} (\mathcal O^2, \partial \mathcal O^2, G)$
of dimension 
$$-\widetilde\chi(\mathcal O^2,\mathrm{Ad}\rho)-(\mathsf{c}+\frac{\mathsf{b}}{2})\ \mathrm{rank} (G)+
\frac12 \widetilde\chi(\partial \mathcal O^2,\mathrm{Ad}\rho) ,
$$
where 
$\mathsf{c}$ is the number of components of $\partial \mathcal O^2$ 
that are circles and $\mathsf{b}$ the number of components of $\partial \mathcal O^2$ 
that are intervals with mirror points $[\![0,1]\!]$.
\end{Proposition}

In this proposition, the components of the boundary that are circles do not contribute to $\widetilde \chi(\partial \mathcal O^2,\mathrm{Ad}\rho)$.

\begin{proof}
The same proof as in \cite[Proposition~2.10]{HP2020} applies here, in particular we have smoothness and the following equalities 
\begin{align*}
\dim (\mathcal R^{\mathrm{good}} (\mathcal O^2, \partial \mathcal O^2, G) ) &=
\dim \ker \big(H^1(\mathcal O^2,\mathrm{Ad}\rho)\to H^1(\partial \mathcal O^2,\mathrm{Ad}\rho)\big) \\ &=
\dim ( H^1(\mathcal O^2,\mathrm{Ad}\rho) )- \dim ( H^1(\partial \mathcal O^2,\mathrm{Ad}\rho) ). 
\end{align*}
As $H^0(\mathcal O^2,\mathrm{Ad}\rho)=H^2(\mathcal O^2,\mathrm{Ad}\rho)=0$,
$$\dim H^1(\mathcal O^2,\mathrm{Ad}\rho)=-\widetilde\chi(\mathcal O^2,\mathrm{Ad}\rho).$$
We count the contribution of each component $\partial_i\mathcal O^2 $ of $\partial \mathcal O^2$ 
to   $\dim H^1(\partial \mathcal O^2,\mathrm{Ad}\rho)$.
\begin{itemize}
 \item 
When $\partial_i\mathcal O^2 \cong S^1$, then, by duality and $\partial$-regularity:
\begin{equation}
 \label{eqn:reg}
\dim H^1(\partial_i\mathcal O^2,\mathrm{Ad}\rho)=\dim H^0(\partial_i\mathcal O^2,\mathrm{Ad}\rho)=\dim \mathfrak{g}^{\rho(\partial_i\mathcal O^2)}
=\operatorname{rank} G.
 \end{equation}
\item When $\partial_i\mathcal O^2 \cong [\![0,1]\!]$, then 
\begin{align*}
\widetilde \chi(\partial_i\mathcal O^2,\mathrm{Ad}\rho )& 
= \dim H^0(\partial_i\mathcal O^2,\mathrm{Ad}\rho) - \dim H^1(\partial_i\mathcal O^2,\mathrm{Ad}\rho) \\
 \operatorname{rank} G & = \dim H^0(\partial_i\mathcal O^2,\mathrm{Ad}\rho) + \dim H^1(\partial_i\mathcal O^2,\mathrm{Ad}\rho), 
\end{align*}
where the last line follows from Lemma~\ref{Lemma:orientation} and the assumption that $\rho$ is $\partial$-regular on the orientation covering.
From this 
we deduce
\begin{equation}
\label{eqn:H1Interval}
\dim H^1(\partial_i\mathcal O^2,\mathrm{Ad}\rho)=\frac{1}{2} (  \operatorname{rank} G  - \widetilde \chi(\partial_i\mathcal O^2,\mathrm{Ad}\rho ) ).
\end{equation}
\end{itemize}
From \eqref{eqn:reg} and \eqref{eqn:H1Interval}:
\begin{equation}
\label{eqn:dimH1bc}
\dim H^1(\partial \mathcal O^2,\mathrm{Ad}\rho)= (\mathsf{c}+\frac{\mathsf{b}}{2})\ \mathrm{rank} G-
\frac12 \widetilde\chi(\partial \mathcal O^2,\mathrm{Ad}\rho) ,
\end{equation}
which concludes the proof of the proposition.
\end{proof}



\begin{Proposition}
\label{Prop:boundaryNO}
Let $\mathcal O^2$ be a non-orientable compact 2-orbifold with orientation covering $\widetilde{\mathcal O^2}$. 
Let $\rho\colon \pi_1(\mathcal O^2)\to G$ be a representation whose restriction to $\widetilde{\mathcal O^2}$ 
is good and $\partial$-regular. Then 
$$
 \dim_{[\rho]} X(\mathcal O^2, G)= \frac{1}{2} \dim_{\rho} X(\widetilde{\mathcal O^2}, G)
+ \frac{1}{2}\sum_{[\![0,1]\!]\subset \partial \mathcal O^2 }
\widetilde{\chi}( [\![0,1]\!], \mathrm{Ad}\rho)
%
$$
where the sum runs on the components of $\partial\mathcal O^2$ that are intervals with mirror boundary,
$[\![0,1]\!]$.
\end{Proposition}

\begin{proof}
In the proof of Proposition~\ref{Prop:relsmooth} we use (based on \cite{HP2020}):
\begin{align*}
\dim_{[\rho]} X( \mathcal O^2, G)=&\dim_{[\rho]} \mathcal{R}^{\mathrm{good} }( \mathcal O^2,\partial\mathcal O^2, G)
 +\dim H^1(\partial\mathcal O^2, \mathrm{Ad}\rho),
 \\
\dim_{[\rho]} X( \widetilde{\mathcal O^2}, G)=&\dim_{[\rho]} \mathcal{R}^{\mathrm{good} }( \widetilde{\mathcal O^2},\partial\mathcal O^2, G)
 +\dim H^1(\partial\widetilde{\mathcal O^2}, \mathrm{Ad}\rho).
\end{align*} 
Furthermore, by Proposition~\ref{Prop:Lagrangian}:
$$
 \dim_{[\rho]} \mathcal{R}^{\mathrm{good} }( \mathcal O^2,\partial\mathcal O^2, G)=
 \frac12 \dim_{[\rho]} \mathcal{R}^{\mathrm{good} }( \widetilde{\mathcal O^2},\partial \widetilde{\mathcal O^2}, G).
$$
And from \eqref{eqn:reg} and \eqref{eqn:dimH1bc} in the proof of Proposition~\ref{Prop:relsmooth}:
\begin{align*}
 \dim H^1(\partial\widetilde{\mathcal O^2}, \mathrm{Ad}\rho) & = (2\mathsf{c}+{\mathsf{b}})\operatorname{rank} G,
 \\
 \dim H^1(\partial\mathcal O^2, \mathrm{Ad}\rho) & = (\mathsf{c}+\frac{\mathsf{b}}2)\operatorname{rank} G -\frac{1}{2}
 \widetilde \chi(\partial \mathcal O^2, \mathrm{Ad}\rho) .
\end{align*}
Using these formulas, we just need to know that 
$\widetilde \chi(\partial \mathcal O^2, \mathrm{Ad}\rho)$  is the sum of the twisted Euler characteristics of each component,
and that  $\widetilde \chi(S^1, \mathrm{Ad}\rho)=0$.
\end{proof}

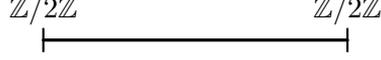
\begin{figure}[h]
\begin{center}
\begin{tikzpicture}[line join = round, line cap = round]
\draw[very thick] (0,0)--(4,0); 
\draw[thick] (0,-.15)--(0,.15);
\draw[thick] (4,-.15)--(4,.15);
\draw(0, .4) node{$\mathbb Z/2\mathbb{Z}$};
\draw(4, .4) node{$\mathbb Z/2\mathbb{Z}$};
\end{tikzpicture}
\end{center}
\caption{The orbifold $[\![0,1]\!]=\mathbb R/ D_\infty$.}
\label{Figure:Interval}
\end{figure}

The orbifold fundamental group of  $ [\![0,1]\!] $ is the infinite dihedral group $D_\infty$, so that $\mathbb{R}/D_\infty\cong [\![0,1]\!] $. 
In  a group presentation 
$$
 D_{\infty}=\langle \sigma_1, \sigma_2\mid  \sigma_1^2=\sigma_2^2= 1\rangle 
$$
the stabilizers of the mirror points are the cyclic groups of order $2$ generated by $\sigma_1$ and $\sigma_2$ respectively, see 
Figure~\ref{Figure:Interval}.
For further applications we need the following lemma:

\begin{Lemma}
 \label{Lemma:InfiniteDihedral}
 For any representation $\rho\colon D_{\infty}=\langle \sigma_1, \sigma_2\mid  \sigma_1^2=\sigma_2^2= 1\rangle \to G$:
 \begin{align*}
 \widetilde \chi( \mathbb{R}/D_\infty  ,\mathrm{Ad}\rho)  
  & = \dim \mathfrak{g}^{\rho(\sigma_1)}+\dim \mathfrak{g}^{\rho(\sigma_2)} -\dim \mathfrak{g}, \\
  \dim H^0(  \mathbb{R}/D_\infty ,\mathrm{Ad}\rho )&= \big(  \dim \mathfrak{g}^{\rho(\sigma_1\sigma_2)}+ 
  \dim \mathfrak{g}^{\rho(\sigma_1)}+\dim \mathfrak{g}^{\rho(\sigma_2)} -\dim \mathfrak{g}\big)/2 , \\
  \dim H^1(  \mathbb{R}/D_\infty ,\mathrm{Ad}\rho )&= \big(  \dim \mathfrak{g}^{\rho(\sigma_1\sigma_2)}- 
  \dim \mathfrak{g}^{\rho(\sigma_1)}-\dim \mathfrak{g}^{\rho(\sigma_2)} +\dim \mathfrak{g}\big)/2 .
 \end{align*}
\end{Lemma}

\begin{proof}
We compute the twisted Euler characteristic by using the simplicial structure of $\mathbb R/D_\infty$ with
one 1-cell (with trivial stabilizers) and two 0-cells (with stabilizers generated by $\sigma_1$ and $\sigma_2$ respectively):
\begin{equation}
\label{eqn:twistDinfty}
\begin{split}
\widetilde \chi(\mathbb R/D_\infty,\mathrm{Ad}\rho)  
&= \dim \mathfrak{g}^{\rho(\sigma_1)}+\dim \mathfrak{g}^{\rho(\sigma_2)} -\dim \mathfrak{g}\\
 & = \dim H^0(  \mathbb{R}/D_\infty ,\mathrm{Ad}\rho ) -  \dim H^1(  \mathbb{R}/D_\infty ,\mathrm{Ad}\rho ). 
\end{split}
\end{equation}
%
The orientation covering is denoted by $S^1\to \mathbb{R}/D_\infty$,  and $\pi_1(S^1)$ 
corresponds to the infinite cyclic subgroup of $D_{\infty}$ generated by $\sigma_1\sigma_2$.
As $S^1$ is a manifold and $\chi(S^1)=0$:
\begin{equation}
 \label{eqn:dimh1h0}
\dim H^1(S^1,\mathrm{Ad}\rho)\cong\dim H^0(S^1,\mathrm{Ad}\rho)\cong \mathfrak{g}^{\rho(\sigma_1\sigma_2)}.
 \end{equation}
By Lemma~\ref{Lemma:orientation}: 
\begin{equation}
 \label{eqn:h0plush1}
\dim H^0(  \mathbb{R}/D_\infty ,\mathrm{Ad}\rho ) +  \dim H^1(  \mathbb{R}/D_\infty ,\mathrm{Ad}\rho )=
\dim \mathfrak{g}^{\rho(\sigma_1\sigma_2)}. 
\end{equation}
Then the lemma follows from
\eqref{eqn:twistDinfty}
and
\eqref{eqn:h0plush1}.
\end{proof}

Notice that in the previous lemma we do not require anything for the representation $\rho$, that could be trivial.
The first direct application goes to the stabilizer of a corner reflector, the dihedral group $  D_{2k}$, by pre-composing
any representation $D_{2k}\to G$ with the natural surjection from the infinite dihedral group $D_{\infty}\to D_{2k}  $:


\begin{figure}[h]
\begin{center}
\begin{tikzpicture}[line join = round, line cap = round, scale=.9]
%
 \draw[ white, fill=gray!30!white, opacity=.3] (0,0)--(3,0)--(1,2)  to[out=210, in=90] (0,0);
\draw (0,0)--(3,0)--(1,2); 
\begin{scope}[shift={(0.05,+.1)}]
 \draw (0,0)--(3,0)--(1,2); 
\end{scope}
\draw[fill=black ]  (3-.05,0+.05) circle(.1);
\draw[fill=black ]  (1-.05,0+.05) circle(.1);
\draw[fill=black ]  (2.07-.05,1+.05) circle(.1);
\draw(3.3,-.3) node{$x_0$};
\draw(2.5,1.2) node{$x_2$};
\draw(1,-.4) node{$x_1$};
\end{tikzpicture}
\end{center}
\caption{Corner reflector modeled on $\mathbb R^2/ D_{2k}$, with the notation in Corollary~\ref{Coro:CornerCenter}.}
\label{Figure:Corner}
\end{figure}
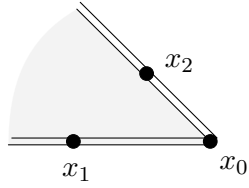

\begin{Corollary}
\label{Coro:CornerCenter}
Let $x_0\in\mathcal O^2 $ be a corner reflector and $x_1$ and $x_2$  mirror points in a neighborhood of $x_0$
separated by $x_0$, see Figure~\ref{Figure:Corner}. If $C_k\subset  \operatorname{Stab}(x_0)$ is the orientation preserving cyclic subgroup
of index 2,
then:
$$
  \dim \mathfrak{g}^{\rho(  \operatorname{Stab}(x_0) )}=   \big(  \dim \mathfrak{g}^{\rho( C_k )}+ 
  \dim \mathfrak{g}^{\rho( \operatorname{Stab}(x_1) )}+\dim \mathfrak{g}^{\rho( \operatorname{Stab}(x_2) )} -\dim \mathfrak{g}\big)/2 .
$$
\end{Corollary}

%

\subsection{Euclidean 2-orbifolds}
\label{section:Euclidean}

Along this subsection, let $\mathcal O^2$ be a Euclidean two orbifold without boundary. 
Set $\Gamma=\pi_1(\mathcal O^2)$. 
As $\operatorname{Isom}(\mathbb R^2)$ is the semi-direct product 
$\mathbb R^2\rtimes \mathrm{O}(2)$, we have an exact sequence (Bieberbach theorem):
\begin{equation}
 \label{eqn:exacteuclidean}
1\to \Gamma_0\to \Gamma\to \Lambda\to 1,
\end{equation}
with $\Gamma_0\cong\mathbb Z^2$ and $\Lambda\subset \mathrm{O}(2)$ the linear part. 
When $\mathcal O^2$ is orientable,
 $\Lambda$ is cyclic and when  $\mathcal O^2$ is non-orientable, $\Lambda$ is dihedral (or cyclic of order 2).
Furthermore, when $\mathcal O^2$ is orientable or a Coxeter group, $\Gamma_0$ is the maximal torsion
free subgroup and the sequence \eqref{eqn:exacteuclidean} splits, as $\Lambda$ 
corresponds to the stabilizer of a point in $\mathbb R^2$.
A  Coxeter group  in  $\operatorname{Isom}(\mathbb R^2)$  is a group generated by reflections along a square or a Euclidean triangle 
with angles an integer divisor of $\pi$.

\begin{figure}[ht]
\begin{center}
 \begin{tikzpicture}[line join = round, line cap = round, scale=.8]
   \begin{scope}[shift={(-8,1)}]
   \draw (0,0) circle [x radius=1.7, y radius=.8];
    \draw (1.15,0.1)  arc[x radius = 1.2, y radius = .5, start angle= 345, end angle= 195];
    \draw (0.95,0.0)  arc[x radius = 1, y radius = .5, start angle= 15, end angle= 165];
\draw (0,-1.75) node{$T^2$};
   \end{scope}
  \begin{scope}[shift={(-4,0)}]
 \draw (-1,0)--(1,0)--(1, 2)--(-1,2)--cycle; 
     \draw[red, fill=red] (-1,0) circle(0.04);
     \draw[red, fill=red] (1,0) circle(0.04);
     \draw[red, fill=red] (1,2) circle(0.04);
     \draw[red, fill=red] (-1,2) circle(0.04);
    \draw (-1+.4,0) to[out=90, in=0] (-1  , 0+.4); 
     \draw[dashed, very thin] (-1+.4,0) to[out=180, in=-90] (-1  ,  .4); 
          \draw (1-.4,0) to[out=90, in=180] (1  , 0+.4); 
     \draw[dashed, very thin] (1-.4,0) to[out=0, in=-90] (1  ,  .4);   
    \draw (-1+.4,2) to[out=-90, in=0] (-1  , 2-.4); 
     \draw[dashed, very thin] (-1+.4,2) to[out=180, in=90] (-1  ,  2-.4); 
          \draw (1-.4,2) to[out=-90, in=180] (1  , 2-.4); 
      \draw[dashed, very thin] (1-.4,2) to[out=0, in=90] (1  ,  2-.4);   
     \draw (-6/5,0) node{$2 $};
     \draw (6/5,0) node{$2 $};
     \draw (6/5,2) node{$2 $};
     \draw (-6/5,2) node{$2 $};
     \draw (0,-.75) node{$S^2(2,2,2, 2)$};
   \end{scope}
\begin{scope}[shift={(0,0)}]
 \draw (-1,0)--(1,0)--(0, 1.73)--cycle; 
     \draw[red, fill=red] (-1,0) circle(0.04);
     \draw[red, fill=red] (1,0) circle(0.04);
     \draw[red, fill=red] (0,1.73) circle(0.04);
    \draw (-1+.4,0) to[out=80, in=0] (-1+.4*.5  , 0+.4*.87); 
     \draw[dashed, very thin] (-1+.4,0) to[out=160, in=-100] (-1+.4*.5  , 0+.4*.87); 
    \draw (1-.4,0) to[out=110, in=180+0] (1-.4*.5  , 0+.4*.87); 
     \draw[ dashed, very thin] (1-.4,0) to[out=20, in=-80] (1-.4*.5  , 0+.4*.87); 
    \draw  (0-.4*.5, 1.73-.4*.87) to[out=-60, in=180+60] (0+.4*.5, 1.73-.4*.87);
    \draw[ dashed, very thin]  (0-.4*.5, 1.73-.4*.87) to[out=30, in=180-30] (0+.4*.5, 1.73-.4*.87);
     \draw (-6/5,0) node{$3 $};
     \draw (6/5,0) node{$3 $};
     \draw (1/5,1.8) node{$3 $};
     \draw (0,-.75) node{$S^2(3,3,3)$};
   \end{scope}
       \begin{scope}[shift={(4,0)}]
 \draw (-1,0)--(1,0)--(-1, 2)--cycle; 
     \draw[red, fill=red] (-1,0) circle(0.04);
     \draw[red, fill=red] (1,0) circle(0.04);
     \draw[red, fill=red] (-1,2) circle(0.04);
 \draw (-1+.4,0) to[out=90, in=0] (-1  , 0+.4); 
     \draw[dashed, very thin] (-1+.4,0) to[out=180, in=-90] (-1  ,  .4); 
  \draw (1-.6,0) to[out=100, in=215] (1-.6*0.707, 0+  .6*0.707); 
  \draw[ dashed, very thin] (1-.6,0) to[out=10, in=-55] (1-.6*0.707, 0+  .6*0.707); 
\draw (-1,2-.6) to[out=0-30, in=-45-30] (-1+.6*0.707, 2- .6*0.707); 
  \draw[ dashed, very thin]    (-1,2-.6) to[out=90-30, in=135+30](-1+.6*0.707, 2- .6*0.707); 
    
%
     \draw (-6/5,0) node{$2 $};
     \draw (6/5,0) node{$4 $};
     \draw (-1-1/5,2) node{$4 $};
     \draw (0,-.75) node{$S^2(2,4,4)$};
   \end{scope}
       \begin{scope}[shift={(8,0)}]
 \draw (-1,0)--(0.15,0)--(-1, 2)--cycle; 
     \draw[red, fill=red] (-1,0) circle(0.04);
     \draw[red, fill=red] (.15,0) circle(0.04);
     \draw[red, fill=red] (-1,2) circle(0.04);
 \draw (-1+.4,0) to[out=90, in=0] (-1  , 0+.4); 
     \draw[dashed, very thin] (-1+.4,0) to[out=180, in=-90] (-1  ,  .4);     
    \draw (.15-.4,0) to[out=110, in=180+0] (.15-.4*.5  , 0+.4*.87); 
     \draw[ dashed, very thin] (.15-.4,0) to[out=20, in=-80] (.15-.4*.5  , 0+.4*.87);
 \draw (-1 ,2-.6) to[out=-30, in=225+30] (-1+.5*.6, 2-.87*.6);
 \draw[ dashed, very thin] (-1 ,2-.6) to[out=90-30, in=225-30] (-1+.5*.6, 2-.87*.6);
     \draw (-6/5,0) node{$2 $};
     \draw (.15+1/5,0) node{$3 $};
     \draw (-6/5,2) node{$6 $};
     \draw (-.4,-.75) node{$S^2(2,3,6)$};
   \end{scope}
\end{tikzpicture}
\end{center}
\caption{The five closed, orientable, and Euclidean  2-orbifolds.}
\label{Figure:Euclidean}
\end{figure}
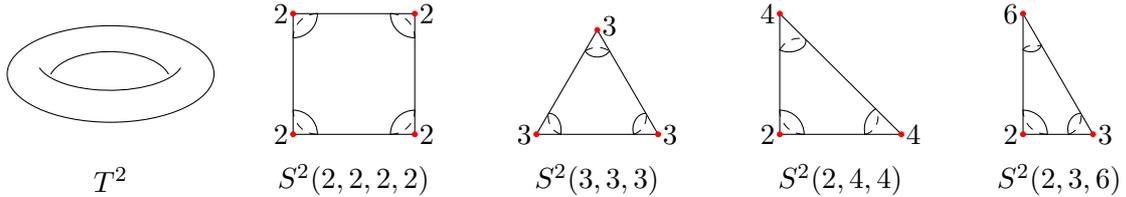

There are precisely five closed orientable Euclidean 2-orbifolds (Figure~\ref{Figure:Euclidean}): 
\begin{itemize}
 \item the 2-torus $T^2$,
\item a 2-sphere with 4 cone points of order $2$, $S^2(2,2,2,2)$, 
and 
\item 
three
2-spheres with cone points of order $p$, $q$ and $r$ satisfying 
$\frac{1}{p}+\frac{1}{q}+\frac{1}{r}=1$,
$S^2(p,q,r)$ for $(p,q,r)=(3,3,3)$, $(2,4,4)$ and $(2,3,6)$. 
\end{itemize}
The orbifold
$S^2(2,2,2,2)$ is sometimes called a  \emph{pillowcase}, geometrically is the 
double of a rectangle, and $S^2(p,q,r)$ is called a 
 \emph{turnover}, as it is the double of a triangle. For different values of 
 the ramifications, turnovers can also be spherical or Euclidean. 
Table~\ref{Table:kO} gives the cardinality of $\Lambda\cong \Gamma/\Gamma_0$
for these five Euclidean orbifolds.

\begin{table}[h]
\begin{center}
 \begin{tabular}{c|c}
 $\mathcal O^2$ & $ | \Gamma/\Gamma_0 | $ \\
 \hline
 $T^2$ &  1 \\
 $S^2(2,2,2,2)$ & 2 \\
 $S^2( 3,3,3)$ & 3  \\
 $S^2( 2,4,4)$ & 4  \\
 $S^2( 2,3,6)$ & 6 
 \end{tabular}
 \caption{Values of $k(\mathcal O^2)=\vert \Gamma/\Gamma_0\vert$
 when $\mathcal O^2$ is orientable, for $\Gamma_0$ maximal torsion free.}
 \label{Table:kO}
\end{center}\end{table}

As $\Gamma=\pi_1(\mathcal O^2)$ is virtually abelian, most of the representations we consider are not irreducible.
Instead, we deal with strong regularity:

\begin{Definition}
A representation of a Euclidean 2-orbifold $\rho\colon\pi_1 (\mathcal O^2)\to G$ is 
called \emph{strongly regular} if for the maximal normal subgroup
$\Gamma_0 < \pi_1(\mathcal O^2)$, with $\Gamma_0\cong \mathbb Z^2$,
\begin{enumerate}[(a)]
\item $\dim \mathfrak{g}^{\rho(\Gamma_0)}=\operatorname{rank} G$, and 
\item the projection
of $\rho(\Gamma_0)$ is contained in a connected abelian subgroup of $G/\mathcal Z(G)$.
\end{enumerate}
\end{Definition}

\begin{Theorem}
\label{Thm:EuclideanRed}
Assume that $\mathcal O^2$ is compact Euclidean and that 
$\rho\colon\pi_1(\mathcal O^2)\to G$ is \emph{strongly regular}. 
Then:
\begin{enumerate}[(i)]
 \item  $\rho$ is a smooth point of $R(\Gamma, G)$ of dimension
$$
 -\widetilde\chi(\mathcal O^2,\mathrm{Ad}\rho)+\dim G+ \dim \mathfrak{g}^{\rho(\Gamma)}.
$$
 \item  The component $X_0(\Gamma,G)$ of $X(\Gamma,G)$ that contains the character of $\rho$ has dimension:
$$
\dim X_0(\Gamma,G)  =-\widetilde\chi(\mathcal O^2,\mathrm{Ad}\rho)+2 \dim  \mathfrak{g}^{\rho(\Gamma)}=
\dim H^1(\mathcal O^2,\mathrm{Ad}\rho).
$$
\end{enumerate}  
\end{Theorem}

\begin{proof} (i) 
First consider the case $\mathcal O^2=T^2$, here we adapt an argument of \cite{HeusenerPorti15}. In this case 
$\widetilde\chi(\mathcal O^2,\mathrm{Ad}\rho)=0$, and by a theorem of Richardson 
\cite[Theorem~B]{Richardson}, 
$\rho$ is in the closure of $\mathbb{T}\times \mathbb{T}$ for 
$\mathbb{T}\subset G$ a maximal torus, $\mathbb{T}\cong (\mathbb{C}^*)^{\operatorname{rank}(G)}$ 
(meaning torus of an algebraic group). Therefore $\rho$ is contained in a component of dimension
at least 
$\dim G+\operatorname{rank} G$.  
As $\dim \mathfrak{g}^{\rho(\Gamma)}= \operatorname{rank} G$,
by Proposition~\ref{Prop:dimR}, $\dim T^{Zar}_\rho R(\Gamma,G)= \dim G+ \operatorname{rank} G $.
Therefore
$$
\dim_{\rho} R(\Gamma, G)\geq \dim G+\operatorname{rank} G=
\dim  T^{Zar}_\rho R(\Gamma,G).
$$
Using that the dimension of the Zariski tangent space is always larger than or equal to the dimension
of a variety, with equality only at smooth points, it follows that 
 $\dim_{\rho} R(\Gamma, G)=\dim  T^{Zar}_\rho R(\Gamma,G)$ and $\rho$ is a smooth point.

For general $\mathcal O^2$, notice that 
$ -\widetilde\chi(\mathcal O^2,\mathrm{Ad}\rho)+\dim G+ \dim \mathfrak{g}^{\rho(\Gamma)}
=\dim  T^{Zar}_\rho R(\Gamma,G)$, by Proposition~\ref{Prop:dimR} and all we need to show is that $\rho$ is a
smooth point. 
For that purpose, we use that the variety of representations of $\Gamma_0$ at 
$\rho\vert_{\Gamma_0}$ is smooth. As Goldman  obstructions are  natural  \cite[\S3]{HPS},
they are also equivariant. 
In particular
if we start with a equivariant infinitesimal deformation of $\Gamma_0$, then the sequence of obstructions to integrability is also equivariant.
Thus, by Proposition~\ref{prop:regularcovering} and Proposition~\ref{Prop:duality}, the sequence of obstructions
to integrability of $\Gamma$ vanishes.  This proves (i).

For (ii), notice that $\dim \mathfrak{g}^{\rho(\Gamma)}=\dim H^0(\Gamma,\mathrm{Ad}\rho)$ is 
upper semi-continuous  on
$R(\Gamma, G)$   \cite[Ch. III, Theorem 12.8]{Hartshorne}.  By smoothness,
the dimension of the Zariski tangent space reaches its minimum (along the irreducible component) at $\rho$. 
Hence, as
$$ 
\dim G+  \dim\mathfrak{g}^{\rho(\Gamma)} -\widetilde \chi(\mathcal O^2,\mathrm{Ad}\rho)= \dim T^{Zar}_\rho R(\Gamma,G)
$$
by Proposition~\ref{Prop:dimR},
$\dim\mathfrak{g}^{\rho(\Gamma)}$ reaches its minimum along the irreducible component at $\rho$. 
(Here we use that  $\widetilde \chi(\mathcal O^2,\mathrm{Ad}\rho)$ is constant along components,
because the elements of finite order cannot be deformed.) 
Equivalently the dimension of the orbit  by conjugation 
$$
\dim G-\dim  \mathfrak{g}^{\rho(\Gamma)}
$$ 
reaches its maximum at  $\rho$. 
It follows (for instance from 
\cite[Section~6.3]{Dolgachev}) 
that
$$
\dim X_0(\Gamma, G)=\dim R_0(\Gamma, G)- \big( \dim G-\dim \mathfrak{g}^{\rho(\Gamma)}  \big),
$$
which proves (ii).
\end{proof}

\begin{Remark}
\label{Remark:nonsmooth}
It does not follow from Theorem~\ref{Thm:EuclideanRed} that the variety of characters $X(\Gamma,G)$ is smooth at the character of $\rho$.
\end{Remark}

To explain Remark~\ref{Remark:nonsmooth},  we point out that 
even if $\rho$ is strongly regular,
its orbit may be not closed, so $\rho$ could be smooth but its character not. 
For instance, the parabolic representation of $\mathbb{Z}^2$ in $\mathrm{SL}(2,\mathbb C)$ is a smooth point of 
$R(\mathbb{Z}^2, \mathrm{SL}(2,\mathbb C))$, but its character is a singular point of  $X(\mathbb{Z}^2, \mathrm{SL}(2,\mathbb C))$.
The orbit of the parabolic representation is not closed because it accumulates at the trivial representation (that it is not regular). 

\medskip

Orientable Euclidean 2-orbifolds appear as peripheral subgroups of hyperbolic three-orbifolds of finite volume. Thus 
a natural representation of $\Gamma=\pi_1(\mathcal O^2)$ in  $\mathrm{PSL}(2,\mathbb C)$ occurs as the holonomy of 
horospherical cusps, homeomorphic to $\mathcal{O}^2\times [0,+\infty)$.
We are interested in the composition of this holonomy with the principal
representation $\tau\colon \mathrm{PSL}(2,\mathbb C)\to G$, in view of further computations
for hyperbolic three-orbifolds. 

\begin{Proposition}
 Let $\mathcal O^2$ be a closed orientable Euclidean 2-orbifold. Consider the holonomy
 $\operatorname{hol}\colon\pi_1(\mathcal{O}^2)\to\mathrm{PSL}(2,\mathbb C)$ as the
 horospherical section of a cusp,
 composed with the principal representation $\tau\colon \mathrm{PSL}(2,\mathbb C)
 \to G$. Then $\dim_{[\tau\circ \operatorname{hol}]} X(\mathcal O^2,G)$  is given by 
 Table~\ref{table:dimparabolic}.
 Furthermore, the dimension of $\mathfrak{g}^{\tau\circ\mathrm{hol}(\Gamma) } $ 
 is given by Table~\ref{table:stabparabolic}.
\end{Proposition}


\begin{table}[h]

\begin{center}
 \begin{tabular}{|c|ccccc|}
 \hline\noalign{\smallskip}
 \backslashbox{$G$}{
   $\mathcal O^2$ } &  $T^2$ &  $S^2(2,2,2,2)$ &  $S^2( 3,3,3)$ &  $S^2( 2,4,4)$ &  $S^2( 2,3,6)$ \\
   \hline \noalign{\smallskip}
$   \mathrm{PSL}(n,\mathbb C) $ & $2(n-1)$	
& $2\lfloor \frac n2  \rfloor$	&$2\lfloor\frac n3  \rfloor$	&   $2\lfloor\frac n4  \rfloor$	& $2\lfloor\frac n6  \rfloor$	
\\ 
$  \mathrm{PSp}(2m,\mathbb C) $ & 
$2m$ & $2m$& $2\lfloor\frac m3  \rfloor$ & $2\lfloor\frac m2  \rfloor$ & $2\lfloor\frac m3  \rfloor$ 
\\
$  \mathrm{PO}(2m+1,\mathbb C) $ & 
 $2m$ & $2m$& $2\lfloor\frac m3  \rfloor$ & $2\lfloor\frac m2  \rfloor$ & $2\lfloor\frac m3  \rfloor$ 
\\
$  \mathrm{PO}(2m,\mathbb C) $ & 
$2m+2$ & $4\lfloor \frac m2\rfloor $& $2\lfloor\frac {m}3  \rfloor$ & $2(\lfloor\frac {m}4  \rfloor+\lfloor\frac {m+1}4  \rfloor)$ &
$2(\lfloor\frac {m}6  \rfloor+\lfloor\frac {m+2}6  \rfloor)$ 
\\
$  \mathrm{G}_2   $ 
& 4 & 4 & 2 & 0 & 2 \\
$  \mathrm{F}_4   $ 
& 8 & 8 & 4 & 4 & 4  \\
$  \mathrm{E}_6   $ 
& 12 & 8 & 4 & 4 & 6 \\
$  \mathrm{E}_7  $ 
& 14 & 14 & 6 & 4 & 6 \\
$  \mathrm{E}_8   $ 
& 16 & 16 & 8 & 8 & 8 \\
\hline
\end{tabular}
\end{center}
\caption{Dimension of $X(\mathcal{O}^2, G)$ at $[\tau\circ\mathrm{hol}]$, for 
$\mathrm{hol}\colon\Gamma\to\mathrm{PGL}(2,\mathbb{C})$ the holonomy of
a horospherical cusp and $\tau\colon \mathrm{PGL}(2,\mathbb{C})\to G$
the principal representation.}\label{table:dimparabolic}
\end{table}

\begin{table}[h]
\begin{center}
 \begin{tabular}{|c|ccccc|}
  \hline\noalign{\smallskip}
 \backslashbox{$\mathfrak{g}$}{
   $\mathcal O^2$ }  &  $T^2$ &  $S^2(2,2,2,2)$ &  $S^2( 3,3,3)$ &  $S^2( 2,4,4)$ &  $S^2( 2,3,6)$ \\
   \hline \noalign{\smallskip}	
$ \mathfrak{sl}(n,\mathbb C) $
& $n-1$ & $\lfloor \frac {n-1}2  \rfloor$  & $\lfloor\frac {n-1}3  \rfloor$	  
& $\lfloor\frac {n-1}4  \rfloor$ & $\lfloor\frac {n-1}6  \rfloor$	
\\ 
$  \mathfrak{sp}(2m,\mathbb C) $
&  $m$  & 0&  $ \lfloor \frac{m+1}{3}\rfloor $& 0& 0
\\
$  \mathfrak{so}(2m+1,\mathbb C) $
&  $m$  & 0&  $ \lfloor \frac{m+1}{3}\rfloor $& 0& 0
\\
$  \mathfrak{so}(2m,\mathbb C) $
&  $m$  & $\delta^{m-1}_{2\mathbb Z}$ &  $ \lfloor \frac{m}{3}\rfloor +\delta^{m-1}_{3\mathbb Z}$& 
$\delta^{m-1}_{4\mathbb Z}$ &  $\delta^{m-1}_{6\mathbb Z}$
\\
$\mathfrak{g}_2$ & 2 & 0 & 0 & 0 & 0 \\
$\mathfrak{f}_4$ & 4 & 0 & 0 & 0 & 0 \\
$\mathfrak{e}_6$ & 6 & 2 & 0 & 2 & 0 \\
$\mathfrak{e}_7$ & 7 & 0 & 1 & 0 & 0 \\
$\mathfrak{e}_8$ & 8 & 0 & 0 & 0 & 0 \\
\hline
\end{tabular}
\end{center}
\caption{ Dimension of $\mathfrak{g}^{\tau\circ\mathrm{hol}(\Gamma) } $ 
for $\tau\circ\mathrm{hol}$ as in Table~\ref{table:dimparabolic},
where $\delta^i_{j\mathbb Z}=1$ if $i\in j\mathbb Z $ and 
$\delta^i_{j\mathbb Z}=0$ otherwise.}
\label{table:stabparabolic}
\end{table}

\begin{proof}
The principal representation $\tau$ maps regular elements to regular elements. Furthermore the parabolic holonomy of the maximal torsion-free subgroup
of $\pi_1(\mathcal O^2)$ is contained in a connected abelian subgroup of $ \mathrm{PGL}(2,\mathbb C)$,
hence $\tau\circ\mathrm{hol}\vert_{\pi_1(\mathcal O^2)}$
is strongly regular and Theorem~\ref{Thm:EuclideanRed} applies. 
The computation of dimension is based on the decomposition of $\mathrm{Ad}\circ\tau$ as a sum of 
$\operatorname{Sym}^{2 d_{\alpha}}$  \eqref{eqn:exponents}, according to the exponents $d_1,\ldots, d_r$ of $G$, where $r$ is the rank of $G$.
Let $V_{2d_{\alpha}}$ is the space of the representation $\operatorname{Sym}^{2 d_{\alpha}}$. We use the following three facts:
\begin{enumerate}
 \item $\dim V_{2d_{\alpha}}= 2 d_{\alpha}+1$, by \eqref{eqn:dimgexp}.
 \item For a cyclic group of rotations $C_k<\mathrm{PSL}(2,\mathbb C)$ of order $k$, with hyperbolic rotation angles in $\frac{2\pi}k\mathbb Z$, 
$\dim V_{2d_{\alpha}}^{C_k}= 2\lfloor\frac{d_{\alpha}}{k}\rfloor +1$.
 \item $\dim V_{2d_{\alpha}}^{\Gamma} =\delta^i_{k\mathbb Z}$, $\Gamma=\pi_1(\mathcal O^2) $, $\Gamma_0\cong \mathbb{Z}^2$ 
a maximal torsion-free subgroup and $k= \vert\Gamma/\Gamma_0\vert$,
where $\delta^i_{k\mathbb Z}=1$ if $i\in k\mathbb Z$ and $0$ if $i\not\in k\mathbb Z$.
\end{enumerate}
From this, the computation is elementary.
\end{proof}

\begin{Lemma}
Let $\mathcal O^2 =S^2(2,2,2,2)$, $\mathcal O^2= S^2(3,3,3)$, $\mathcal O^2= S^2(2,4,4)$, 
or  $\mathcal O^2= S^2(2,3,6)$. There is an irreducible representation of 
$\Gamma=\pi_1(\mathcal O^2)$  in $\mathrm{PSL}(k,\mathbb C)$ iff and only if 
 $k= \vert\Gamma/\Gamma_0\vert$. 
\end{Lemma}

\begin{proof}
We first construct the irreducible representation of $\Gamma$ in $\mathrm{PSL}(k,\mathbb C)$, 
 for $k= \vert\Gamma/\Gamma_0\vert$.
  From the exact sequence~\eqref{eqn:exacteuclidean},  
  start with a diagonal representation of $\Gamma_0$, then represent the cyclic group 
 $\Gamma/\Gamma_0$ as a cyclic permutation of  the 
 $k$ coordinates of $\mathbb{C}^k$. (In each case,  
 Proposition~\ref{Proposition:rhogood} applies and the dimension of the variety of characters is $2$.)

Let $\rho\colon\Gamma\to  \mathrm{PSL}(k,\mathbb C)$ be an irreducible representation. Let $\mathrm{Fix}(\rho(\Gamma_0)) \subset\mathbb {CP}^{k-1} $
denote the fixed point set of $\rho( \Gamma_0) $, that is nonempty by Kolchin's theorem. 
We consider the action of
the cyclic group $\Gamma/\Gamma_0$ on $\mathrm{Fix}(\rho(\Gamma_0)) $.
It can be checked that any possibility other than
$k=\vert \Gamma/\Gamma_0\vert$
and that 
$\mathrm{Fix}(\rho(\Gamma_0)) $ has 
$k$ points in generic position that are permuted cyclically
by
$\Gamma/\Gamma_0$ 
yields  a contradiction with irreducibility.
\end{proof}

%
%
%
%
%
 
\begin{Example}
We consider planar Euclidean Coxeter groups, namely generated 
by reflections along a square or a Euclidean triangle 
with angles an integer divisor of $\pi$, denoted by:
$$
Q(2,2,2,2),\quad T(3,3,3),\quad T(2,4,4),\quad T(2,3,6) .
$$
The respective orientation coverings are
$S^2(2,2,2,2)$, $S^2(3,3,3)$, $S^2(2,4,4)$ and  $S^2(2,3,6)$.
The group $\Gamma=\pi_1(\mathcal O)$ acts naturally in the Euclidean plane,  and one can construct representations 
in $G$ by realizing this action on a flat of the symmetric space associated to~$G$.
On this flat, one allows translations, but the rotational part is restricted to the action of the Weyl group, 
which in its turn is a Coxeter group that acts faithfully on a maximal flat (eg the Cartan subalgebra).
Hence one can construct representations by analyzing the Weyl group. 

We discuss the case of rank 2. Here the Weyl group is a dihedral group of order $2m$ generated 
by reflections along two lines at angle $\pi/m$, the walls of the Weyl chamber. 
For type $A_2$ the angle is $\pi/3$, for $B_2$, $\pi/4$, and for $G_2$, $\pi/6$.
This yields discrete faithful representations of the triangle group  $T(3,3,3)$ in $\mathrm{SL}(3,\mathbb{C})$
 and $G_2(\mathbb{C})$, of $T(2,4,4)$ in $\mathrm{Sp}(4,\mathbb{C})$, of $T(2,3,6)$ in $G_2(\mathbb{C})$ and of the quadrilateral
group $Q(2,2,2,2)$ in $\mathrm{Sp}(4,\mathbb{C})$ or $G_2(\mathbb{C})$ (or the corresponding split real forms).

The triangle groups constructed here have a one-parameter space of deformations (homoteties in the plane)
and the quadrilateral group a 2-parameter space. Those are half the dimension
of the deformation space of their orientable cover.
\end{Example}
 
 \subsection{Spherical 2-orbifolds}

Finally, we  also need to consider spherical two-orbifolds,
namely finitely covered by the 2-sphere. They are rigid, but we  require a cohomological computation later in the paper.

\begin{Lemma}
Let $\mathcal O^2$ be a spherical 2-orbifold and $\rho\colon\pi_1(\mathcal O^2)\to G$ a representation. Then 
 $$
\dim H^i(\mathcal O^2,\mathrm{Ad}\rho) =\begin{cases} \dim \mathfrak g^{\rho(\pi_1(\mathcal O^2))}, & \textrm{for } i=0,2 \\
                                         0, & \textrm{for } i=1
                                        \end{cases},
$$
and thus 
\begin{equation}
 \label{eqn:spherical}
\widetilde\chi(\mathcal O^2,\mathrm{Ad}\rho)= 2 \dim   \mathfrak g^{\rho( \pi_1(\mathcal O^2)  )}. 
\end{equation}
\end{Lemma}

\begin{proof}
Let $\mathcal O^2$ be a spherical 2-orbifold. As it 
is finitely covered by $S^2$, by Proposition~\ref{prop:regularcovering} we have
$$
H^i(\mathcal O^2,\mathrm{Ad}\rho)\cong H^i(S^2,\mathfrak{g})^{ \pi_1(\mathcal O^2) }
\cong H^i(S^2,\mathbb Z)\otimes \mathfrak{g}^{\rho(\pi_1(\mathcal O^2))}.
$$
Hence the lemma follows. 
\end{proof}

\begin{Remark}
 \label{Rem:sphericalSmooth}
For a spherical orbifold, the component of $R(\mathcal O^2,G)$ that contains $\rho$ is an orbit by conjugation of dimension
 $$\dim G-\dim  \mathfrak{ g}^{\rho(  \pi_1(\mathcal O^2)  )}.$$ In particular, by \eqref{eqn:spherical}
 it is a smooth variety of dimension 
 $$ -\widetilde\chi(\mathcal O^2,\mathrm{Ad}\rho) +\dim G + \dim \mathfrak g^{\rho( \pi_1(\mathcal O^2)  )}. 
 $$
\end{Remark}

\medskip

\section{Computing dimensions of the Hitchin component}
\label{Section:H2O}

Assume that $\mathcal O^2$ is hyperbolic, namely that the standard orbifold-Euler characteristic is negative, 
$\chi(\mathcal O^2)<0$.
For $G$ a \emph{simple} complex adjoint group, 
the principal representation $\tau\colon \mathrm{PSL}(2,\mathbb C)\to G$ is 
constructed from Jacobson-Mozorov theorem, and it restricts
to $\tau\colon \mathrm{PGL}(2,\mathbb R)\to G_{\mathbb R}$, for $ G_{\mathbb R}$ a (non-connected) split real form of $G$. 
The Hitchin component 
$$
\mathrm {Hit}(\mathcal O^2,G_{\mathbb R})
$$
is the connected component of  $ \mathcal R^{\mathrm{good}} (\mathcal O^2,   G_{\mathbb R})$
that contains the composition of $\tau$ with the holonomy representation of any Fuchsian structure on $\mathcal O^2$.
Hitchin components for surface groups have been intensively studied,
here we just mention that Alessandrini, Lee, and  Schaffhauser \cite{ALS} have introduced them for
 (possibly non-orientable) 2-orbifolds; furthermore they have shown   that 
 they are homeomorphic to $\mathbb R^N$, 
as for surfaces. 
 
The purpose of this section is to provide formulas for 
the dimension of the Hitchin component
of orbifolds, using the tools of   Section~\ref{Section:dim2}. 
Some of the formulas are equivalent to the ones 
already computed in \cite{ALS},
but the approach and presentation of results is different. 
In particular we give some applications in Propositions~\ref{prop:growthHitchin} 
and~\ref{Prop:AssPSp} to the dimension growth
of some families of Hitchin representations. Long and Thistlethwaite in \cite{LongThistlethwaite} 
have also computed the dimension of Hitchin 
components   in $\mathrm{PGL}(n,\mathbb R)$ for turnovers (2-spheres with three cone points).


\begin{Remark}
 For $\mathcal O^2$ closed and orientable, Atiyah-Bott-Goldman differential form defines a symplectic structure on 
 $\mathrm {Hit}(\mathcal O^2,G_{\mathbb R})$, Proposition~\ref{Prop:Symplectic}.
 \end{Remark}

The following is Proposition~\ref{Prop:LagrangianHitchinIntro} from the introduction:
 
 \begin{Proposition}
 \label{Prop:LagrangianHitchin}
For  $\mathcal O^2$ closed and non-orientable, with orientation covering $\widetilde{\mathcal O^2}$, via the restriction map $\mathrm {Hit}(\mathcal O^2,G_{\mathbb R})$ embeds in 
 $\mathrm {Hit}(\widetilde{\mathcal O^2},G_{\mathbb R})$ 
as a Lagrangian submanifold. 
\end{Proposition}

\begin{proof}
By Proposition~\ref{Prop:Lagrangian} the restriction map yields a 
Lagrangian \emph{immersion} from $\mathrm {Hit}(\mathcal O^2,G_{\mathbb R})$ to 
$\mathrm {Hit}(\widetilde{\mathcal O^2},G_{\mathbb R})$.
Furthermore, using the irreducibility of Hitchin representations, 
one can prove that this restriction map is injective (see for instance  Lemma~2.9 and Proposition~2.10 in \cite{ALS}). 
The restriction map to a finite index subgroup is proper, therefore
it is an embedding (or alternatively one can quote \cite[Corollary~2.13]{ALS}).
\end{proof}

\begin{Remark}
The representation $\tau\circ\mathrm{hol}$ is good, Lemma~\ref{Lemma:PrincipalGood}, so the real dimension of the Hitchin component of 
$\mathcal{O}^2$ in $G_{\mathbb{R}}$ is precisely the complex dimension of
$X(\mathcal O^2, G)$ at the character of $\tau\circ\mathrm{hol}$, namely
$-\tilde{\chi}(\mathcal O^2, \mathrm{Ad}\rho)$; we often use:
$$
\dim_{\mathbb{R}} \mathrm {Hit}(\mathcal O^2,G_{\mathbb R}) =\dim_{\mathbb{C},[\tau\circ \mathrm{hol}]} X(\mathcal O^2,G)
=-\tilde{\chi}(\mathcal O^2, \mathrm{Ad}\rho).
$$
\end{Remark}

\subsection{Hitchin component for $\mathrm{PGL}(n, \mathbb{R})$}.

%

Following Long and Thistlethwaite  \cite{LongThistlethwaite}, for $n,k$ positive integers we set
$$
\sigma(n,k  )=  q n+(q+1) r,
$$
where $q$ and $r$ are nonnegative integers such that
$ 
n= q\, k+r
$
(the quotient and reminder of integer division).

\begin{Proposition}
\label{Prop:stabilizerSL}
Let $x\in\mathcal O^2$ and $\tau=\operatorname{Sym}^{n-1}$.
For $\rho=\tau \circ \mathrm{hol}\colon\pi_1(\mathcal O^2)\to \mathrm{PGL}(n,\mathbb R)$,
$$
\dim  \mathfrak{g}^{\rho(\mathrm{Stab}(x) )}
=
\begin{cases}
 \sigma(n,k)- 1 &\textrm{ if }x\textrm{ is a cone point with }\mathrm{Stab}(x)\cong C_k,\\
 (\sigma(n,k)- 1)/2  &\textrm{ if }x\textrm{ is a corner with }\mathrm{Stab}(x)\cong D_{2k}\textrm{ and }n\textrm{ is odd}, \\
 (\sigma(n,k)-2)/2  &\textrm{ if }x\textrm{ is a corner with }\mathrm{Stab}(x)\cong D_{2k}\textrm{ and }n\textrm{ is even}, \\
 {(n^2-1)}/{2} &\textrm{ if }x\textrm{ is a mirror point and }n\textrm{ is odd} ,\\
 ({n^2}-2)/{2} &\textrm{ if }x\textrm{ is a mirror point and }n\textrm{ is even}. \\
\end{cases}
$$
\end{Proposition}

\begin{proof}
 For an  elliptic element $\gamma\in \Gamma$ of order $k$, we may compute 
 $\dim  \mathfrak{g}^{\rho(\gamma)}$ using that, up to conjugacy:
  $$
\rho(\gamma)
=\pm \operatorname{Sym}^{n-1}\begin{pmatrix} e^{\frac{\pi i}k} & 0  \\ 0 &  e^{-\frac{\pi i }k} \end{pmatrix}=
\pm \mathrm{diag} ( e^{\frac{\pi i}k(n-1)}, e^{\frac{\pi i}k(n-3)}, \ldots ,    e^{\frac{\pi i}k(n-1)} ).
$$
As $n=qk+r$, among the $k$ eigenvalues of this matrix,  $(k-r)$ eigenvalues have multiplicity $q$ and
$r$ eigenvalues have multiplicity $q+1$. Thus
$$
\dim  \mathfrak{g}^{\rho(\gamma)}= (k-r)q^2+r(q+1)^2-1= q n+(q+1) r-1=\sigma(n,k)-1.
$$
This computation yields the formula for cone points and mirror points. For corner points, 
apply Corollary~\ref{Coro:CornerCenter} and the previous computations.
\end{proof}

We then have: 

\begin{Proposition}
\label{prop:HitPGL}
Let $\mathcal O^2$ be a connected hyperbolic 2-orbifold, with underlying surface $  \vert \mathcal O^2\vert$. Then
$$
\dim \mathrm{Hit}(\mathcal{O}^2, \mathrm{PGL}(n,\mathbb R))  
=
 -({n^2-1})\chi(\vert \mathcal O^2\vert)+\sum_{i=1}^\mathsf{cp} (n^2-\sigma(n,k_i)) 
 +\sum_{j=1}^\mathsf{cr} \frac{n^2- \sigma(n,l_j) }{2}
+ \mathsf{b}\, 
\lfloor \frac{n^2}{2}\rfloor,
$$
where the cone points have order  $k_1,k_2,\dotsc,k_{\mathrm{cp}} $,
the corner reflectors have order  $2l_1,2l_2,\dotsc,$ $2l_\mathrm{cr} $ and 
$\mathsf{b}$ is the number of components of $\partial\mathcal O^2$ 
homeomorphic to $\mathbb R/ D_\infty$, an interval with mirror boundary. 
\end{Proposition}

\begin{proof}
We apply Proposition~\ref{Proposition:rhogood} to get that the dimension of the Hitchin component is
$-\chi( \mathcal O^2,\mathrm{Ad}\rho )$, which, by Proposition~\ref{Prop:stabilizerSL}, equals: 
\begin{equation}
 \label{eqn:dimHitSL}
 -({n^2-1})\chi(\mathcal O^2\setminus \Sigma)-\sum_{i=1}^\mathsf{cp} (\sigma(n,k_i)-1)+
\begin{cases}
 -\sum_{j=1}^\mathsf{cr} \frac{ \sigma(n,l_j)-2 }{2}  + (\mathsf{me} -\mathsf{mp})   \frac{n^2-2}{2}  &\textrm{if }n\textrm{ even}
 \\
 -\sum_{j=1}^\mathsf{cr} \frac{ \sigma(n,l_j) -1}{2}  + (\mathsf{me} -\mathsf{mp})  \frac{n^2-1}{2} & \textrm{if }n\textrm{ odd}
\end{cases}
\end{equation}
where  $\mathsf{me}$ is the number of  mirror edges (joining corner reflectors or mirror points in the boundary), and 
 $\mathsf{mp}$ the number of boundary points that are also mirrors.
Next we use the formulas (from additivity of the Euler characteristic):
 \begin{align*}
\chi(\vert\mathcal O^2\vert)=&\, \chi(\mathcal O^2\setminus\Sigma)+ \mathsf{cp}+\mathsf{cr}+\mathsf{mp}-\mathsf{me} ,\\
0= \chi(\partial \vert\mathcal O^2\vert)=&\,- \mathsf{me}-\mathsf{b}  +\mathsf{cr}+\mathsf{mp},
 \end{align*}
that are equivalent to
 \begin{align}
\label{eqn:xis}  
 -\chi(\mathcal O^2\setminus\Sigma)= &\,
- \chi(\vert\mathcal O^2\vert)+ \mathsf{cp}+\mathsf{b}  ,\\
\label{eqn:mimp}
\mathsf{me}-\mathsf{mp}=&\,\mathsf{cr}-\mathsf{b} .
 \end{align}
Using  \eqref{eqn:xis} and \eqref{eqn:mimp},   \eqref{eqn:dimHitSL} becomes:
$$
 -({n^2-1})\chi(\vert \mathcal O^2\vert)+\sum_{i=1}^\mathsf{cp} (n^2-\sigma(n,k_i))
 +\sum_{j=1}^\mathsf{cr} \frac{n^2- \sigma(n,l_j) }{2}
+
 \begin{cases}
 \mathsf{b}  \, \frac{n^2}{2}  &\textrm{if }n\textrm{ even}
 \\
 \mathsf{b} \, \frac{n^2-1}{2} & \textrm{if }n\textrm{ odd}
\end{cases}
$$
which proves the proposition.
\end{proof}

%

\begin{Example}[Proved in \cite{LongThistlethwaite} when $\mathsf{cp}=3$]
 \label{Exampe:spheres}  
Let $S^2(k_1,\dotsc,k_\mathsf{cp})$ be a sphere with $\mathsf{cp}\geq 3$ cone points of order $k_1,\dotsc, k_\mathsf{cp}$ respectively (and satisfying 
$\frac{1}{k_1}+\frac{1}{k_2}+\frac{1}{k_3}<1$ for $\mathsf{cp}=3$).
Then, the dimension of the Hitchin component of $S^2(k_1,\dotsc,k_\mathsf{cp})$ in $\mathrm{PGL}(n,\mathbb R)$ is
$$
n^2(\mathsf{cp}-2)+ 2-\sum_{i=1}^\mathsf{cp} \sigma(n,k_i).
$$
Consider next $P(k_1,\dotsc,k_\mathsf{cr})$  the orbifold generated by reflections on a 
hyperbolic polygon with angles $\frac{\pi}{k_1},\dotsc, \frac{\pi}{k_\mathsf{cr}}$. In particular $S^2(k_1,\dotsc,k_\mathsf{cr})$
is its orientation covering and therefore the 
dimension of the Hitchin component of $P(k_1,\dotsc,k_\mathsf{cr})$ in $\mathrm{PGL}(n,\mathbb R)$ is
$$
\frac{n^2(\mathsf{cr}-2)+ 2-\sum_{i=1}^\mathsf{cr} \sigma(n,k_i)}2.
$$
\end{Example}


\begin{Proposition}
\label{Prop:MaxHSL}
 Every component of $\mathcal R^{\mathrm{good}} (\mathcal O^2,\mathrm{PGL}(n,\mathbb R))$ that contains
 representations that are boundary regular has dimension at most the dimension of the Hitchin component.
\end{Proposition}

\begin{proof}
Assume first that $\mathcal O^2$ is orientable, 
hence the singular locus is a finite union of cone points,
that have cyclic stabilizers.
We claim that for an elliptic element $\gamma$, 
$\dim(\mathfrak{g}^{\rho(\gamma)})$ is minimized for $\rho$ in the Hitchin
component.
Assuming the claim,  $\widetilde \chi(\mathcal O^2,\rho)$ is minimized for
$\rho$ in the Hitchin component, and apply Proposition~\ref{Proposition:rhogood}.
To prove the claim, since  $\rho(\gamma)$ has order $k$, its eigenvalues have
multiplicities $\{n_1,\ldots,n_k\}$, with $n_1+\cdots+n_k=n$.
Therefore  $\dim(\mathfrak{g}^{\rho(\gamma)})= n_1^2+\cdots+n_k^2-1$. This
function is minimized when $|n_i-n_j|\leq 1$, 
because if we replace $n_1$ by $n_1+1$ and $n_2$ by $n_2-1$, then $n_1^2+n_2^2$
increases by $2(n_1-n_2+1)$. 
As, for a rotation of angle $2\pi/k$, $\operatorname{Sym}^{n-1}(  \rho(\gamma))$
satisfies $\vert n_i-n_j\vert\leq 1$, it minimizes the required dimension.

In the non-orientable case we apply Proposition~\ref{Prop:boundaryNO} and Lemma~\ref{Lemma:InfiniteDihedral} and use the claim we have proved, 
that
for an elliptic element $\gamma$, 
$\dim(\mathfrak{g}^{\rho(\gamma)})$ is minimized for $\rho$ in the Hitchin
component.
%
\end{proof}

\begin{Example}
In Table~\ref{Table:dimHitchin334} we compute the   
dimension of  $\mathrm {Hit}(S^2(3,3,4),\mathrm{PSL}(n,\mathbb R) )$.
Notice that the difference with $n^2/12$ is bounded, this is a particular case of the next proposition.
\end{Example}

\begin{table}[h]
  \begin{center}
 \begin{tabular}{r|l}
  $n\mod 12$ & $\dim \mathrm {Hit}( S^2(3,3,4)  , \mathrm{PGL}(n,\mathbb R))$ \\
  \hline 
  0 & $n^2/12+2$ \\
  $\pm 1,\pm 5$ & $(n^2-1)/12$ \\
  $\pm 2$ & $(n^2-4)/12$ \\
  $\pm 3$ & $(n^2+15)/12$ \\
  $\pm 4$ & $(n^2+8)/12$ \\
  $ 6$  & $n^2/12+1$
 \end{tabular}
  \end{center}
  \caption{Dimension of the Hitchin component of $S^2(3,3,4)$.
  The dimension for the group generated by reflections on a triangle of angles  
  $(\frac\pi3,\frac\pi3,\frac\pi4)$ is half of it. } \label{Table:dimHitchin334}
\end{table}

\begin{Proposition}
\label{prop:growthHitchin}
There exists a uniform constant $ C(\mathcal O^2)$ depending only on $\mathcal O^2$ such that
$$
\left| 
\dim   
\mathrm {Hit}(\mathcal O^2, \mathrm{PGL}(n,\mathbb R))   
+\chi(\mathcal O^2) (n^2-1)
\right|
\leq C(\mathcal O^2).
$$
Here $0> \chi(\mathcal O^2)\in\mathbb{Q}$ denotes the (untwisted) orbifold-Euler characteristic.
\end{Proposition}

\begin{proof}
Let $\rho=\tau\circ\mathrm{hol}$.
Assume first that $\mathcal O^2$ is orientable.
If $\gamma$ is an elliptic element of order $k$, 
we shall prove that 
$
\left|\frac{n^2-1}{k}-\dim \mathfrak{g}^{ \rho(\gamma)}\right|
$
is bounded uniformly on $n$.
%
Namely, if $n=k q+ r$, then $
\dim \mathfrak{g}^{\rho(\gamma)}
= k q^2+ 2 r q +r-1
$. On the other hand $n^2= k^2 q^2+ 2 r k q + r^2$, thus:
$$
\left|\frac{n^2-1}{k}- \dim \mathfrak{g}^{\rho(\gamma)}  \right|= 
\left|\frac{r^2-1}{k}-r+1\right|=(r-1)\left(1-\frac{r+1}{k}\right)
\leq k .
$$
In the orientable case, as all stabilizers are cyclic this estimate tells that 
the difference $\tilde \chi(\mathcal O^2,\mathrm{Ad}\rho)-\chi(\mathcal O^2)(n^2-1)$
is uniformly bounded, for $\rho=\mathrm{Sym}^{n-1}\circ\mathrm{hol} $. With
Proposition~\ref{Proposition:rhogood}, this yields the proposition for $\mathcal O^2$ orientable.

In the non-orientable case, consider $\widetilde{\mathcal O^2}\to\mathcal O^2$ the orientation covering.
By Proposition~\ref{Prop:boundaryNO} and Lemma~\ref{Lemma:InfiniteDihedral}
$$
\left|\dim \mathrm {Hit}(\mathcal O^2, \mathrm{PGL}(n,\mathbb R)) -
\frac{1}{2} \dim \mathrm {Hit}(\widetilde{\mathcal O^2}, \mathrm{PGL}(n,\mathbb R))
\right|\leq
\frac12\big( \dim G-2\dim \mathfrak{g}^{\rho (C_2)}\big) \textsf{b}(\mathcal O^2)
$$
where $\rho( C_2)$ denotes the image of any cyclic stabilizer of order 2 and $\textsf{b}(\mathcal O^2)$ is the number 
of components of $\partial \mathcal O^2$ homeomorphic to $\mathbb R/ D_\infty=[ \![0,1]\! ]$ (an interval with mirror
end-points). As 
$$
\frac12\big( \dim G-2\dim \mathfrak{g}^{\rho (C_2)}\big)= \frac{n^2-1}2- (\sigma(n,2)-1)=
\begin{cases}
 1/2 & \textrm{for even }n\\
 0 & \textrm{for odd }n
\end{cases}
$$
the proposition follows.
\end{proof}

\subsection{Hitchin component for $\mathrm{PSp}^\pm(2m)$}

For   $G=\mathrm{PSp}^\pm(2m)$
the principal representation $\tau\colon\mathrm{PGL}(2,\mathbb R)\to G$  is the restriction of
$\operatorname{Sym}^{2m-1}\colon\mathrm{PGL}(2,\mathbb R)\to \mathrm{PGL}(2m-1,\mathbb R) $. 
Let 
$$J=\begin{pmatrix} 0 & 1 \\ -1 & 0 \end{pmatrix}.$$ 
As 
$$
A^tJA=\det(A) J,
$$ 
for every  $ A\in \mathrm{GL}(2,\mathbb R)$,
by restricting to matrices with determinant $\pm 1$, we get the inclusion 
of $\operatorname{Sym}^{2m-1}(\mathrm{PGL}(2,\mathbb R))$ in  $G=\mathrm{PSp}^\pm(2m)$.
Furthermore, the antisymmetric matrix is $\operatorname{Sym}^{2m-1}(J)$.


Write
$$
2m=k q+r
$$
with $q,m\in \mathbb{Z}$, $q,r\geq 0$ and $r< k$. Recall that:
$$
\sigma(2 m, k)= 2m q + r(q+1).
$$

\begin{Proposition}
\label{Prop:HitPsp}
 Let $\rho\in \mathrm{Hit}(\mathcal O^2, G)$, for $G=\mathrm{PSp}^\pm(2m)$ or $\mathrm{PO}(m,m+1)$
 and let  $x\in\mathcal O^2$.
 \begin{enumerate}
  \item If $x$ is a cone point with $\operatorname{Stab}(x)\cong C_k$ then
    $$
   \dim\mathfrak{g}^{\rho(C_k) }=
   \begin{cases}
\frac{\sigma(2 m, k)}{2} & \textrm{for } k \textrm{ even,} \\
\frac{\sigma(2 m, k)}{2}
 +\lfloor \frac{q+1}{2}\rfloor 
 & \textrm{for } k \textrm{ odd.} \
   \end{cases}
  $$
  \item If $x$ is a mirror point with $\operatorname{Stab}(x)\cong C_2$ then $ \dim\mathfrak{g}^{\rho(C_2) }= m^2$.
  \item If $x$ is a corner reflector with $\operatorname{Stab}(x)\cong D_{2k}$ then
$$
   \dim\mathfrak{g}^{\rho(D_k) }=
   \begin{cases}
\frac{\sigma(2 m, k)}{4}-\frac{m}{2} & \textrm{for } k \textrm{ even,} \\
\frac{\sigma(2 m, k)}{4}-\frac{m}{2}
 +\frac{1}{2}\lfloor \frac{q+1}{2}\rfloor 
 & \textrm{for } k \textrm{ odd.} \
   \end{cases}
  $$
 \end{enumerate}
\end{Proposition}

  
\begin{proof}
By Remark~\ref{Remark:equalpPspO}, it is sufficient to make the computations for $G=\mathrm{PSp}(2m,\mathbb C)$.
The bilinear form $\operatorname{Sym}^{2m-1}(J)$ has matrix
$$
\begin{pmatrix}
 0 & \cdots & 0 & 0 & 1 \\
 0 & \cdots & 0 & -1 & 0 \\
 0 &  \cdots & 1 & 0 & 0 \\
 \vdots & \ddots  &   \vdots&  \vdots & \vdots \\
 -1& \cdots &0 & 0 & 0
\end{pmatrix}.
$$
thus  
$$
\mathfrak{sp}(2m)=\{ (a_{i,j}\in M_{2m,2m}(\mathbb{C}) \mid  a_{i,j}=(-1)^{i+j+1}a_{2m+1-j,2m+1-i}  \}.
$$ 
From this expression, as a diagonal matrix $D\in \mathrm{PSp}(2m,\mathbb C)$ may be written as 
$$
D=\operatorname{diag}(\lambda_1,\dotsc,\lambda_m,\lambda_m^{-1},\dotsc, \lambda_1^{-1} ),
$$
we have that 
$$
\dim(\mathfrak{sp}(2m)^D)=\frac{1}{2}\left(\dim(\mathfrak{gl}(2m)^D )+ \dim ( \mathfrak{ad}^D ) \right)
$$
where $\mathfrak{ad}=\{ (a_{i,j}\in M_{2m,2m}(\mathbb{C})\mid a_{ij}= 0 \textrm{ if } i+j\neq 2m+1\}  
\subset \mathfrak{sp}(2m)$ is the anti-diagonal, namely the subspace of matrices
$$
\begin{pmatrix}
 0 & \cdots & 0 & 0 & * \\
 0 & \cdots & 0 & * & 0 \\
 0 &  \cdots & * & 0 & 0 \\
 \vdots & \ddots  &   \vdots&  \vdots & \vdots \\
 *& \cdots &0 & 0 & 0
\end{pmatrix}.
$$
A rotation of angle $\frac{2\pi}k$ in $\mathrm{PGL}(2,\mathbb C)$  is conjugate
to $\pm \begin{pmatrix}  e^{\frac{\pi i}k} & 0 \\ 0 &   e^{-\frac{\pi i}k} \end{pmatrix}$ and
\begin{align*}
{D}&=
\operatorname{Sym}^{2m-1} \left( 
\pm \begin{pmatrix}  e^{\frac{\pi i}k} & 0 \\ 0 &   e^{-\frac{\pi i}k} \end{pmatrix}\right)
\\
&=\pm \operatorname{diag}( e^{\frac{\pi i}k(2m-1) }, e^{\frac{\pi i}k(2m-3) }, \ldots ,e^{\frac{\pi i}k} , 
e^{-\frac{\pi i}k }, \ldots , e^{-\frac{\pi i}k(2m-1) }) .
\end{align*}
Therefore,   the (adjoint) action of 
${D}$, 
on the antidiagonal $ \mathfrak{ad}$ has eigenvalues
$$
\{
e^{\frac{2\pi i}k(2m-1) }, e^{\frac{2\pi i}k(2m-3) }, \ldots ,e^{\frac{2\pi i}k} , 
e^{-\frac{2\pi i}k }, \ldots , e^{-\frac{2\pi i}k(2m-1) }
\}.
$$
Thus, $\dim  \mathfrak{ad}^{{D} }$ is the number appearances of  $1$ in this list of eigenvalues:
$$
\dim  \mathfrak{ad}^{{D} }=\begin{cases}
                               0 & \textrm{for }k\textrm{ even},\\
                               2\lfloor\frac{2 m-1 +k}{2k}\rfloor=2\lfloor\frac{q+1}{2}\rfloor & \textrm{for }k\textrm{ odd}.
                              \end{cases}
$$
On the other hand, 
$\mathfrak{gl}(2m)^{D}= 2mq+r(q+1)$ (because $D$ has $r$ eigenvalues with multiplicity $q+1$
and $k-r$ eigenvalues with multiplicity $q$), what concludes item 1. 

Item 2 is a particular case of item 1, and item 3 follows from Corollary~\ref{Coro:CornerCenter} and the previous cases.
\end{proof}

With a similar proof as for Proposition~\ref{prop:growthHitchin}, we can show:

\begin{Proposition}
\label{Prop:AssPSp} 
There exists a uniform constant $ C(\mathcal O^2)$ depending only on $\mathcal O^2$ such that
 $$
\Bigg|\dim {\mathrm{Hit}}(\mathcal O^2, \mathrm{PSp}^{\pm} (2m) )  +\chi(\mathcal O^2)  \dim  ( \mathrm{PSp}^{\pm} (2m)) 
+ \Bigg( \sum_{k\textrm{ even}}\tfrac{\mathsf{cp}_k}{k}+ 
\sum_{l\textrm{ even}}\tfrac{\mathsf{cr}_l}{2l}+\tfrac{\mathsf{b}}{2}\Bigg)  m
\Bigg| \leq  C(\mathcal O^2),
$$ 
where $ \chi(\mathcal O^2)$ denotes the untwisted Euler characteristic, 
$\mathsf{cp}_k$  the number of cone points with stabilizer of order $k$, 
$\mathsf{cr}_k$   the number of corner reflectors with stabilizer of order $2l$, 
and $\mathsf{b}$    the number of components of $\partial\mathcal O^2$ 
homeomorphic to $\mathbb R/ D_\infty$, an interval with mirror boundary. 
\end{Proposition}

Using  the same argument as in the proof of Proposition~\ref{Prop:MaxHSL}, 
and longer computations,
one can also prove:

\begin{Proposition}
\label{Prop:MaxHPSp}
Every component of $\mathcal R^{\mathrm{good}}(\mathcal O^2,\mathrm{PSp}^{\pm}(2m))$
that contains representations that are boundary regular
has dimension at most the dimension of the Hitchin component.
\end{Proposition}

\subsection{Exponents of a simple Lie algebra}

Recall from Subsection~\ref{subsection:principal}   that the exponents
$d_1,\ldots ,d_r\in\mathbb N$ of the Lie algebra $\mathfrak{g}$ are defined by the equation 
$$
 \mathrm{Ad}\circ\tau= \bigoplus_{\alpha=1}^r 
 \operatorname{Sym}^{2 d_\alpha}, 
$$
Equation~\eqref{eqn:exponents}, where $\tau $ is the principal representation. Here $r=\operatorname{rank}
\mathfrak{g}$ and $\sum_{\alpha=1}^r(2d_\alpha+1)=\dim \mathfrak{g}$.


\begin{Lemma}
\label{lemma:stabilizersHitchin}
For $\rho\in \mathrm{Hit}(\mathcal{O}^2, G)$ and $x\in\Sigma_{\mathcal{O}^2}$ a point 
in the branching locus with nontrivial stabilizer
$\Gamma_x$:
 $$
 \dim\mathfrak{g}^{\rho(\Gamma_x) }=
 \begin{cases}
  \sum_{\alpha=1}^{r } (2\lfloor \frac{d_{\alpha}}{k}\rfloor + 1) & 
  \textrm{ if } \Gamma_x\cong C_k \\
  \left( \sum_{\alpha=1}^{r } \lfloor \frac{d_{\alpha}}{k}\rfloor \right) 
  + \# \{d_{\alpha}\mid d_{\alpha}\textrm{ is even}\} & \textrm{ if } \Gamma_x\cong D_{2k} 
 \end{cases}
 $$
where $C_k$ is cyclic of order $k$, and $D_{2k}$ dihedral of order $2k$,  $r= \mathrm{rank} ( \mathfrak{g} )$, and
$\{d_1,\ldots, d_{r }\}$ are the exponents of the Lie algebra.  
 \end{Lemma}

\begin{proof}
For a Fuchsian representation, the  image of the cyclic group $C_k$ is generated by a matrix conjugate to
\begin{equation}
\label{eqn:cyclicmatrix}
\pm \begin{pmatrix}
     e^{\frac{\pi i}{k}} & 0 \\
     0 & e^{-\frac{\pi i}{k}}
    \end{pmatrix} .
\end{equation}
Using \eqref{eqn:exponents},
the contribution of the $\alpha$-th exponent to $ \dim\mathfrak{g}^{\rho(\Gamma_x) }$ is the
multiplicity of eigenvalue $1$ in the ${2 d_{\alpha}}$-th symmetric power of \eqref{eqn:cyclicmatrix}.
Namely,  the number of appearances of $1$ in 
$
\{
     e^{\frac{ 2\pi i}{k} d_{\alpha}},e^{\frac{2\pi i}{k}(d_{\alpha}-1)},\dotsc, e^{\frac{2\pi i}{k}(-d_{\alpha})}
\}
$
which is precisely $2\lfloor \frac{d_{\alpha}}{k}\rfloor + 1$.


For the dihedral group,
Corollary~\ref{Coro:CornerCenter}, the cyclic case, and \eqref{eqn:dimgexp} yield  that the dimension is
$$
\dim\mathfrak{g}^{\rho(D_k)}=
(2\dim\mathfrak{g}^{\rho(C_2)}+\dim\mathfrak{g}^{\rho(C_k)}-\dim\mathfrak{g}   )/2\\ 
=\sum_{\alpha=1}^r(2\lfloor\frac{d_{\alpha}}{2}\rfloor+ 1 +\lfloor\frac{d_{\alpha}}{k}\rfloor  -d_{\alpha})
$$
and $2\lfloor\frac{d_{\alpha}}{2}\rfloor+ 1   -d_{\alpha}=
\begin{cases}
 1 & \textrm{for } d_{\alpha} \textrm{ even,}\\
 0 & \textrm{for } d_{\alpha} \textrm{ odd.}
\end{cases}
$
 \end{proof}

Lemma~\ref{lemma:stabilizersHitchin} allows to give a new proof of Proposition~\ref{Prop:stabilizerSL}.

As the exponents for $\mathfrak{sp}(2 m)$ and $\mathfrak{s0}( m, m+1)$ are the same,
($1,3, 5\cdots,2m-1$), by Lemma~\ref{lemma:stabilizersHitchin}  we have:

\begin{Remark}
\label{Remark:equalpPspO}
Let $\rho_1\in \mathrm{Hit}(\mathcal O^2, \mathrm{PSp}^{\pm} (2m)) $ and 
$\rho_2\in \mathrm{Hit}(\mathcal O^2, \mathrm{PO}(m+1,m) ) $.
For any $x\in \mathcal O^2$,
$$
 \dim\mathfrak{g}^{\rho_1(\Gamma_x) }= \dim\mathfrak{g}^{\rho_2(\Gamma_x) }.
$$
\end{Remark}

Another consequence of Lemma~\ref{lemma:stabilizersHitchin} is the following.

\begin{Corollary}[\cite{ALS}]
Let $\mathcal O^2$ be a connected hyperbolic 2-orbifold. 
If $\{d_1,\ldots, d_{\mathrm{rank} G}\}$ are the exponents of $G$. 
Then the dimension of
${\mathrm{Hit}}(\mathcal O^2,G)$ is
$$
 -\chi(\vert \mathcal O^2\vert)\dim G
 +\sum_{\alpha=1}^{\mathrm{rank} G }\left(
 \sum_{i=1}^\mathsf{cp} 2( d_\alpha-\lfloor\frac{d_\alpha}{k_i}\rfloor)
+\sum_{j=1}^\mathsf{cr} (d_\alpha- \lfloor\frac{ d_\alpha}{l_j}\rfloor) 
+  2 \mathsf{b}\, \lfloor  \frac{d_\alpha+1}{2}  \rfloor 
\right)
$$
where the cone points have order  $k_1,k_2,\dotsc,k_\mathsf{cp} $, the corner reflectors have order  
$2l_1,2l_2,\dotsc,$ $2l_\mathsf{cr}$, 
and 
$\mathsf{b}$ is the number of components of $\partial\mathcal O^2$ homeomorphic to $\mathbb R/ D_\infty$, an interval with mirror boundary. 
\end{Corollary}

\begin{proof}
We follow the same scheme as the proof of Proposition~\ref{prop:HitPGL}.
Set $\delta^i_{j\mathbb Z}=1$ if $i\in j\mathbb Z $ and 
$\delta^i_{j\mathbb Z}=0$ otherwise.
Firstly we apply Proposition~\ref{Proposition:rhogood} and  Lemma~\ref{lemma:stabilizersHitchin}
to get that the dimension of the Hitchin component is
$-\chi( \mathcal O^2,\mathrm{Ad}\rho )$, which equals: 
$$
 -\chi(\mathcal O\setminus \Sigma)\dim G
 +\sum_{\alpha=1}^{\mathrm{rank} G}\left(
- \sum_{i=1}^\mathsf{cp} (2\lfloor\frac{d_\alpha}{k_i}\rfloor+1)
-\sum_{j=1}^\mathsf{ cr} ( \lfloor\frac{ d_\alpha}{l_j}\rfloor +\delta^{d_\alpha}_{2\mathbb Z} ) 
+ (\mathsf{me} -\mathsf{mp}) (2 \lfloor  \frac{d_\alpha}{2}  \rfloor +1)
\right)
$$ 
where  $\mathsf{me}$ is the number of  mirror edges (joining corner reflectors or mirror points in the boundary), and 
$\mathsf{mp}$ the number of boundary points that are also mirrors, and $\delta^{d_\alpha}_{2\mathbb Z}= 1$ if $d_\alpha$
is even, and $0$ if $d_\alpha$ is odd. Then the corollary follows from \eqref{eqn:xis} and \eqref{eqn:mimp}.
\end{proof}

This computation of dimensions is contained in \cite[Theorem~1.2]{ALS}, where it is proved that the Hitchin component is a cell.

For exceptional groups the dimension can be directly computed from this corollary and  the exponents in Table~\ref{Table:exponents}.
For instance, for (the real split form of the adjoint group)  $G_2$, as the indices are $1$ and $5$ the dimension of the Hitchin component is
$$
-14\chi(\vert\mathcal O^2\vert) + 8 \mathsf{cp}_2 + 10 \mathsf{cp}_{3,4,5}  + 12 \mathsf{cp}_{\geq 6}
 + 4 \mathsf{cr}_4 + 5 \mathsf{cr}_{6,8,20} + 6 \mathsf{cr}_{\geq 12} + 8 b,
$$
where $\mathsf{cp}_i$ is the number of cone points with stabilizer $C_i$ and $\mathsf{cr}_{2i}$ of corner
reflectors with stabilizer $D_{2i}$.

\section{Representations of 3-orbifolds}
\label{Section:dim3}

We restrict to orientable 3-orbifolds. In particular their singular locus is a union of circles, 
proper segments, and trivalent graphs.
The isotropy groups or stabilizers of edges are cyclic, and the isotropy groups or stabilizers of trivalent
vertices are non-cyclic finite subgroups of $\mathrm{SO}(3)$ (hence dihedral, tetrahedral, octahedral or icosahedral). See Figure~\ref{Figure:local_models}.

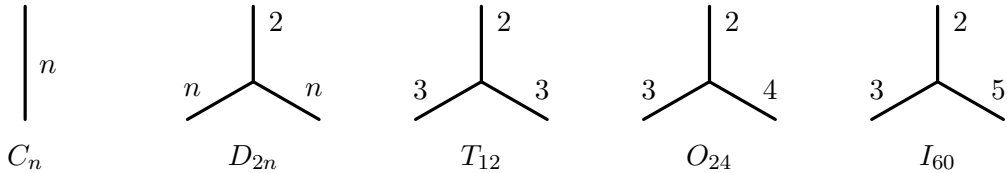
\begin{figure}[ht]
\begin{center}
 \begin{tikzpicture}[line join = round, very thick, line cap = round, scale=1]
   \begin{scope}[ shift={(-3,0)}]
  \draw (0,-.5)--(0,1);
   \draw (.3, 1/5) node{$n $};
   \draw (0,-1) node{$C_{n}$};
   \end{scope}
 \begin{scope}[shift={(0,0)}]
  \draw (0,0)--(0,1);
  \draw (0,0)--(-.866,-.5);
  \draw (0,0)--(.866, -.5);
   \draw (3/10, 4/5) node{$2 $};
   \draw (-4/5,-1/10) node{$n$};
   \draw (4/5,-1/10) node{$n$};
   \draw (0,-1) node{$D_{2n}$};
   \end{scope}
   \begin{scope}[shift={(3,0)}]
  \draw (0,0)--(0,1);
  \draw (0,0)--(-.866,-.5);
  \draw (0,0)--(.866, -.5);
   \draw (.3, 4/5) node{$2 $};
   \draw (-4/5,-1/10) node{$3$};
   \draw (4/5,-1/10) node{$3$};
   \draw (0,-1) node{$T_{12}$};
   \end{scope}
   \begin{scope}[shift={(6,0)}]
  \draw (0,0)--(0,1);
  \draw (0,0)--(-.866,-.5);
  \draw (0,0)--(.866, -.5);
   \draw (.3, 4/5) node{$2 $};
   \draw (-4/5,-1/10) node{$3$};
   \draw (4/5,-1/10) node{$4$};
   \draw (0,-1) node{$O_{24}$};
   \end{scope}
   \begin{scope}[shift={(9,0)}]
  \draw (0,0)--(0,1);
  \draw (0,0)--(-.866,-.5);
  \draw (0,0)--(.866, -.5);
   \draw (.3, 4/5) node{$2 $};
   \draw (-4/5,-1/10) node{$3$};
   \draw (4/5,-1/10) node{$5$};
   \draw (0,-1) node{$I_{60}$};
   \end{scope}
\end{tikzpicture}
\end{center}
\caption{Local models for the branching locus, with the corresponding isotropy
groups: the cyclic group $C_n$, the dihedral group $D_{2n}$, the tetrahedral group $T_{12}$, 
the octahedral group $O_{24}$, and the icosahedral group $I_{60}$. The subindex denotes the order of the group.
The label on an edge denotes the order of the cyclic isotropy group of points in the interior of the edge.}
\label{Figure:local_models}
\end{figure}

\subsection{Geometric representations.}
Let $\mathcal O^3$ denote an orientable hyperbolic 3-orbifold of finite type, possibly of infinite volume, and non-elementary as Klenian group
(hence the Zariski closure of its holonomy is $\mathrm{PSL}(2,\mathbb C)$). 
In particular it has a compactification 
$\overline{\mathcal O^3}$ that is an orbifold with boundary. Let
$$
\partial \overline{\mathcal O^3}= \partial_1 \overline{\mathcal O^3}\sqcup\cdots \sqcup\partial_k \overline{\mathcal O^3}
$$
denote its decomposition in connected components.

Consider the composition of the holonomy $\mathrm{hol}\colon \pi_1(\mathcal O^3)\to\mathrm{PSL}(2,\mathbb C)$ 
with the principal representation $\tau\colon \mathrm{PSL}(2,\mathbb C)\to G$. As we assume that $\mathcal{O}^3$ is non-elementary,
by Lemma~\ref{Lemma:PrincipalGood} and a standard argument on the Zariski closure, we have:

\begin{Remark}
\label{Remark:good}
The representation
 $\rho=\tau\circ \mathrm{hol}\in R(\mathcal O^3,G)$
 is good. 
\end{Remark}

The following is part of Theorem~\ref{thm:main}.

\begin{Theorem}
\label{Theorem:dimcanonical}
The character of  $\rho=\tau\circ \mathrm{hol}$ is a smooth point of $X(\mathcal O^3,G)$. Furthermore 
\begin{align*}
\dim_{[\rho]} X(\mathcal O^3,G)= &
-\frac{1}{2} \sum_{i=1}^k\widetilde\chi(\partial_i \overline{\mathcal O^3}, \mathrm{Ad}\rho)+  
\sum_{i=1}^k  \dim \mathfrak g^{\rho(\pi_1 ( \partial_i \overline{\mathcal O^3} ) )}
 \\
 =& \frac{1}{2}\sum_{i=1}^k  \dim_{[\rho\vert\partial_i]} X(\partial_i\overline{ \mathcal O^3},G) 
 = \frac{1}{2} \dim_{[\rho\vert\partial]} X( \partial \overline{ \mathcal O^3},G)  .
\end{align*}
\end{Theorem}

\begin{proof}
We follow the proof for the manifold case in  \cite{MFP12} for representations 
in $\mathrm{PSL}(n,\mathbb{C})$, 
or \cite{KapovichBook} in $\mathrm{PSL}(2,\mathbb{C})$.
In particular we use Selberg's lemma: there exists $\mathcal O_0^3\to\mathcal O^3$ a finite regular covering which is a manifold.
 
 The first step is to prove that the inclusion of $\partial \overline{\mathcal O^3}$ in $ \overline{\mathcal O^3}$ induces an injection in cohomology
 $$
 H^1(\overline{\mathcal O^3}, \mathrm{Ad}\rho)\hookrightarrow H^1(\partial \overline{\mathcal O^3}, \mathrm{Ad}\rho).
 $$
When $\mathcal O^3$ is a manifold, then this is proved in  \cite{MFP12} using a theorem of vanishing in $L^2$-cohomology 
(as, using de Rham cohomology and a metric on the bundle and differential forms, every element in the 
kernel is represented by an $L^2$-form), by using the decomposition~\eqref{eqn:exponents} of $ \mathrm{Ad}\circ\tau $.
For orbifolds in general,  using  Selberg's lemma, the inclusion follows from the manifold case 
and Proposition~\ref{prop:regularcovering}.
 
The next step is to prove that 
 $$
 \dim H^1(\overline{\mathcal O^3}, \mathrm{Ad}\rho)= \frac{1}{2} \dim H^1(\partial \overline{\mathcal O^3}, \mathrm{Ad}\rho).
 $$
As in the manifold case, this follows from Poincar\'e duality, Proposition~\ref{Prop:duality}, and the long exact sequence
of the pair $(\overline{\mathcal O^3}, \partial \overline{\mathcal O^3})$. Namely, if the morphisms in the long exact sequence in cohomology are
 $$
  H^1(\overline{\mathcal O^3}, \mathrm{Ad}\rho)\overset{i^*}\to H^1(\partial \overline{\mathcal O^3}, \mathrm{Ad}\rho)
  \overset\beta \to  H^2(\overline{\mathcal O^3}, \partial \overline{\mathcal O^3}; \mathrm{Ad}\rho)
 $$
 then
 $$
 \langle i^*(x) , y\rangle_{\partial \overline{\mathcal O^3}} = \langle x  , \beta(y)  \rangle_{\overline{\mathcal O^3}} 
 \qquad \forall x\in H^1(\overline{\mathcal O^3}, \mathrm{Ad}\rho),
  \forall y\in H^1(\partial \overline{\mathcal O^3}, \mathrm{Ad}\rho),
 $$
 where $\langle, \rangle$ denotes the perfect pairing in Proposition~\ref{Prop:duality}, and the pairing
 on $\partial \overline{\mathcal O^3}$ is the sum of pairings on each connected component.
 As the pairings are nondegenerate and $\ker\beta=\operatorname{im} i^*$, the claim on the dimension follows
 from a linear algebra argument.

Finally, as  $ H^1(\overline{\mathcal O^3}, \mathrm{Ad}\rho)\cong H^1(\Gamma, \mathrm{Ad}\rho)$ by Proposition~\ref{Prop:duality}, with
$\Gamma=\pi_1({\mathcal O^3})$,
we use Goldman's obstruction theory to integrability to prove that $[\rho]$ is a smooth point of the character variety of 
local dimension $ H^1(\overline{\mathcal O^3}, \mathrm{Ad}\rho)$. 
For $\Gamma_0<\Gamma$ a torsion-free subgroup, any infinitesimal deformation
in $H^1(\Gamma, \mathrm{Ad}\rho)$ yields a $\Gamma/\Gamma_0$-equivariant infinitesimal deformation of $\Gamma_0$, 
and we may apply the same argument as in Theorem~\ref{Thm:EuclideanRed}. Alternatively as $X(\Gamma_0, G)$ is analytically smooth at 
the character of the restriction of $\rho$, 
by Cartan's linearization \cite[Lemma 1]{Cartan} there exist local analytic coordinates that linearize the action of $\Gamma/\Gamma_0$
in a neighborhood of the character of $\rho\vert_{\Gamma_0}$ in $X(\Gamma_0, G)$. 
Hence the fixed point set of $\Gamma/\Gamma_0$ is a smooth subvariety
and, as $\rho\vert_{\Gamma_0}$ is good, this fixed point set can be locally identified with  $X(\Gamma, G)$. 
\end{proof}

\begin{Definition}
\label{Def:canonical}
The component of Theorem~\ref{Theorem:dimcanonical}
is called the  \emph{canonical} or \emph{distinguished} component.
\end{Definition}

To complete the proof of Theorem~\ref{thm:main} we need to show that $X(\mathcal O^3,G)\to X(\partial \mathcal O^3,G)$
is locally an injection. 
This does not follow directly from the injection $H^1(\mathcal{O}^3,\mathrm{Ad}\rho)\to H^1(\partial \mathcal{O}^3,\mathrm{Ad}\rho)$,
as $X(\partial \mathcal O^3,G)$ can be non-smooth because of rank 2 cusps.
(Rank one cusps are part of components with negative Euler characteristic and we discuss only rank 2 cusps.)
Consider  a nonsingular (manifold) cusp, $T^2\times [0,+\infty)$ equipped with the warped product  metric  
$e^{-2 t} g_{T^2}+ d t^2$, where $g_{T^2}$ is a flat metric on the torus and $t\in [0,+\infty)$ is 
the Busemann function coordinate. 
Set $\rho=\tau\circ\operatorname{hol}\colon\pi_1( T^2)\to G$.
Using de Rham cohomology, we may talk about $L^2$-forms on $T^2\times [0,+\infty)$.

\begin{Lemma}
\label{lem:kerL2}
For any non-contractible loop $l\subset T^2$, the restriction induces a surjection
$$
 H^1(T^2\times [0,+\infty), \mathrm{Ad}\rho )\to H^1(l,  \mathrm{Ad}\rho)
$$
whose kernel consists of the cohomology classes represented by $L^2$-forms. 
\end{Lemma}

\begin{proof}
Since  $  \mathrm{Ad}\circ \tau=\oplus_\alpha \operatorname{Sym}^{2 d_{\alpha}}$ by \eqref{eqn:exponents}
and $  \mathrm{Ad}\circ\operatorname{Sym}^{n-1}=\oplus_{i=1}^{n-1} \operatorname{Sym}^{2i}$,
it  suffices to prove the lemma for $\tau=\operatorname{Sym}^{n-1}$ and $\mathfrak{g}=\mathfrak{sl}(n,\mathbb{C})$.
In 
 \cite{MFPABP}
 precisely the situation when $\tau=\operatorname{Sym}^{n-1}$ is discussed. 
  In particular
by \cite[Lemma~3.3]{MFPABP} the image of the map induced by 
the inclusion ${\mathfrak{sl}(n,\mathbb{C})}^{\pi_1(T^2)}\subset {\mathfrak{sl}(n,\mathbb{C})}$, 
\begin{equation}
 \label{eqn:l2imagetorus}
  H^1(T^2, {\mathfrak{sl}(n,\mathbb{C})}^{\pi_1(T^2)})\to  H^1(T^2, \mathfrak{sl}(n,\mathbb{C})_{\mathrm{Ad}\rho  }  )
\end{equation}
is represented by $L^2$-forms.
(Here $\mathfrak{sl}(n,\mathbb{C})_{\mathrm{Ad}\rho  } $ denotes the coefficients on the Lie algebra, usually just denoted by 
${\mathrm{Ad}\rho  } $, whilst on the invariant subspace   ${\mathfrak{sl}(n,\mathbb{C})}^{\pi_1(T^2)}$
the action is trivial by construction.)
Furthermore, either by  \cite[Lemma~3.4]{MFPABP} or by a straightforward computation, the map
$$
 H^1(T^2, \mathfrak{sl}(n,\mathbb{C})_{\mathrm{Ad}\rho})\to H^1(l, \mathfrak{sl}(n,\mathbb{C})_{\mathrm{Ad}\rho}
 )
 $$
is a surjection with kernel precisely the image of \eqref{eqn:l2imagetorus}.
\end{proof}

\begin{Remark}
\label{rem:coinv}
Let $\gamma\in\pi_1(T^2)$ be an element represented by the loop $l$ as in Lemma~\ref{lem:kerL2}. We have a natural isomorphism:
$$
H^1(l, \mathrm{Ad}\rho)\cong  
 H_0(l, \mathrm{Ad}\rho )\cong H_0(\langle\gamma\rangle, \mathrm{Ad}\rho )\cong
 \mathfrak{g}_{\rho( \gamma) }=\mathfrak{g}_{ \rho(\pi_1(T^2))}
$$
\end{Remark}

%

Let $\widetilde G $ be the universal covering of $G$ and let $\chi_1,\dotsc,\chi_r$ denote its fundamental
characters (the characters of the fundamental representations). 
E.g.~for $\widetilde G =\mathrm{SL}(n,\mathbb{C})$ the fundamental characters are the symmetric functions on the eigenvalues.
By a theorem of Steinberg \cite{SteinbergIHES}, 
they define an isomorphism:
\begin{equation}
 \label{eqn:SteinbergMap}
(\chi_1,\dotsc,\chi_r)\colon \widetilde G^{\mathrm{reg}}/G\cong \mathbb{C}^r ,
\end{equation}
where $ \widetilde G^{\mathrm{reg}}$ denotes the set of regular elements in $ \widetilde G$.
When the group $G$ is not simply connected, the characters $(\chi_1,\dotsc,\chi_r)$ still define local functions in a neighborhood of
$\rho(\gamma)$ in $G$, where $\gamma\in\pi_1(T^2)$ is represented by the loop $l$.
We identify the variety of representations of the cyclic group $\langle \gamma\rangle\cong\mathbb Z$
with $G$ via the image of $\gamma$. In particular the differential form $d\chi_i\colon
\mathfrak g\to\mathbb C$ may be viewed as a linear map
$$
d\chi_i\colon
H^1( \langle \gamma\rangle, \mathrm{Ad}\rho)\cong {\mathfrak{g}}_{ \rho(\gamma) } 
\to\mathbb C,
$$
because the characters $\chi_i$ are constant on orbits by conjugation (hence the characters are trivial on coboundaries).
This uses the natural identification between   $H^1( \langle \gamma\rangle, \mathrm{Ad}\rho)$ and the space of coinvariants 
$\mathfrak{g}_{\rho(\gamma)}$ in Remark~\ref{rem:coinv}.

\begin{Lemma}
\label{lem:dchi}
$(d\chi_1, \ldots, d\chi_r)\colon H^1( \langle \gamma\rangle , \mathrm{Ad}\rho)\to\mathbb C^r$ is an isomorphism.
\end{Lemma}

\begin{proof}
We follow the proof of \cite[Corollary~19]{HP2020}. Steinberg has shown that the map \eqref{eqn:SteinbergMap}
has a smooth section, $\mathbb{C}^r\to \widetilde G^{\mathrm{reg}}$ \cite{SteinbergIHES}. 
In particular $(d\chi_1, \ldots, d\chi_r):\mathfrak{g}\to \mathbb{C}^r$ 
is surjective and so is the induced map  
$(d\chi_1, \ldots, d\chi_r):H^1( \langle \gamma\rangle   ,\mathrm{Ad}\rho)\cong {\mathfrak{g}}_{ \rho(\gamma) } \to \mathbb{C}^r$.
Hence the lemma follows from equality of dimensions. 
\end{proof}

Consider $\partial_{\mathrm{hyp}}\overline{\mathcal O^3}$ the union of components of  
$\partial\overline{ \mathcal  O^3}$
that have negative Euler characteristic. In particular, for a manifold $\mathcal O^3=M^3$,
$$
\partial\overline{ M^3}= \partial_{\mathrm{hyp}}\overline{M^3}\sqcup T^2_1\sqcup\dotsb\sqcup T^2_k
$$
where $k$ is the number of rank-2 cusps of $M^3$.

\begin{Proposition}
\label{prop:injmfld}
 When $M^3=\mathcal O^3$ is a manifold, then the restriction to the boundary and the characters of peripheral
 curves of cusps yield
  a local embedding 
$$
   X(M^3, G)\to  X(\partial_{\mathrm{hyp}} \overline{M^3}  , G) \times \mathbb{C}^{r k},
$$
  where  $k$ is the number of rank-2 cusps and the coordinates in $\mathbb{C}^{r k}$ are (locally defined) fundamental characters
  of  chosen peripheral elements, one  for each cusp.
\end{Proposition}

\begin{proof}
The proof of Theorem~\ref{Theorem:dimcanonical}
is based in the vanishing of $L^2$-cohomology, here  we use the fact that a form is $L^2$ on $M^3$ if it is so in the restriction to each end.
Hence using Lemmas~\ref{lem:kerL2} and~\ref{lem:dchi} the kernel of the differential of the map of the proposition is represented
by $L^2$-forms, hence trivial. 
%
%
\end{proof}

The following concludes the proof of Theorem~\ref{thm:main}.

\begin{Corollary}
The restriction to the boundary yields a local embedding
 $X(\mathcal O^3, G)\to  X(\partial \overline{\mathcal O ^3}  , G) $.
\end{Corollary}

\begin{proof}
In the manifold case, we use that the local injection in Proposition~\ref{prop:injmfld} factors through the restriction to 
$X(\partial   \overline{M^3}, G)$:
$$
X(M^3, G)\to X(\partial   \overline{M^3}, G)\to  X(\partial_{\mathrm{hyp}} \overline{M^3}  , G) \times \mathbb{C}^{r k},
$$
Thus the map induced by restriction $X(M^3, G)\to X(\partial   \overline{M^3}, G)$ must be locally an injection.
For an orbifold $\mathcal O^3$,
consider a finite manifold covering $\mathcal O^3_0\to\mathcal O^3$, 
we have the commutative diagram:
$$
\begin{CD}
X(\mathcal O_0^3, G)     @>>>  X(\partial \overline{\mathcal O_0^3}, G)\\
@AAA        @AAA\\
X(\mathcal O^3, G)     @>>>  X(\partial \overline{\mathcal O^3}, G)
\end{CD}
$$
As $\mathcal O_0^3$ is a manifold, we have already shown that 
$X(\mathcal O_0^3,G)\to X(\partial \overline{\mathcal O_0^3}, G)$ 
is locally injective.
Furthermore, 
every character $[\rho']$ in a neighborhood of $[\rho]=[\tau\circ\mathrm{hol}]$ is good,
because being good is an open property in the variety of characters and 
$\tau\circ\mathrm{hol}$ is good by Remark~\ref{Remark:good}. As $\mathcal O_0^3$ is also 
non-elementary, the restriction of $[\rho']$ to $\pi_1(\mathcal O^3_0)  $ is also good, 
and
its centralizer in $G$ is trivial and therefore the possible extension of the restricted representation
$\rho'\vert_{\pi_1(\mathcal O^3_0)}$ to
$\pi_1(\mathcal O^3)$ is unique.
Namely, the restriction map
$X(\mathcal O^3,G)\to X(\mathcal O_0^3,G)$ is also locally injective, and we are done from the commutativity of the diagram.
\end{proof}

A particular case of Proposition~\ref{prop:injmfld} is the following
corollary (proved for representations of 
manifold groups in $\mathrm{SL}(2,\mathbb C)$ in \cite{Bromberg,KapovichBook}, and for representations of manifold groups in  
$\mathrm{SL}(n,\mathbb C)$ in \cite{MFPABP}).

\begin{Corollary}
\label{cor:injmfld}
When $M^3$ is a manifold whose interior has finite volume, chose $\gamma_1,\ldots,\gamma_k $ a nontrivial peripheral element for
each cusp. Let $\chi_1,\dotsc,\chi_r$ denote the fundamental characters of $\widetilde G$. Then
$$
(\chi_{1,\gamma_1},\dotsc, \chi_{r,\gamma_k}) \colon U \subset X(M^3, G)\to\mathbb{C}^{k r}. 
$$
define a local homeomorphism for a neighborhood $U\subset X(M^3, G)$ of $[\tau\circ\mathrm{hol}]$.
\end{Corollary}

Recall that the adjoint or $\operatorname{Sym}^2$ gives an isomorphism 
$\mathrm{PSL}(2,\mathbb C)\cong\mathrm{SO}(3,\mathbb{C})$. By an argument on dimensions (see Table~\ref{table:dimparabolic}):


\begin{Corollary} 
\label{coro:geometricSL3}
Let $\mathcal O^3$ be a compact orientable three-orbifold, with hyperbolic interior of finite volume.
If all components of $\partial \mathcal O^3$ are
homeomorphic to $ S^2(2,2,2,2)$, $ S^2(2,4,4)$ or $S^2(2,3,6)$, then
$$
X_0(\mathcal O^3, \mathrm{SL}(3,\mathbb C) )\cong X_0(\mathcal O^3, \mathrm{SO}(3,\mathbb C) )\cong X_0(\mathcal O^3,  \mathrm{PSL}(2,\mathbb C)),
$$
where $X_0$ denotes the distinguished component.
\end{Corollary}

\subsection{Lower bound of the dimension}

In this subsection we provide a lower bound \`a la Thurston 
(for $\mathrm{PSL}(2,\mathbb{C})$) or \`a la  Falbel-Guilloux (for general $G$), in both cases for manifolds.
For a compact three orbifold $\mathcal{O}^3$, let 
$$
\partial \mathcal O^3= \partial_1 {\mathcal O^3}\sqcup\cdots \sqcup\partial_k {\mathcal O^3}
$$
denote the decomposition in connected components of its boundary.

\begin{Theorem}
\label{Thm:lowerbounddim}
Let $G$ be  semi-simple  $\mathbb C$-algebraic Lie group, $\mathcal O^3$ a
compact, orientable  very good orbifold.
Let $\rho\colon\pi_1(\mathcal O^3)\to G$ be a good representation. Assume that:
\begin{enumerate}[(a)]
\item if  $  \partial_i {\mathcal O^3}   $ is a hyperbolic  boundary component,
then the centralizer of its image is zero-dimensional;
\item if  $  \partial_i {\mathcal O^3}  $ is a Euclidean boundary component, then the restriction $\rho\vert_{\pi_1(\mathcal O^3)}$
is strongly regular.
%
%
\end{enumerate}
Then
\begin{align*}
   \dim_{[\rho]} X_0(\Gamma, G)&\geq 
   \frac{1}{2} 
    \sum_{i=1}^k \dim_{[\rho\vert \partial_i ]} X_0(  \partial_i {\mathcal O^3}    , G) \\
 &=-\frac{1}{2} \sum_{i=1}^k\widetilde\chi(\partial_i {\mathcal O^3}, \mathrm{Ad}\rho)+  
\sum_{i=1}^k  \dim \mathfrak{g}^{\rho(\pi_1 ( \partial_i {\mathcal O^3} ) ))}.
 \end{align*}
\end{Theorem}

\begin{proof}
 
The branching locus $\Sigma$ of the compact orbifold $\mathcal O^3$ is a trivalent graph. 
Chose a finite subset $\Sigma_{0}\subset \Sigma$ as follows: for each component of 
$\Sigma$ that is a circle chose precisely one point, and chose also the trivalent vertices of $\Sigma$. 
So $\Sigma_0$ has the smallest cardinality so that $\Sigma\setminus \Sigma_0$ is a disjoint union of edges 
(open or not, as they can meet $\partial\mathcal{O}^3$).
Consider the orbifold 
$$
\mathcal O_0=\mathcal O^3\setminus \mathcal{N}(\Sigma_0),
$$
ie.~remove an open tubular neighborhood for each point in 
$\Sigma_0$. 
The branching locus of $\mathcal O_ 0$ is a union of (proper) edges (it contains no vertices nor circles).
The boundary components of $\mathcal O_0$ are either the 
boundary components of $\mathcal{O}^3$  or the boundary of a neighborhood of a point in 
$\Sigma_0$ (spherical).
For each Euclidean or spherical boundary component $\partial_i\mathcal O_0$, chose $k_i$ disjoint embedded loops
$\gamma_{i,1},\ldots,\gamma_{i,k_i}$ in $\mathcal O_0$
based at $\partial_i\mathcal O_0$ so that 
$$
{\mathfrak{g}}^{\rho(\pi_1( \partial{\mathcal{O}_0}))} \cap 
{\mathfrak{g}}^{ \rho(\gamma_{i,1}) } \cap\dotsb 
\cap   {\mathfrak{g}}^{ \rho(\gamma_{i,{k_i}}) }     =0.
$$
Namely, the centralizer in  $\mathfrak{g} $ of the image of the group generated by 
$\pi_1( \partial_i{\mathcal{O}_0})$ and $\gamma_{i,1},\ldots,\gamma_{i,k_i}$ is trivial.
The next step in the  construction is to remove an open tubular neighborhood of the $\gamma_{i,j}$:
$$
\mathcal O_1=\mathcal O_0\setminus \bigcup_{i,j} \mathcal{N}(\gamma_{i,j}).
$$
Finally remove an open tubular neighborhood of the branching locus of $\mathcal O_1$ (which is a union of edges):
$$
\mathcal O_2=\mathcal O_1\setminus \mathcal{N}(\Sigma_{\mathcal O_ 1}).
$$
In particular $\mathcal O_2$ is a manifold.

The connected 3-manifold with non-empty boundary $\mathcal O_2$ retracts to a 2-dimensional CW-complex with a single $0$-cell.
This gives a presentation of $\pi_1(\mathcal O_2)$ from this CW-complex 
(1-cells are generators and 2-cells relations).
Each generator contributes with a copy of $G$ in the variety of representations, 
and each relation decreases at most $\dim G$ the dimension;
this gives the standard bound:
\begin{multline}
\label{eqn:BoundManifold}
\dim  R(\mathcal O_2, G)\geq (1- \chi(\mathcal O_2))\dim G
=
\left(1- \frac{\chi(\partial \mathcal O_2)}2\right)\dim G
\\
=
-\frac12\widetilde\chi(\partial \mathcal O_2,\mathrm{Ad}\rho))+\dim G ,
\end{multline}
using that $ \partial \mathcal O_2 $ is a surface.

Next we compute lower bounds for the dimension of the variety of representations of the
orbifolds $\mathcal O_1$, $\mathcal O_0$, and finally for  $\mathcal O$.
We shall use the following key lemma of Falbel and Guilloux, in fact it is a local version of Proposition~1 in \cite{FalbelGuilloux}:

\begin{Lemma}[\cite{FalbelGuilloux}]
\label{Lemma:FG}
Let $W$ be a smooth complex (analytic) variety and $W'\subset W$ a smooth subvariety.
Let $X$ be a complex variety with an analytic map $f\colon X\to W$.
For $p\in X$ with $f(p)\in X$, there exists a neighborhood $ U\subset X$ of $p$
such that 
$$\operatorname{codim}_p(f^{-1}(W')\cap U, X)\leq \operatorname{codim}_{f(p)}(W',W).$$
\end{Lemma}

\begin{proof}
By the implicit function theorem, there exists a neighborhood $ V\subset W$ of $f(p)$ and an analytic map
$F\colon V\to\mathbb{C}^k$ such that $V\cap W'= F^{-1}(0)$, where $k=\operatorname{codim}_{f(p)}(W',W)$.
Then, for a  neighborhood of $p$, $U\subset f^{-1}(V) \subset X$,   
$$
f^{-1}(W')\cap U= (F\circ f)^{-1}(0)\cap U,
$$
and the estimate follows,
because $f^{-1}(W')\cap U$ is a fiber of a map  $F\circ f\colon U\to \mathbb C^k$.
\end{proof}

To get $\mathcal O_1$ from $\mathcal O_2$, we add the  edges of ${\Sigma}_{\mathcal O_1}$. 
Topologically, we add 2-handles whose co-core is an edge  of ${\Sigma}_{\mathcal O_1}$, hence ``2-handles with singular co-core''. 
For each singular edge $e$ of $ {\Sigma}_{\mathcal O_1} $ consider a meridian $\mu\in\pi_1(\mathcal O_2)$, which is represented by the 
attaching circle of the 2-handle, and has finite order in  
$\pi_1(\mathcal O_1)$.
By homogeneity, the
conjugation orbit $G\cdot\rho(\mu)= O(\rho(\mu))\subset G$ is a smooth analytic 
submanifold of $G$, of codimension
equal the dimension of the centralizer $\dim(\mathfrak{g}^{\rho(\mu)})$.
Thus, by Lemma~\ref{Lemma:FG}, when we add the edge $e$ to $\mathcal O_ 2$, the dimension
of the variety of representations decreases at most by $\dim(\mathfrak{g}^{\rho(\mu)})$:
\begin{multline}
\label{eqn:codimO1O2}
\dim  R(\mathcal O_2,G)-\dim  R(\mathcal O_1,G)=
\operatorname{ codim}(  R(\mathcal O_1,G),  R(\mathcal O_2,G))
\leq 
\sum_\mu \operatorname{ codim}( O(\rho(\mu)), G ) \\ =
\sum_\mu \dim(\mathfrak{g}^{\rho(\mu)}).
\end{multline}
Furthermore, adding the 2-handle to $\mathcal O_2$ corresponding to $\mu$ 
means replacing an annulus in $\partial \mathcal O^2$ by two cone points with cyclic stabilizer generated by $\mu$.
Hence, the twisted Euler characteristic of the boundary increases by
$2\dim(\mathfrak{g}^{\rho(\mu)})$. Counting the contribution of all edges of 
${\Sigma}_{\mathcal O_1}$:
\begin{equation}
\label{eqn:chiO1O2}
 \widetilde\chi(\partial \mathcal O_1,\mathrm{Ad}\rho)
 =\widetilde\chi(\partial \mathcal O_2,\mathrm{Ad}\rho)+2 \sum_\mu \dim(\mathfrak{g}^{\rho(\mu)}).
\end{equation}
Hence from~\eqref{eqn:BoundManifold}, \eqref{eqn:codimO1O2}, and~\eqref{eqn:chiO1O2}:
\begin{equation}
\label{eqn:BoundO1}
\dim R(\mathcal O_1,G)\geq -\frac{1}{2}\widetilde\chi(\partial \mathcal O_1,\mathrm{Ad}\rho)  +\dim G.
\end{equation}

To get $\mathcal O_0$ from  $\mathcal O_1$, we add the neighborhoods of the loops $\gamma_{i,j}$,
i.e.~we add 2-handles (without singular co-core).
Consider the $i$-th boundary component $\partial_i\mathcal{O}_1$. 
The corresponding boundary component 
$\partial_i\mathcal{O}_0$ is obtained from $\partial_i\mathcal{O}_1$ by 
the surgery corresponding to the addition of
$k_i$ 2-handles.
Namely  $k_i$ annuli in $\partial_i\mathcal O_1$  
corresponds to $2 k_i$ smooth disks in $\partial_i\mathcal O_0$.
Thus
\begin{equation}
 \label{eqn:partial01}
 \widetilde\chi(\partial_i\mathcal{O}_0,\mathrm{Ad}\rho)=
\widetilde\chi(\partial_i\mathcal{O}_1,\mathrm{Ad}\rho)+2 k_i \dim G .
\end{equation} 
On the other hand, following the idea of \cite{FalbelGuilloux}, we apply 
Lemma~\ref{Lemma:FG} to $\partial_i\mathcal O_1$ and the free product
$\pi_1(\partial_i\mathcal O_0)*\langle \gamma_{i,1}\rangle *\cdots * \langle \gamma_{i,k_i}\rangle$, 
that is, the group obtained by filling the 
$k_i$ meridians of the surface $\partial_i\mathcal O_1$:
\begin{equation}
 \label{eqn:loops}
W'=R(\pi_1(\partial_i\mathcal O_0)*\langle \gamma_{i,1}\rangle *\cdots * \langle \gamma_{i,k_i}\rangle, G)
\subset
R(\partial_i\mathcal O_1, G)=W. 
\end{equation}
By Proposition~\ref{Prop:dimR}, and since we assume that the centralizer of 
$\rho(\pi_1(\partial_i\mathcal O_1) )$  in $\mathfrak{g}$ is trivial, $W$ is smooth and
\begin{equation}
\label{eqn:dimW}
\dim W=\dim R(\partial_i\mathcal O_1, G) =-  \widetilde\chi( \partial_i\mathcal O_1 ,\mathrm{Ad}\rho ) + \dim G 
 .
\end{equation}
By  Proposition~\ref{Prop:dimR} when $\partial_i\mathcal O_0$ is hyperbolic,
Theorem~\ref{Thm:EuclideanRed} when Euclidean, and Remark~\ref{Rem:sphericalSmooth}
when spherical,
$ R(\pi_1(\partial_i\mathcal O_0), G)$ is smooth and
\begin{equation}
\label{eqn:dimRO0}
\dim R(\partial_i\mathcal O_0, G) =
-   \widetilde\chi( \partial_i\mathcal O_0 ,\mathrm{Ad}\rho ) +\dim G+
\dim \mathfrak g^{\rho (\pi_1(\partial_i\mathcal O_0)) }. 
\end{equation}
As the variety of representations of a free product is the Cartesian product
of varieties or representations of its factors, $W'$ is smooth and:
\begin{equation}
\label{eqn:dimRprod}
\dim W'=\dim R(\pi_1(\partial_i\mathcal O_0)*\langle \gamma_{i,1}\rangle *\dotsb * \langle \gamma_{i,k_i}\rangle,  G) 
=\dim R(\partial_i\mathcal O_0, G) + k_i\dim G.
\end{equation}
From \eqref{eqn:dimRO0} and  \eqref{eqn:dimRprod}:
\begin{equation}
 \label{eqn:dimW'}
 \dim W'= - \widetilde\chi( \partial_i\mathcal O_0 ,\mathrm{Ad}\rho )
 + (k_i +1)\dim G
  +\dim \mathfrak g^{\rho (\pi_1(\partial_i\mathcal O_0))}
 . 
\end{equation}
It follows from  \eqref{eqn:dimW} and \eqref{eqn:dimW'} that  the codimension of the inclusion \eqref{eqn:loops} is
\begin{equation*}
\operatorname{codim}(W',W)= \dim W-\dim W'
= \widetilde\chi( \partial_i\mathcal O_0 ,\mathrm{Ad}\rho )-
  \widetilde\chi( \partial_i\mathcal O_1 ,\mathrm{Ad}\rho )
  -k_i \dim G
  -\dim \mathfrak g^{\rho (\pi_1(\partial_i\mathcal O_0))},
\end{equation*}
that combined with \eqref{eqn:partial01} yields
\begin{equation}
\label{eqn:codimWW'}
 \operatorname{codim}(W',W)= k_i\dim G- \dim \mathfrak{g}^{\rho (\pi_1(\partial_i\mathcal O_0))}.
\end{equation}
So when we apply  Lemma~\ref{Lemma:FG} we get from the contribution  \eqref{eqn:codimWW'} of each boundary component:
\begin{equation}\label{eqn:codimR0R1}
\operatorname{codim}(R(\mathcal O_0, G) , R(\mathcal O_1, G))\leq \sum_i 
\Big( k_i\dim G- \dim \mathfrak{g}^{\rho (\pi_1(\partial_i\mathcal O_0))}  \Big) .
\end{equation}
Thus, with \eqref{eqn:BoundO1} and \eqref{eqn:partial01}, \eqref{eqn:codimR0R1} becomes:
\begin{align*}
 \dim R(\mathcal O_0, G)& \geq  \dim R(\mathcal O_1, G)-
 \sum_i \Big( k_i\dim G- \dim \mathfrak{g}^{\rho (\pi_1(\partial_i\mathcal O_0))}  \Big)\\
 & \geq \sum_i \Big(  -\frac{1}{2} \widetilde\chi(\partial_i\mathcal O_1,\mathrm{Ad}\rho)  )  
 - k_i\dim G +\dim \mathfrak{g}^{\rho (\pi_1(\partial_i\mathcal O_0)}\Big) +\dim G \\
 & =  \sum_i \Big(  -\frac{1}{2} \widetilde\chi(\partial_i\mathcal O_0,\mathrm{Ad}\rho)  )+ 
 \dim \mathfrak{g}^{\rho (\pi_1(\partial_i\mathcal O_0)}\Big) +\dim G.
\end{align*}
Finally, for each spherical boundary component  $\partial_i\mathcal O_0$, by 
\eqref{eqn:spherical}
$$
-\frac{1}{2} \widetilde\chi(\partial_i\mathcal O_0,\mathrm{Ad}\rho)  )+ 
\dim \mathfrak{g}^{\rho (\pi_1(\partial_i\mathcal O_0)}=0.
$$
Hence we can get rid of the contribution of 
the neighborhoods of vertices and get the initial orbifold~$\mathcal O^3$:
\begin{align*}
\dim R(\mathcal O^3, G)
&\geq 
\sum_i \Big(  -\frac{1}{2} \widetilde\chi(\partial_i\mathcal O^3,\mathrm{Ad}\rho)  )+ 
 \dim \mathfrak{g}^{\rho (\pi_1(\partial_i\mathcal O^3)}\Big) +\dim G
\\ &=
 \dim  X(\partial \mathcal O^3, G)+\dim G
 .
\end{align*}
As $\dim X(\mathcal O^3, G)=\dim {R}(\mathcal O^3, G)-\dim G$, 
because the representation is good, this concludes
the proof of the theorem.
\end{proof}

\subsection{Dimension growth}

Let $M^3$ be an orientable hyperbolic 3-manifold of finite volume with $k\geq 1$ cusps.
As $M^3$ is a manifold,
its holonomy lifts to $\mathrm{SL}(2,\mathbb C)$.  

Let $X_0(M,\mathrm{SL}(n,\mathbb C))$ be the    \emph{canonical} or \emph{distinguished component}
of $X(M,\mathrm{SL}(n,\mathbb C))$ (Definition~\ref{Def:canonical}), i.e.~the component  that
contains the composition of a lift of the holonomy with $\operatorname{Sym}^{n-1}$.
By Theorem~\ref{Theorem:dimcanonical}, 
$$
\dim X_0(M,\mathrm{SL}(n,\mathbb C)) = (n-1) k,
$$
where $k$ is the number of cusps. 
This linear growth differs from the quadratic growth of
2-orbifolds in Section~\ref{Section:H2O}. Those 2-orbifolds may appear as basis of Seifert fibered Dehn filling and yield 
components in the variety of characters of higher dimension.
Next we discuss two examples, the figure eight knot and the Whitehead link exteriors. The following
is Proposition~\ref{Prop:fig8intro} from the introduction.

\begin{Proposition}
 \label{Prop:fig8}
Let $\Gamma$ be the fundamental group of the exterior of the figure eight knot. Besides the canonical 
component (that has dimension $n-1$), for large $n$  
$X(\Gamma, \mathrm{SL}(n,\mathbb C))$ has at least $3$ components
that contain irreducible representations, whose dimension grow respectively as ${n^2}/{12}$, 
${n^2}/{20}$ and ${n^2}/{42}$.
\end{Proposition}

\begin{proof}
Let $K_{p/q}$ denote the Dehn surgery on the figure eight knot with coefficients $p/q$.
There are 6 Dehn fillings on the figure eight knot that yield small Seifert manifolds \cite{Gordon,MartelliPetronio}: 
$K_{\pm 3}$  are small Seifert manifolds fibered over the 2-orbifold  $\mathcal O ^2_3=S^2(3,3,4)$,
$K_{\pm 2}$ over $\mathcal O ^2_2=S^2(2,4,5)$, and $K_{\pm 1}$ over $\mathcal O ^2_1=S^2(2,3,7)$.
They induce representations in $\mathrm{SL}(2,\mathbb{R})$
of the filled 3-manifolds $K_p$, that map the fiber to $-\operatorname{Id}$ and induce the holonomy representation 
of the (unique) hyperbolic structure on $\mathcal O ^2_p$. 
The composition with $\operatorname{Sym}^{n-1}$ is a representation that maps the fiber to $(-1)^{n-1}\operatorname{Id}$ and 
it induces a representation from 
$\pi_1(\mathcal O^2_p)$ to $\mathrm{PSL}(n,\mathbb R)$ in the  Hitchin component. 
Therefore the complexification of the Hitchin component of $\mathcal{\mathcal O}^2_p$ yields a component 
of $X(K_{p},\mathrm{SL}(n,\mathbb C))$
of the same dimension, and hence a subvariety of $X(\Gamma, \mathrm{SL}(n,\mathbb C))$.
Let $X_p\subset X(M,\mathrm{SL}(n,\mathbb C))$ be the component of $X(\Gamma, \mathrm{SL}(n,\mathbb C))$ that contains the 
subvariety induced by 
$X(K_{p},\mathrm{SL}(n,\mathbb C))$. 
Next we estimate the dimension of $X_1$, $X_2$ and $X_3$, 
which allow to distinguish them, but does not allow to distinguish 
the component induced by $X_p$ from $X_{-p}$, as the dimension estimates are the same.

The lower bound on the dimension of $X_p$
is given by Proposition~\ref{prop:growthHitchin}:
$$
\dim X_p\geq -\chi(\mathcal O^2_p)n^2-c(\mathcal O^2_p).
$$
Aiming to find an upper bound of $\dim X_p$,
we bound above the dimension of its Zariski tangent space, which is isomorphic to 
$H^1(\Gamma,\mathrm{Ad}\rho)\cong H^1(M,\mathrm{Ad}\rho)$ for $M$ the (compact) exterior 
of the figure eight knot, $M=S^3\setminus \mathcal{N}(K)$. In particular $\partial M\cong T^2$.
Here $\rho=\operatorname{Sym}^{n-1}\circ\rho_0$, for 
$\rho_0\colon \Gamma\to \mathrm{SL}(2,\mathbb R)$ the representation that factors through
$K_p$.

First we bound $\dim H^1(\partial M,\mathrm{Ad}\rho)$. Notice that since $\chi(T^2)=0$,  
$$
\dim  H^1(\partial M,\mathrm{Ad}\rho)=
2\dim  H^0(\partial M,\mathrm{Ad}\rho)=2\dim \mathfrak{g}^{\mathrm{Ad}\rho(\pi_1(\partial M))}.
$$
To compute
$\dim \mathfrak{g}^{\mathrm{Ad}\rho(\pi_1(\partial M))}$,
one may check explicitly that the image of $\rho_0(\pi_1( \partial M))$ contains an  element of $\mathrm{SL}(2,\mathbb R)$
of infinite order (one may use for instance the A-polynomial). As the symmetric power of a hyperbolic matrix 
in $\mathrm{SL}(2,\mathbb C)$  is regular, 
$\rho(\pi_1( \partial M))$ contains regular elements and 
$\dim \mathfrak{g}^{\mathrm{Ad}\rho(\pi_1(\partial M))}= \operatorname{ rank} G= n-1 $, Thus 
$\dim  H^1(\partial M,\mathrm{Ad}\rho)=2 (n-1)$.

Let $$
i^*\colon H^1(M,  \mathrm{Ad}\rho)\to H^1(\partial M,\mathrm{Ad}\rho)
$$
be the morphism induced by inclusion. Using the long exact sequence of the pair ($M, \partial M)$ and Poincar\' e duality  as in 
the proof of Theorem~\ref{Theorem:dimcanonical}, its rank is 
\begin{equation}
\label{eqn:rank}
 \operatorname{rank} ( i^*)=n-1.
\end{equation}
Furthermore, by Mayer-Vietoris exact sequence  applied to the decomposition 
$K_p=M\cup_{\partial M} (D^2\times S^1)$, we get that 
\begin{equation}
\label{eqn:ker}
\dim  \ker i^*\leq \dim H^1(K_p, \mathrm{Ad}\rho).
\end{equation}

Next we need a lemma:

\begin{Lemma}\label{lemma:profSF}
 The projection $\pi_1(K_p)\to\pi_1(\mathcal{O}^2_p)$ induces an isomorphism
 \begin{equation}
 \label{eqn:isobasis}
H^1(\pi_1(K_p), \mathrm{Ad}\rho)\cong H^1(\pi_1(\mathcal{O}^2_p), \mathrm{Ad}\rho). 
\end{equation}
\end{Lemma}

Assuming the lemma, we conclude the proof of Proposition~\ref{Prop:fig8}.
Putting together \eqref{eqn:rank}, \eqref{eqn:ker} and \eqref{eqn:isobasis}:
$$
\dim H^1(M, \mathrm{Ad}\rho)\leq \dim H^1(\mathcal O^2_p,  \mathrm{Ad}\rho)+ (n-1).
$$
It follows that
$$
-n^2\chi(\mathcal O^2_p)- c(\mathcal O^2_p)\leq     \dim X_p    \leq-n^2\chi(\mathcal O^2_p)+ c(\mathcal O^2_p)+ n-1.
$$
For large $n$ this allows to distinguish three components, according to the different values for $\chi(\mathcal O^2_p)$.
\end{proof}

\begin{proof}[Proof of Lemma~\ref{lemma:profSF}]
 We use the central exact sequence
$$
1\to \mathcal Z(\pi_1(K_p)) \to  \pi_1(K_p)\to\pi_1(\mathcal{O}^2_p)\to 1,
$$
where the center $\mathcal Z(\pi_1(K_p))\cong \mathbb Z$
is the group generated by the fiber.  
We prove the lemma, with
crossed morphisms or derivations, as in
Subsection~\ref{subsection:ZariskiTS}.
Since $\pi_1(K_p)\to\pi_1(\mathcal O^2_p)$ is a surjection, we have an injection of spaces of derivations
\begin{equation}
\label{eqn:Z1iso}
 Z^1(\pi_1(\mathcal O^2_p),\mathfrak{g})\hookrightarrow Z^1(\pi_1(K_p),\mathfrak{g}).
\end{equation}
We prove surjectivity  of the map in~\eqref{eqn:Z1iso}: 
let $d\in Z^1(\pi_1(K_p),\mathfrak{g})$ be a derivation (namely
a map
$d\colon \pi_1( K_p)\to\mathfrak{g}$
 satisfying 
$d(\gamma_1\gamma_2)=d(\gamma_1)+\mathrm{Ad}_{\rho(\gamma_1)}d(\gamma_2)$, $ \forall\gamma_1,\gamma_2\in \pi_1(K_p)$).
Let $t\in \pi_1(K_p)$ be a generator of the center $\mathcal Z(\pi_1(K_p))\cong \mathbb Z$;
from the relation
$$
t\gamma t^{-1}= \gamma,\qquad\forall\gamma\in\pi_1( K_p),
$$
we deduce
$$
d (t)+\mathrm{Ad}_{\rho(t)} d(\gamma)-\mathrm{Ad}_{\rho(t\gamma t^{-1})} d(t)
= d(\gamma),\qquad\forall\gamma\in\pi_1( K_p).
$$
Since $\rho(t)=\pm\operatorname{Id}$,
$\mathrm{Ad}_{\rho(t)}$ is the identity on $\mathfrak{g}$, hence
$$
(1-\mathrm{Ad}_{\rho(\gamma)}) d (t )=0, \qquad\forall\gamma\in\pi_1( K_p).
$$
This equality implies that $d(t)=0$, because  $\mathfrak{g}^{\rho(\pi_1(K_p))}=0$
by irreducibility of  $\rho$.
This proves that every crossed morphism of $\pi_1( K_p)$ factors through a 
crossed morphism of $\pi_1(\mathcal{O}^2)$,
hence \eqref{eqn:Z1iso} is surjective.
This isomorphism between cocycle spaces induces an isomorphism of coboundary space and
this proves the lemma.
\end{proof}

\begin{Remark}
 For small values of $n$, the arguments are not useful,  but for $n=3$
 the canonical component has dimension 2, and, 
according to Table~\ref{Table:dimHitchin334}, $(\pm 3)$-Dehn fillings also provide two 
subvarieties of dimension 2. It is proved in \cite{FGKRT} and \cite{HMP} that the canonical component and the 
subvarieties induced  by the $(\pm 3)$-Dehn filling 
are precisely the three components of
$X(\Gamma, \mathrm{SL}(3,\mathbb C))$ that  contain irreducible representations. 
\end{Remark}

\begin{Remark}
At the moment, we do not know whether for large $n$ the component that contains  representations that factor through $K_p$ 
is the same as the component corresponding to $K_{-p}$, for $p=1,2,3$. To distinguish them would require a better 
upper bound of the dimension of their Zariski tangent space. 
This would allow also to distinguish components corresponding to Galois conjugates  (5-th or 7-th roots of unity).
\end{Remark}

\begin{Remark}
 An analog of Proposition~\ref{Prop:fig8} can be obtained for the variety of characters of the figure eight knot in 
 $\mathrm{Sp}(2m,\mathbb{C})$, for large $m$.
\end{Remark}


Twist knots have Dehn fillings that are small Seifert manifolds \cite{BW}, and therefore there is a similar behavior. 
Twist knots are obtained by  Dehn filling on one component of the Whitehead link, that we discuss next.

\begin{Proposition}
  \label{Prop:Whitehead}
Let $\Gamma$ be the fundamental group of the exterior of the Whitehead link. Besides the canonical 
component (that has dimension $2n-2$), for large $n$, $X(\Gamma, \mathrm{SL}(n,\mathbb C))$ has at least $3$ components
that contain irreducible characters, whose dimension grow respectively as ${n^2}/{3}$, 
${n^2}/{4}$ and ${n^2}/{6}$.
\end{Proposition}

The proof of Proposition~\ref{Prop:Whitehead} is the same as Proposition~\ref{Prop:fig8}. Here there are partial Dehn fillings that yield Seifert fibered manifolds,
with basis a disc with two cone points: the $(-3)$-filling fibers over $ D^2(3,3)$, 
the $(-2)$-filling  over $D^2(2,4)$, and the $(-1)$-filling is fibered over $D^2(2,3)$ \cite[Table~A.1]{MartelliPetronio}.
This yields the three rational orbifold Euler characteristics, $1/3$, $1/4$ and $1/6$.

As the basis orbifolds are rather simple, it is also easy to state the dimension of their Hitchin component.
The dimension of the Hitchin component of 
$ D^2(3,3)$ is
$$
\begin{cases}
 \frac{n^2}{3}+1 & \textrm{for }n\equiv 0\mod 3 \\
 \frac{n^2-1}{3} & \textrm{for }n\not\equiv 0\mod 3
\end{cases}
$$
The dimension of the Hitchin component of 
$ D^2(2,4)$ is
$$
\begin{cases}
 \frac{n^2}{4}+1 & \textrm{for }n\equiv 0\mod 4 \\
 \frac{n^2-1}{4} & \textrm{for }n \equiv 1 \mod 2 \\
 \frac{n^2}{4} & \textrm{for }n \equiv 2 \mod 4 
\end{cases}
$$
The dimension of the Hitchin component of 
$ D^2(2,3)$ is
$$
\begin{cases}
 \frac{n^2}{6}+1 & \textrm{for }n\equiv 0\mod 6 \\
 \frac{n^2-1}{6} & \textrm{for }n \equiv \pm  1 \mod 6 \\
 \frac{n^2+2}{6} & \textrm{for }n \equiv 2 \mod 6 \\
 \frac{n^2+3}{6} & \textrm{for }n \equiv 3 \mod 6 
\end{cases}
$$

\begin{Remark}
When we apply these computation to $\mathrm{SL}(3,\mathbb C)$, for $ D^2(2,4)$ and  $D^2(2,3)$ it
yields dimension $2$, but for $ D^2(3,3)$ it has a component of dimension $4$, the same as the dimension
of the canonical component.
It has been proved by Guilloux and Will \cite{GuillouxWill} that this is in fact a whole component of $X(\Gamma,  \mathrm{SL}(3,\mathbb C))$.
\end{Remark}

For Montesinos links, the same arguments yield the following:

\begin{Proposition}
 Let $L\subset S^3$ be a Montesinos link. Assume that either it has at least 4 tangles, 
 or that it has three tangles (eg.~a pretzel link)
 with indices $\alpha_1$, $\alpha_2$, and $\alpha_3$ satisfying 
 $\frac{1}{\alpha_1}+\frac{1}{\alpha_2}+ \frac{1}{\alpha_3}\leq 1$. Then
 $\dim X(S^3\setminus L,\mathrm{SL}(n,\mathbb C))$ grows quadratically with $n$.
\end{Proposition}

\section{Representations in $\mathbf{SL}(3,\mathbb C)$}
\label{Section:SL3}

In this section we give explicit computations of some varieties of characters in $\mathrm{SL}(3,\mathbb C)$. 
We compute varieties of characters of groups generated by two elements, using
the description of $X(F_ 2,\mathrm{SL}(3,\mathbb{C}))$ due to Lawton~\cite{LawtonJA}, where
$F_2=\langle a, b\mid\rangle$ is the free group of rank two.

Setting coordinates 
\begin{equation}
 \label{eqn:coordinatesSL3}
\begin{array}{lllll}
 x= t_{a}, 	& y= t_{b},   	&  z= t_{ab},  		& r= t_{ab^{-1}}, 	 & \tau=t_{[a,b]}, \\  
 u= t_{a^{-1}},	& v= t_{b^{-1}}, 	& w= t_{(ab)^{-1}},  	& s= t_{a^{-1}b}, &
\end{array} 
\end{equation}
Lawton has proved:

\begin{Theorem}[\cite{LawtonJA}]
\label{Thm:Lawton}
 $X(F_ 2,\mathrm{SL}(3,\mathbb{C}) )$ is isomorphic to the hypersurface of $\mathbb{C}^9$ defined by
$$
\tau^2-P \tau+Q=0
$$
for some polynomials $P,Q\in \mathbb{Z}[x,y,z,u,v,w,r,s]$.
The solutions of $\tau^2-P \tau+Q=0$ are
$t_{[a,b]}$ and $t_{[b,a]}$. Namely $t_{[a,b]}+t_{[b,a]}=P$ and $t_{[a,b]}t_{[b,a]}=Q$.
\end{Theorem}

The polynomials  $P$ and $Q$ in Theorem~\ref{Thm:Lawton} are:
%
$$
P=xuyv-uyr-xvs-uvz-xyw+rs+xu+yv+zw-3
$$
and
\begin{small}
 \begin{align*}
Q= \,
& 
v{u}^{2}{x}^{2}y+u{v}^{2}{y}^{2}x-ryx{u}^{2}-uxvrw-r{y}^{2}vu+r{x}^{2}
{v}^{2}-rzyxv+s{u}^{2}{y}^{2}-swvuy-svu{x}^{2}
\\
&
-uxysz-s{v}^{2}yx-{u}^{3
}vy+{u}^{2}{v}^{2}w-z{u}^{2}xv-u{v}^{3}x-z{v}^{2}yu-w{x}^{2}uy
-u{y}^{3
}x
-w{y}^{2}vx
\\
&
-{x}^{3}vy+{x}^{2}{y}^{2}z+uw{r}^{2}-2\,{r}^{2}xv+{r}^{2}
zy+surx+svry+wszr
+rz{u}^{2}+ru{v}^{2}+rv{z}^{2}
\\
&
+xr{w}^{2}+rw{y}^{2}+{x
}^{2}yr-2\,{s}^{2}uy+{s}^{2}wv+xz{s}^{2}+{u}^{2}vs
+us{z}^{2}+sz{v}^{2}
+sy{w}^{2}
+sw{x}^{2}
\\ 
&
+sx{y}^{2}+{u}^{2}wy-2\,{w}^{2}vu+xuyv+wuzx+u{y}^{
2}z+x{v}^{2}w+wvzy
+{x}^{2}zv-2\,{z}^{2}yx
\\ 
&
+{r}^{3}+3\,uyr-3\,rvw-3\,rzx
+{s}^{3}-3\,swu+3\,xvs-3\,syz
+{u}^{3}+3\,uvz+{v}^{3}+{w}^{3}
\\
&+3\,xyw+{x
}^{3}+{y}^{3}+{z}^{3}-6\,rs-6\,xu-6\,yv-6\,zw+9 
\end{align*}
\end{small}
%
%
%
This theorem tells that an $\mathrm{SL}(3,\mathbb{C})$-character of
$F_2$ is a polynomial on the
coordinates $x$, $y$, $z$, $u$, $v$, $w$, $r$, $s$, $\tau$. Its
expression can be computed by using basic identities on traces,
including Cayley-Hamilton's formula
\begin{equation}
 \label{eqn:CayleyHamilton}
A^3-\operatorname{tr}(A) A^2+\operatorname{tr}(A^{-1}) A -\operatorname{Id}=0,\qquad
\forall A\in\mathrm{SL}(3,\mathbb C),
 \end{equation}
and  elementary identities on traces.

\subsection{Two dimensional examples}

Let us compute the variety of $\mathrm{SL}(3,\mathbb C) $-characters of some 2-orbifolds.

We start with the turnover $S^2( 3,3,4)$, Figure~\ref{Figure:S2334}. Its fundamental group has presentation 
$$
\pi_1( S^2( 3,3,4) )\cong \langle a, b\mid a^3=b^3= (ab)^4=1\rangle-
$$
We look at the possible eigenvalues of the elements of finite order:
\begin{enumerate}[(a)]
 \item A matrix in  $\mathrm{SL}(3,\mathbb C) $ of order $3$ is either central or has eigenvalues $\{1,\omega,\omega^2\}$, for 
 $\omega\in\mathbb C$ a primitive third root of unity:
  $\omega^3=1$ and 
 $\omega\neq 1$.
 \item A matrix in  $\mathrm{SL}(3,\mathbb C) $ of order $4$ is either trivial or has eigenvalues 
 $$
 \{1,i, -i\}, \ \{-1,i, i\}, \ \{-1, -i, -i\}, \textrm{ or } \ \{1,-1, -1\} .
 $$
\end{enumerate}

\begin{figure}[ht]
\begin{center}
 \begin{tikzpicture}[line join = round, line cap = round, scale=1]
   \begin{scope}[shift={(-4,0)}]
 \draw (-1,0) to[out=5, in=180-5](1,0) to[out=120+5, in=-60-10] (0, 1.73) to[out=240+10, in=60-5] cycle; 
     \draw[red, fill=red] (-1,0) circle(0.04);
     \draw[red, fill=red] (1,0) circle(0.04);
     \draw[red, fill=red] (0,1.73) circle(0.04);
     \draw (-1+.4,0.03) to[out=80, in=0] (-1+.4*.6  , 0+.4*.87); 
      \draw[dashed, very thin] (-1+.4,0.03) to[out=160, in=-100] (-1+.4*.6  , 0+.4*.87); 
     \draw (1-.4,0.03) to[out=110, in=180+0] (1-.4*.6  , 0+.4*.87); 
       \draw[ dashed, very thin] (1-.4,0.03) to[out=20, in=-80] (1-.4*.6  , 0+.4*.87); 
    \draw  (0-.4*.37, 1.73-.4*.87) to[out=-60, in=180+60] (0+.4*.37, 1.73-.4*.87);
     \draw[ dashed, very thin]  (0-.4*.37, 1.73-.4*.87) to[out=30, in=180-30] (0+.4*.37, 1.73-.4*.87);
     \draw (-6/5,0) node{$3 $};
     \draw (6/5,0) node{$3 $};
     \draw (1/5,1.8) node{$4 $};
     \draw (0,-.75) node{$S^2(3,3,4)$};
   \end{scope}
      \begin{scope}[shift={(0,0)}]
  \draw (-1,0) to[out=5, in=180-5] (1,0) to[out=120+5, in=-60-10] (0, 1.73) to[out=240+10, in=60-5] cycle; 
 \draw[black, fill=gray!30!white, opacity=.3] (-1*.9,0+.1*.57)  to[out=5, in=180-5]  (1*.9,0+.1*.57) to[out=120+5, in=-60-10] (0, 1.73*.9+.1*.57)
  to[out=240+10, in=60-5] cycle; 
 \draw[black] (-1*.9,0+.1*.57)  to[out=5, in=180-5]  (1*.9,0+.1*.57) to[out=120+5, in=-60-10] (0, 1.73*.9+.1*.57)
  to[out=240+10, in=60-5] cycle; 
     \draw[red, fill=red] (-1*.95,0+.05*.57) circle(0.06);
     \draw[red, fill=red] (1*.95,0+.05*.57) circle(0.06);
     \draw[red, fill=red] (0 , 1.73*.95+.05*.57 ) circle(0.06);
     \draw (-6/5,0) node{$3 $};
     \draw (6/5,0) node{$3 $};
     \draw (1/5,1.8) node{$4 $};
     \draw (0,-.75) node{$T(3,3,4)$};
   \end{scope}
      \begin{scope}[shift={(5,.8)}]
        \draw (0,0) to[out=120, in=90 ] (-2.5,0);
        \draw (0,0) to[out=-120, in=-90 ] (-2.5,0);
        \draw[fill=gray!30!white, opacity=.3] (-.1,0) to[out=120, in=90 ] (-2.4,0);
        \draw[fill=gray!30!white, opacity=.3] (-.1,0) to[out=-120, in=-90 ] (-2.4,0);
        \draw (-.1,0) to[out=120, in=90 ] (-2.4,0);
        \draw (-.1,0) to[out=-120, in=-90 ] (-2.4,0);
         \draw[red, fill=red] (-1.5,0) circle(0.04);
         \draw[red, fill=red] (-.05,0) circle(0.06);
         \draw (-1.4,0.2) node{$3 $};
         \draw (0.13,0) node{$4 $};
     \draw (-1.25,-1.5) node{$D(3;4)$};
      \end{scope}
\end{tikzpicture}
\end{center}
\caption{The turnover $S^2(3,3,4)$ and two non-orientable quotients: a triangle with mirror boundary $T(3,3,4)$, which  is a Coxeter group,
and a disc with a cone point, mirror boundary, 
and a corner reflector, $D(3;4)$.}
\label{Figure:S2334}
\end{figure}
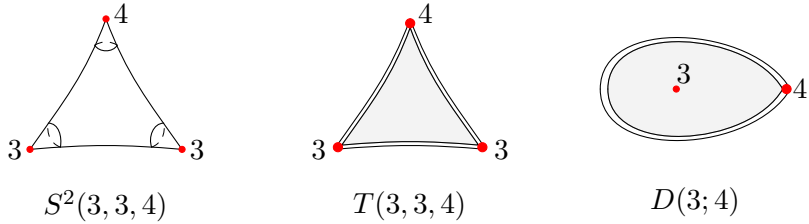

\begin{Remark}
Let $\rho\colon\pi_1( S^2( 3,3,4) )\to  \mathrm{SL}(3,\mathbb C) $ be an irreducible representation. The eigenvalues of $\rho(a)$
and $\rho(b)$ are  $\{1,\omega,\omega^2\}$ and the eigenvalues of $\rho(ab)$ are 
$\{1,i, -i\}$, $\{-1,i, i\}$, $\{-1, -i, -i\}$,  or $\{1,-1, -1\}$. Furthermore,
using that for an irreducible representation
$$
\dim_{[\rho]} X( S^2( 3,3,4) , \mathrm{SL}(3,\mathbb C) )=
- \widetilde \chi(S^2(3,3,4),\mathrm{Ad}\rho)
,
$$
and  $\dim \mathfrak{g}^{\rho(ab)} $ depends on whether $\rho(ab)$ has repeated eigenvalues or not, we have:
$$
\dim_{[\rho]} X( S^2( 3,3,4) , \mathrm{SL}(3,\mathbb C) )=
\begin{cases} 2 \textrm{ if the eigenvalues of }\rho(ab)\textrm{ are } \{1,i, -i\} ,\\
 0\textrm{ otherwise.}
\end{cases}
$$ 
\end{Remark}

Using the coordinates \eqref{eqn:coordinatesSL3}, 
the two dimensional components of the variety of characters are obtained by setting $x=y=u=v=0$ and $z=w=1$,
because this fixes the eigenvalues of $\rho(a)$, $\rho(b)$ and $\rho(ab)$. By replacing those values in 
Theorem~\ref{Thm:Lawton}, we obtain:

 \begin{Example}
\label{Example:S2334}
 The component of $X(S^2(3,3,4),  \mathrm{SL}(3,\mathbb{C}))$ that has positive dimension is
 isomorphic to
 $$
\{(r,s,\tau)\in\mathbb C^3 \mid \tau^2-( rs-2 ) \tau + ( {r}^{3}+{s}^{3}-5\,rs+5 ) =0 \} .
 $$
\end{Example}

For the components  that are isolated points, 
we set, again $x=y=u=v=0$ and, 
according to the eigenvalues of $\rho(ab)$:
$$
(z,w)= (-1+2\, i, -1-2\, i), \quad  (z,w)= (-1-2\, i, -1+2\, i), \quad\textrm{ or } \quad (z,w)= (-1, -1).
$$
Notice that  this does not fix the conjugacy class of $\rho(ab)$, because it has eigenvalues of multiplicity $2$. We need to 
find further restrictions on the traces: by 
taking traces on the relation $(ab)^{-2} a=(ab)^2a$, using \eqref{eqn:CayleyHamilton} and replacing $x=y=u=v=0$ we get
\begin{equation}
 \label{eqn:sr}
s= z s\qquad\textrm{ and }\quad r=wr .  
\end{equation}
For the 2-dimensional component it holds $z=w=1$, hence 
\eqref{eqn:sr} do not give any further information. For the isolated components, \eqref{eqn:sr} yield
$r=s=0$. For these values $P^2-4Q=0$ and therefore there is a unique value for $\tau$. Hence we get:

 \begin{Example}
\label{Example:S2334isolated}
There are three components of $X(S^2(3,3,4),  \mathrm{SL}(3,\mathbb{C}))$ that contain irreducible representations and are 
zero-dimensional. They have coordinates $x=y=u=v=r=s=0$ and
$$
(z,w, \tau)= (-1+2\, i, -1-2\, i, 1), \quad   (-1-2\, i, -1+2\, i, 1), \quad\textrm{ or } \quad  (-1, -1, -1).
$$
\end{Example}

Next we describe the Hitchin component $\operatorname{Hit}(S^2(3,3,4),  \mathrm{PSL}(3,\mathbb{R})) $.
Notice that the representation $\mathrm{Sym}^2\colon \mathrm{SL}(2,\mathbb R)\to \mathrm{SL}(3,\mathbb R) $ factors through
$\mathrm{PSL}(2,\mathbb R)$, hence $\operatorname{Hit}(S^2(3,3,4),  \mathrm{PSL}(3,\mathbb{R})) $ lifts to a component
of $X(S^2(3,3,4),  \mathrm{SL}(3,\mathbb{R}))$. 
We consider the 2-dimensional component of Example~\ref{Example:S2334} and we require that its coordinates  are real:
 $$
\{(r,s,\tau)\in\mathbb R^3 \mid \tau^2-( rs-2 ) \tau + ( {r}^{3}+{s}^{3}-5\,rs+5 ) =0 \} .
 $$
 By looking at the discriminant of this equation, this real set has
has three components. 
One of them is the isolated point $(r,s,\tau)=(2,2,1)$, which is in fact a singular point of the defining complex surface.
The other two components are homeomorphic to $\mathbb{R}^2$, one of them is the lift of the Hitchin component. By looking at the
symmetric power of the holonomy, we deduce:

 \begin{Example}
\label{Example:S2334Hitchin}
The Hitchin component  $\operatorname{Hit}(S^2(3,3,4),  \mathrm{PSL}(3,\mathbb{R})) $
is isomorphic to the component of 
 $$
\{(r,s,\tau)\in\mathbb R^3 \mid \tau^2-( rs-2 ) \tau + ( {r}^{3}+{s}^{3}-5\,rs+5 ) =0 \} 
 $$
 that contains
the point of coordinates $(r,s,\tau)=(2+2\sqrt{2},2+2\sqrt{2},5+4\sqrt{2})$.
\end{Example}

Next we consider the turnover $T(3,3,4)$, Figure~\ref{Figure:S2334}. Its fundamental group is an extension of $\pi_1(S^2(3,3,4))$
by an involution $\sigma$, that satisfies 
\begin{equation}
\label{eqn:relssigma}
\sigma^2=1, \qquad \sigma a\sigma= a^{-1},\qquad\textrm{ and }\qquad\sigma b\sigma= b^{-1}.
\end{equation}
Denote by $\sigma_*$ the involution induced on the fundamental group: 
$$
\qquad\qquad\sigma_*(\gamma)=\sigma\gamma\sigma^{-1},\qquad\qquad\forall\gamma\in\pi_1( S^2(3,3,4)).
$$
This involution $\sigma_*$
permutes the 
coordinates $r$ and $s$ and preserves $\tau$, the trace of the commutator. Therefore, the component of
$X(S^2(3,3,4),  \mathrm{SL}(3,\mathbb{C}))^{\sigma_*}$ 
of positive dimension is
$$
 \{(r,\tau)\in\mathbb C^2 \mid \tau^2-( r^2-2 ) \tau + ( 2{r}^{3}-5\,r^2+5 ) =0 \} .
$$
This curve is singular   precisely at $(r,\tau)=(2,1)$, and the singularity is an ordinary double point
(a self intersection with transverse tangents).

\begin{Lemma}
\label{Lemma:T2334}
 The restriction map
 $$
 X(T(3,3,4),  \mathrm{SL}(3,\mathbb{C})) \to X(S^2(3,3,4),  \mathrm{SL}(3,\mathbb{C}))^{\sigma_*}
 $$
 desingularizes the curve $\tau^2-( r^2-2 ) \tau + ( 2{r}^{3}-5\,r^2+5 ) =0 $. 
 
The component of $ X(T(3,3,4),  \mathrm{SL}(3,\mathbb{C}))$ of positive dimension  that contains irreducible representations is 
 this desingularization.
\end{Lemma}

\begin{proof}
We look at the fibre of the restriction map. We show that:
\begin{enumerate}[(a)]
 \item  the fibre of an irreducible character consists precisely of one point,
\item $(r,\tau)=(2,1)$ is the only reducible character in $X(S^2(3,3,4),  \mathrm{SL}(3,\mathbb{C}))^{\sigma_*}$, and 
\item the fibre of $(r,\tau)=(2,1)$ consists precisely of  $2$ points.
\end{enumerate}
Those three items prove the lemma, because the singularity is an ordinary double point.

We prove (a):
For an irreducible representation $\rho$   of $\pi_1(S^2(3,3,4))$ in $  \mathrm{SL}(3,\mathbb{C}) $ 
we show that there is a unique choice of $A\in \mathrm{SL}(3,\mathbb{C})$
such that mapping $\sigma$ to $A$ defines an extension of $\rho$ to $\pi_1(T(3,3,4))$.
By irreducibility of  $\rho$, as $\rho$ and $\rho\circ\sigma_*$ have the same character,
there exists a matrix 
$A\in  \mathrm{SL}(3,\mathbb{C})$ such that $A\rho A^{-1}=\rho\circ\sigma_*$.
Furthermore, $A$ is unique up multiplication by a matrix in the center 
(the center of  $\mathrm{SL}(3,\mathbb{C})$ 
is $\{\mathrm{Id},\omega \mathrm{Id},\omega^2\mathrm{Id} \}$). As $\sigma$ is an involution,  $A^2$ commutes with $\rho$, that is irreducible,
and therefore 
$A^2\in \{\mathrm{Id},\omega \mathrm{Id},\omega^2\mathrm{Id} \}$. Hence among $A$, $\omega A$ 
or $\omega^2 A$ there is a unique choice whose square is the identity. Thus there is a unique choice that satisfies 
\eqref{eqn:relssigma}.

To prove (b) we use that for a reducible representation the trace  of a commutator and its inverse are the same, 
thus it must be a zero of the discriminant of the
defining equation. This discriminant is 
\begin{equation}
 \label{eqn:disc}
(r-2)^2(r^2-4r-4) 
\end{equation}
and its zeros are $r=2$ and $r=2\pm 2\sqrt{2}$.  The value $r=2\pm 2\sqrt{2}$ corresponds to the symmetrization of the 
Fuchsian holonomy (and its Galois conjugate), hence it is irreducible. Thus $r=2$ is the only reducible character.

To prove (c), we check first that $r=2$ corresponds to a representation $\rho=\rho_2\oplus \mathrm{Id}$,
for $\rho_2\colon\pi_1( S^2(3,3,4 ))\to  \mathrm{SL}(2,\mathbb{C}) $
an irreducible representation
(for a reducible representation in $\mathrm{SL}(2,\mathbb{C})$
the trace of any commutator is $2$ and, as $\tau=1$, 
$\operatorname{trace}(\rho_2([a,b]))=\tau-1\neq 2$). 
Then, reproducing the argument of the irreducible case (a), there exist a matrix
$A\in  \mathrm{SL}(2,\mathbb{C})$ such that $A\rho_2 A^{-1}=\rho_2\circ\sigma_*$ 
and, by irreducibility of $\rho_2$, $A$ is unique up to sign. Furthermore, again by irreducibility of $\rho_2$,
$A^2$ is central in $\mathrm{SL}(2,\mathbb{C}) $, namely 
$A^2=\pm \operatorname{Id}$. 
The case $A^2=+ \operatorname{Id}$ does not occur, because this 
would imply that $A=\pm \operatorname{Id}$ and with 
\eqref{eqn:relssigma} this would yield that $\rho_2$ itself would be central, but it is irreducible. Therefore
%
%
$A^2=- \operatorname{Id}$ and the choices 
for $\rho(\sigma)$  are
$$
\rho(\sigma)=\begin{pmatrix}
              \pm  i\, A & \begin{matrix}
                       0\\ 0
                      \end{matrix}
 \\
              \begin{matrix}
                       0 & 0
                      \end{matrix}
               & -1
             \end{pmatrix}. 
$$
This concludes the proof of (c) and of the lemma.
\end{proof}

Next we describe the Hitchin component of the turnover. We look therefore at the real points
$$
 \{(r,\tau)\in\mathbb R^2 \mid \tau^2-( r^2-2 ) \tau + ( 2{r}^{3}-5\,r^2+5 ) =0 \} .
$$
By looking at the discriminant \eqref{eqn:disc}, the set of real points has three components, 
the isolated point $(r,\tau)=(2,1)$ an two lines, defined by $r\leq  2-2\sqrt{2}$ and 
$r\geq  2+2\sqrt{2}$. As $r=2+2\sqrt{2}$ is the symmetric power of the holonomy:

\begin{Example}
 \label{Example:T2334Hitchin}
The Hitchin component  $\operatorname{Hit}(T(3,3,4),  \mathrm{PSL}(3,\mathbb{R})) $
is isomorphic (via the restriction to $S^2(3,4,4)$) to the line
$$
 \{(r,\tau)\in\mathbb R^2 \mid \tau^2-( r^2-2 ) \tau + ( 2{r}^{3}-5\,r^2+5 ) =0,\ r\geq 2+2\sqrt{2}  \} .
$$
\end{Example}

Finally, we consider $D(3;4)$, the 
 disc with a cone point of order $3$, mirror boundary, 
and a corner reflector of order $4$ (with isotropy group the dihedral group of 8 elements), Figure~\ref{Figure:S2334}.
It is again the quotient of $S^2(3,3,4)$ by an involution. This involution maps $a$ to $b^{-1}$ and $b$
to $a^{-1}$. 
Therefore it fixes the coordinates $r$ and $s$ but permutes the trace of $[a,b]$ with the trace
of its inverse $[b,a]$. 
The  fixed point set of this involution in $X(S^2(3,3,4),  \mathrm{SL}(3,\mathbb{C}))$ 
is obtained by looking 
at the zero locus of the discriminant of the variable $\tau$:
$$
\{(r,s)\in\mathbb C^2 \mid {r}^{2}{s}^{2}-4\,{r}^{3}-4\,{s}^{3}+16\,rs-16  =0 \} .
$$
With the very same discussion as in Lemma~\ref{Lemma:T2334}:

\begin{Lemma}
 The restriction map
 $$
 X(D(3;4),  \mathrm{SL}(3,\mathbb{C})) \to X(S^2(3,3,4),  \mathrm{SL}(3,\mathbb{C}))^{\sigma'}
 $$
 desingularizes the curve $ {r}^{2}{s}^{2}-4\,{r}^{3}-4\,{s}^{3}+16\,rs-16   =0 $. 
 
The component of $ X( D(3;4) ,  \mathrm{SL}(3,\mathbb{C}))$ of positive dimension  that contains irreducible characters is 
 this desingularization.
\end{Lemma}

The set of real points of $ {r}^{2}{s}^{2}-4\,{r}^{3}-4\,{s}^{3}+16\,rs-16   =0 $ is the discriminant locus of
the set of real points for $S^2(3,3,4)$:  
it has three components, the isolated point $r=s=2$ and two unbounded curves. Thus, similarly to $S^2(3,3,4)$  we have:

\begin{Example}
 \label{Example:D234Hitchin}
The Hitchin component  $\operatorname{Hit}(D(3;4),  \mathrm{PSL}(3,\mathbb{R})) $
is isomorphic (via the restriction to $S^2(3,4,4)$) to the component of 
$$
 \{(r,s)\in\mathbb R^2 \mid    {r}^{2}{s}^{2}-4\,{r}^{3}-4\,{s}^{3}+16\,rs-16   =0   \} 
$$
that contains $r=s=2+2\sqrt{2} $.
\end{Example}

\subsection{A family of 3-dimensional orbifolds}
 
Consider the family of examples with underlying space $S^3$ (perhaps punctured twice) and three 
branching arcs as in Figure~\ref{Figure:Hopf},
with branching labels $n_0$, $n_1$ and $n_2\geq 2$. Whether a singular vertex
is included or not depends on the orders of the adjacent branching locus: if $1/n_0+2/n_1> 1$, then the three edges meet at a vertex,
otherwise the vertex is removed, and similarly for the other vertex. In other words, 
we consider the twice punctured 3-sphere, and we add the vertices needed for the orbifold to be irreducible.

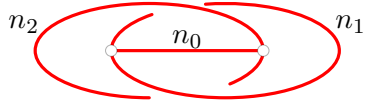
\begin{figure}[h]
\begin{center}

\begin{tikzpicture}[line join = round, line cap = round, scale=.5]
\draw[red, very thick] (-2,0)--(2,0);
\draw(2,0)[red, very thick] arc[x radius = 3cm, y radius = 1.25cm, start angle=0, end angle= 270];
\draw(2,0)[red, very thick] arc[x radius = 3cm, y radius = 1.25cm, start angle=0, end angle= -45];

\draw(-2,0)[red, very thick] arc[x radius = 3cm, y radius = 1.25cm, start angle=180, end angle= 360+100];
\draw(-2,0)[red, very thick] arc[x radius = 3cm, y radius = 1.25cm, start angle=180, end angle= 130];

\draw (0,0.3) node{$n_0$};
\draw (4.3,0.7) node{$n_1$};
\draw (-4.3,0.7) node{$n_2$};

\fill[white]  (2,0) circle (.15);
\draw[thin, black, opacity=.3] (2,0) circle (.15);
\fill[white]  (-2,0) circle (.15);
\draw[thin, black, opacity=.3] (-2,0) circle (.15);
\end{tikzpicture}

\end{center}
\caption{The underlying space of the orbifold
is a 3-sphere, possibly punctured once or  twice, according to the branching indices,  the branching locus consists of three arcs as in the picture.}
\label{Figure:Hopf}
\end{figure}

Let $\mathcal O^3$ denote this orbifold.
The geometry of $\mathcal O^3$ depends on the indices:
\begin{enumerate}[(a)]
 \item When $n_2=2$ and $1/n_0+2/n_1\leq 1$, it is an  interval bundle, over a disc with mirror boundary,
 with an interior cone 
 point of order $n_1$ and a corner
 reflector of order $2n_0$. It is doubly covered by the product of a turnover
 $S^2(n_0,n_1,n_1)$
 by an interval.
 \item When $n_2=2$ and $1/n_0+2/n_1> 1$, it is obtained by coning the boundary of the previous example, and it is spherical. 
 It is doubly covered by the suspension over a turnover $S^2(n_0,n_1,n_1)$.
 \item When $(n_0, n_1, n_2)=(2,3,3)$ the orbifold is spherical \cite{Dunbar}.
 \item If $n_1,n_2\neq 2$ and $(n_0,n_1,n_2)\neq (2,3,3)$, then it is hyperbolic. 
\end{enumerate}

The hyperbolicity is realized by a tetrahedron as in Figure~\ref{Figure:HopfFD}. The vertices are possibly finite, ideal or exterior
(hence truncated by geodesic triangles perpendicular to the other sides). Poincar\' e fundamental domain theorem yields the following presentation
for the fundamental group
$$
\pi_1(\mathcal O^3)\cong\langle a, b\mid a^{n_1}=b^{n_2}= [a,b]^{n_0}=1\rangle.
$$

\begin{figure}[h]
\begin{center}

\begin{tikzpicture}[line join = round, line cap = round, scale=1.4]

 \draw[thick] (0,0) -- (2.7,-.2) -- (1.2,2.5) -- cycle;
 \draw[thick,->](0,0)--(2.7*2/3,-.4/3);
 \draw[thick,->](1.2,2.5)--(2.2, 0.7); 
 \draw[thick,->](1.2,2.5)--(7/3, 1.7); 
 \draw[thick, ->] (2.9*19/30,1.3*19/30)--(2.9*2/3,1.3*2/3);
 \draw[thick]  (2.7,-.2)-- (2.9,1.3) -- (1.2,2.5);
 \draw[dashed, opacity=.7] (2.9,1.3)--(0,0);
 \draw[blue, thick, fill=white] (0,0) circle(.06);
 \draw[blue, thick, fill=white] (1.2,2.5) circle(.06);
 \draw[red, thick, fill=white] (2.9,1.3) circle(.06);
 \draw[red, thick, fill=white] (2.7,-.2) circle(.06);
 \draw (2.1,2.2) node{$\frac{\pi}{2n_0} $};
 \draw (3.2,0.8) node{$\frac{2\pi}{n_1} $};
 \draw (0.5,1.5) node{$\frac{2\pi}{n_2} $};
 \draw (1.5,-.4) node{$\frac{\pi}{2n_0} $};
 \draw (2.1,1.3) node{$\frac{\pi}{2n_0} $};
  \draw (1.05,0.7) node{$\frac{\pi}{2n_0} $};
  \draw[  <-] (0.3,0.8) arc[x radius = .24, y radius = .15, start angle =50, end angle= 260 ];
  \draw (-.4,.8) node{$A$};
  \draw[  <-] (2.8,-0.1) arc[x radius = .21, y radius = .24, start angle =225, end angle= 360+90];
  \draw (3.4,0) node{$B$};

\end{tikzpicture}

\end{center}
\caption{The fundamental domain of the  orbifold 
in this subsection. 
It is an ideal tetrahedron in hyperbolic 3-space, with vertices perhaps ideal or hyperideal. The side parings, $a$ and $b$,  are rotations 
of order $n_1$ and $n_ 2$ around two opposite edges. Those edges have dihedral angle $2\pi/n_1$ and 
$2\pi/n_2$ respectively, the remaining four edges have dihedral angle $\pi/(2 n_0)$. }
\label{Figure:HopfFD}
\end{figure}
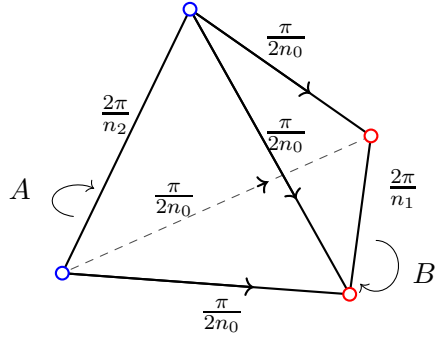

In case (a) the variety of characters 
of $\mathcal O^3$ 
in $\mathrm{SL}(3,\mathbb C)$
is the variety of characters of a non-orientable 2-orbifold. It has dimension $1$ or $0$ according
to the values of $n_0$ and $n_1$.
Cases (b) and (c) are
spherical, hence the variety of characters is finite. In the hyperbolic case, (d), when $n_0=2$ then the canonical component is an isolated point, by 
Theorem~\ref{Theorem:dimcanonical}. For $n_i\geq 3$, it 
has dimension 2.

\begin{Example}
\label{ex:hopf}
Consider the case $n_0=n_1=n_2=3$. The fundamental group is
$$
\pi_1(\mathcal O^3)=\langle a, b\mid a^3=b^3= [a,b]^3=1\rangle.
$$
On the canonical component, the image of any of  $a$, $b$ and $[a,b]$ is not central, therefore
by
Theorem~\ref{Thm:Lawton}, using the notation in  \eqref{eqn:coordinatesSL3} we impose the 
following equalities:
$$
x=y=u=v=P=Q=0.
$$
Thus  the canonical component $X_0(\mathcal O^3,  \mathrm{SL}(3,\mathbb C))$ is the following complex surface:
$$
\left\{
\begin{array}{l}
 r s+z w- 3=0\\
 r^3+s^3+z^3+w^3+rszw-6rs-6zw+9=0
\end{array}
\right.
$$
The symmetric power of the complete structure has coordinates
$$
r=s= \frac{-3+\sqrt{-3}}{2}, \qquad z=w= \frac{-3-\sqrt{-3}}{2}
$$
(The complex conjugate corresponds to a change of orientation).

There are other components of $X(\mathcal O^3,  \mathrm{SL}(3,\mathbb C))$ but it may be checked that they are isolated points.
For instance, if $\rho(a)$ is central, then $\rho(b)$ has order three at this gives finitely many conjugacy classes for $\rho$ 
(whether $\rho(b)$ is central or not). 
If $\rho(a)$ and $\rho(b)$ are noncentral but  $\rho([a,b])$ is central, then it can be computed that
$x=y=z=r=u=v=w=s=0$. This is realized by the representation
$$
\rho(a)=\begin{pmatrix}
 1 & 0 & 0\\ 0 & \omega & 0 \\ 0 & 0& \omega^2
\end{pmatrix}
\qquad\textrm{ and }\qquad
\rho(b)=\begin{pmatrix}
 0 & 0 & 1\\ 1 & 0 & 0 \\ 0 & 1& 0
\end{pmatrix}
$$
\end{Example}

When $n_0=n_1=n_2=3$, 
a component of the variety of characters in $\mathrm{SL}(4,\mathbb{R})$ corresponding to projective structures
has been computed in \cite{PTillmann}.

\bibliographystyle{plain}
\bibliography{reps}

\noindent \textsc{Departament de Matem\`atiques, Universitat Aut\`onoma de Barcelona,
and Centre de Recerca Matem\`atica (UAB-CRM),
08193 Cerdanyola del Vall\`es, Spain }

\noindent \textsf{porti@mat.uab.cat}

\end{document}